\documentclass[11pt]{article}

\usepackage{graphicx,amsmath,amssymb,color}
\usepackage{latexsym,subfigure,stmaryrd}

\usepackage{tikz}

\usetikzlibrary{matrix}

\newtheorem{theorem}{Theorem}[section]
\newtheorem{lemma}[theorem]{Lemma}
\newtheorem{proposition}[theorem]{Proposition}
\newtheorem{corollary}[theorem]{Corollary}

\newcommand{\R}{\mathbb{R}}
\newcommand{\D}{{\Rd \mathrm{D}}}
\newcommand{\N}{\mathbb{N}}

\newcommand{\V}{\boldsymbol{V}}
\newcommand{\n}{\boldsymbol{n}}
\newcommand{\sH}{{\rm H}}

\newcommand{\Om}{{\Omega}}
\newcommand{\sL}{{\rm L}}

\def\nablat{\nabla_{\tau}}

\def\Deltat{\Delta_{\tau}}

\def\restriction#1#2{\mathchoice
              {\setbox1\hbox{${\displaystyle #1}_{\scriptstyle #2}$}
              \restrictionaux{#1}{#2}}
              {\setbox1\hbox{${\textstyle #1}_{\scriptstyle #2}$}
              \restrictionaux{#1}{#2}}
              {\setbox1\hbox{${\scriptstyle #1}_{\scriptscriptstyle #2}$}
              \restrictionaux{#1}{#2}}
              {\setbox1\hbox{${\scriptscriptstyle #1}_{\scriptscriptstyle #2}$}
              \restrictionaux{#1}{#2}}}
\def\restrictionaux#1#2{{#1\,\smash{\vrule height .8\ht1 depth .85\dp1}}_{\,#2}} 
\newcommand{\dive}{\mathrm{div}\,}
\newcommand{\divet}{\mathrm{div_{\tau}}\,}

\def\nablat{\nabla_{\tau}}

\def\Deltat{\Delta_{\tau}}

\newcommand{\dd}{\textup{d}}
\newcommand{\Rd}{\color{black}}
\newcommand{\Bk}{\color{black}}
\newcommand{\Be}{\color{black}}

\newcommand{\Tr}{\mathrm{Tr}}

\def\blacksquare{
\thinspace\nobreak \vrule width 5pt height 5pt depth 0pt}
\newtheorem{remark}[theorem]{Remark}
\newenvironment{proof}{\begin{trivlist}
                       \item[]\hspace{0cm}{\bf Proof:}
                       \hspace{0mm} }{\hfill $\blacksquare$
                     \end{trivlist}}
\newenvironment{proofof}[1]{\begin{trivlist}
                       \item[]\hspace{0cm}{\bf Proof of #1: }
                       \\ }{\hfill $\blacksquare$
                     \end{trivlist}}
                     
\begin{document}

\title {An extremal eigenvalue problem for the {\Be Wentzell}-Laplace operator }
\author{M. Dambrine, D. Kateb, J. Lamboley}
\maketitle

\noindent\textbf{Abstract:} We consider the question of giving an upper bound for the first nontrivial eigenvalue of the {\Be Wentzell}-Laplace operator of a domain $\Om$, involving only geometrical informations. We provide such an upper bound, by generalizing Brock's inequality concerning Steklov eigenvalues, and we conjecture {\Rd that balls maximize the Wentzell eigenvalue, in a suitable class of domains, which would improve our bound}. To support this conjecture, we prove that balls are critical domains for the {\Be Wentzell} eigenvalue, in any dimension, and that they are local maximizers in dimension 2 and 3, using an order two sensitivity analysis. We also provide some numerical evidence.

\tableofcontents

\section{Introduction}

\paragraph{Background.}
Let $d\ge 2$ and $\Omega$ be a bounded domain in $\mathbb{R}^d$ (i.e a bounded connected open set) supposed to be sufficiently smooth  (of class $C^3$),
and we denote $\Delta_\tau$ the Laplace-Beltrami operator on $\partial \Omega$.
Motivated by generalized impedance boundary conditions, we consider the eigenvalue problem for {\Be Wentzell} boundary  conditions 
\begin{equation}\label{Steklov_Ventcel}
\left\{
\begin{array}{rlll}
-\Delta u &=&0&~\textrm{~~in~~}  \Omega\\
-\beta \Delta_\tau u+\partial_n{}u&=&\lambda u&~\textrm{~~on~}~\partial \Omega
\end{array}
\right.
\end{equation}
where $\beta$ is a given real number and  $\partial_n$ denotes the outward unit normal derivative.

The coefficient $\beta$ appears as a surface diffusion coefficient arising in a passage to the limit in the thickness of the boundary layer for coated object (see \cite{LemrabetTeniou,BendaliLemrabet,HadJol05}).  A general derivation of {\Be Wentzell} boundary conditions can be found in \cite{Goldstein}. \Bk The coef\-ficient can be either positive or negative. We first consider the case $\beta\geq 0$ where the obtained  boundary value problem is coercive.

This problem couples surface and volume effects through the Steklov eigenvalue problem in $\Omega$ with the Laplace-Beltrami eigenvalue problem on $\partial\Omega$. 
Let us recall some known facts about these two problems. The {\bf Steklov eigenvalue problem} consists in solving 
\begin{equation}\label{Steklov}
\left\{
\begin{array}{rlll}
\Delta u &=&0&~\textrm{~~in~~}  \Omega\\
\partial_n{}u&=&\lambda^S  u&~\textrm{~~on~}~\partial \Omega
\end{array}
\right.
\end{equation}
It has a discrete spectrum consisting in a sequence
$$
\lambda_0^S(\Omega)=0< \lambda_1^S(\Omega) \le \lambda_2^S(\Omega) \ldots \rightarrow +\infty 
$$
where the $\lambda^S$  {are} called Steklov eigenvalues. Brock-Weinstock inequality states that ${\lambda^S_{1}}$ is maximized by the ball among all open sets of fixed volume $|\Omega|$. It was first proved in the case $d=2$ by Weinstock and extended by Brock to any dimension in \cite{Brock} (Weinstock inequality is slightly stronger but restricted to simply-connected domains: he proved indeed that the disk maximizes $\lambda_{1}^S$ among simply-connected sets of given perimeter). A quantitative form of this inequality was recently obtained by Brasco, De Philippis and Ruffini who proved in \cite{BrascoDePhilippisRuffini} that
$$
\lambda_{1}^S(\Omega) \leq \lambda_{1}^S(B)  \left[ 1 - \delta_{d} \left(\cfrac{|\Omega\Delta B(x_{\partial\Omega})|}{|\Omega|} \right)^2\right],
$$
where $\delta_{d}$ is an explicit nonnegative constant depending only on $d$, $x_{\partial\Omega}$ is the center of mass of $\partial\Omega$ and $B(x_{\partial\Omega})$ is the ball centered in $x_{\partial\Omega}$ with volume $|\Omega|$\footnote{\Rd The results in \cite{BrascoDePhilippisRuffini} are stated with the Fraenkel asymmetry, meaning that the previous inequality is stated for the ball $B$ of volume $|\Om|$ that minimizes $|\Om\Delta B|$, but from the proof (see \cite[Section 5]{BrascoDePhilippisRuffini}) we can conclude that the ball $B(x_{\partial\Om})$ of volume $|\Om|$ and such that $\int_{\partial\Om}(x-x_{\partial\Om})d\sigma=0$ is in fact valid as well.}.
Let us emphasize that no additional topological assumption is needed.

It is well-known that  the {\bf spectrum of the Laplace-Beltrami operator} on $\partial\Omega$, {\Rd that is numbers $\lambda$ such that the equation $-\Delta_{\tau}u=\lambda u$ on $\partial\Om$ has nontrivial solutions}, is also discrete and satisfies:
$$
\lambda_0^{LB}(\partial\Omega)=0< \lambda_1^{LB}(\partial\Omega) \le \lambda_2^{LB}(\partial\Omega) \ldots \rightarrow +\infty 
$$
Again, one can ask if $\lambda_{1}^{LB}$ takes its maximal value on the euclidean {sphere}, among  hypersurfaces of fixed {$(d-1)$-dimensional} volume. Here, the answer is more complicated than for the Steklov problem. It depends on both the topology of the surface and the dimension. In \cite{Hersch2}, Hersch gave a positive answer  if $d= 3$  for surfaces homomorphic to the euclidean {sphere}. In the cases $d>3$ or without topological restriction, the answer is negative (see \cite{Bleecker,ColboisDodziuk,ColboisDrydenElSsoufi}, and Section \ref{remarques:preliminaires} for the 2-dimensional case)).

When $\beta\geq 0$, the {\bf spectrum of the Laplacian with  {\Be Wentzell} conditions} consists in an increasing countable sequence of  eigenvalues 
\begin{equation}\label{eq:spectrum}
\lambda_{0,\beta}(\Omega)=0< \lambda_{1,\beta}(\Omega) \le \lambda_{2,\beta}(\Omega) \ldots \rightarrow +\infty 
\end{equation}
with corresponding real orthonormal {\Rd (in $L^2(\partial\Om)$)} eigenfunctions $u_{0},u_1,u_2,\ldots$  As in the previous cases, the first eigenvalue is zero with constants as corresponding  eigenfunctions. As usual, we adopt the convention that each eigenvalue is repeated according to its multiplicity.  Hence, the first eigenvalue of interest is $\lambda_{1,\beta}$. {\Rd A variational characterization of the eigenvalues is available:} we introduce the Hilbert space 
$$ \sH(\Omega)= \{ u \in \sH^1(\Omega),\  \Tr_{\partial\Omega}(u) \in \sH^1(\partial\Omega) \},
$$
where {\Rd $\Tr_{\partial\Om}$ is the trace operator}, and we define on $\sH(\Omega)$ the two bilinear forms
\begin{equation}
\label{definition:Abeta}
A_{\beta}(u,v)= \int_{\Omega} \nabla u . \nabla v~dx+\beta \int_{\partial \Omega} \nablat u . \nablat v ~d\sigma ,\;\;\;\;\;\;\;\;{B(u,v)= \int_{\partial \Omega}  u  v ,}
\end{equation}
where $\nablat$ is the tangential gradient.
Since we assume $\beta$  is nonnegative, the two bilinear forms are positive and  the variational characterization for the $k$-th eigenvalue is 
\begin{equation}\label{eq:var}
\lambda_{k,\beta}(\Omega)=\min~\left\{\displaystyle\frac{A_{\beta}(v,v) }{B(v,v)},~v \in \sH(\Omega),~{\int_{\partial\Omega}} v u_i =0,~i=0,\ldots,k-1\right\}
\end{equation}
In particular, when $k=1$, the minimum is taken over the functions orthogonal to the eigenfunctions associated to $\lambda_{0,\beta}=0$, i.e constant functions. To describe this spectrum, one can notice that the eigenvalue problem can be rewritten purely on $\partial\Omega$ as:
$$-\beta \Delta_\tau u+\D u=\lambda u$$
where $\D$ denotes the Dirichlet-to-Neumann map, that is a selfadjoint, positive pseudodifferential operator of order one. Therefore, this problem can be seen as a compact perturbation of the usual Laplace-Beltrami operator. This point of view was used in \cite{BDHV10} and justify that high order eigenvalues of the Laplace-{\Be Wentzell} problem look like those of the Laplace-Beltrami operator. 

However, we are interested in this work, in studying low order eigenvalues and more precisely in giving an upper bound for the second eigenvalue $\lambda_{1,\beta}$ involving only geometrical informations. 
{\Rd Please remark that we are not seeking for lower bound, because even with very strong geometrical assumption, there is none. Indeed, a consequence of our results is that
\begin{equation}\label{eq:inf}
\inf\left\{\lambda_{1,\beta}(\Om),\;\;\Om\textrm{ convex }, \;|\Om|=m\right\}=0
\end{equation}
for any value of $\beta\geq 0$ and $m\geq 0$, see Remark \ref{rk:cigare}.}
An important remark at this point is that the bilinear form $A_{\beta}$ is not homogeneous with respect to dilatation of the domain. Therefore, the volume of $\Omega$ plays a crucial role in $\lambda_{1,\beta}$. As a surface term appears also in $A_{\beta}$ (corresponding to the Laplace-Beltrami operator), the perimeter of $\Omega$ (i.e. the volume of $\partial\Omega$) should also play a crucial role. 

Notice that when $\beta=0$ we retrieve the Steklov eigenvalues, and we recover the Laplace-Beltrami eigenvalues by considering $\frac{1}{\beta}\lambda_{1,\beta}$ and letting $\beta$ go to $+\infty${\Rd , see Section \ref{remarques:preliminaires}}.

Note also that the close but distinct eigenvalue problem 
\begin{equation}\label{Steklov_Ventcel:kennedy}
\left\{
\begin{array}{rlll}
-\Delta u &=& \lambda u&~\textrm{~~in~~}  \Omega\\
\Delta u+\alpha \partial_n u +\gamma u&=&0&~\textrm{~~on~}~\partial \Omega
\end{array}
\right.
\end{equation}
was considered by J.B. Kennedy in \cite{Kennedy}. He transforms this problem into a Robin type problem to prove a Faber-Krahn type inequality when the constants $\alpha,\gamma$ are non negative: the ball is the best possible domain among those of given volume. 

\paragraph{The results of the paper.} We first apply the strategy of F. Brock for the Steklov eigenvalue problem to the {\Be Wentzell} eigenvalue problem and 
obtain a first upper bound of $\lambda_{1,\beta}(\Omega)$ in terms of purely geometric quantities (we actually provide a refined version, using \cite{BrascoDePhilippisRuffini}):

\begin{theorem}\label{main:theorem:section:brock}
{\Be Let $\Om$ a smooth set such that $\int_{\partial\Om}x=0$.} Let $\Lambda[\Omega]$ be the spectral radius of the symmetric and positive semi\-definite matrix $P(\Omega)=(p_{ij})_{i,j=1,\ldots,d}$ defined as 
\begin{equation}
\label{definition:matrice:A}
p_{ij}=\displaystyle \int_{\partial\Omega} (\delta_{ij}-n_in_j),
\end{equation}
where $\n$ is the outward normal vector to $\partial\Omega$. Then if $\beta\geq 0$, one has:
\begin{equation}\label{egalite}
S(\Omega):= {\sum_{i=1}^{d}} \cfrac{1}{\lambda_{i,\beta}(\Omega)} {\Be\;\; \geq \cfrac{\int_{\partial\Om}|x|^2}{|\Om|+\beta\Lambda[\Om]}}\;\;\geq  \cfrac{d\omega_{d}^{-1/d}|\Omega|^{\frac{d+1}{d}}}{|\Omega|+\beta\Lambda[\Omega]} {\Rd \left[ 1+\gamma_{d} \left(\cfrac{|\Omega\Delta B|}{|B|}\right)^2\right]}.
\end{equation}
where 
\begin{equation}
\label{constante:Brasco}
\gamma_{d}=\cfrac{d+1}{d} \  \cfrac{2^{1/d}-1}{4},
\end{equation}
$\omega_{d}=|B_{1}|$ and $B$ is the ball of volume $|\Omega|$ and with the same center of mass than $\partial\Omega$.
Equality holds in \eqref{egalite} {if} $\Omega$ is a ball.
\end{theorem}

A trivial consequence of Theorem \ref{main:theorem:section:brock} is the following upper bound for $\lambda_{1,\beta}(\Omega)$. 

\begin{corollary}
\label{cor:bornesup}
{\Be With the same notations as in Theorem \ref{main:theorem:section:brock},} if $\beta\geq 0$, it holds:
\begin{equation}\label{majoration:lambda2}
\lambda_{1,\beta}(\Omega) {\Be \leq d\frac{|\Om|+\beta\Lambda[\Om]}{\int_{\partial\Om}|x|^2}}\leq \cfrac{|\Omega|+\beta\Lambda[\Omega]}{\omega_{d}^{-1/d}|\Omega|^{\frac{d+1}{d}}{\Rd \left[ 1+\gamma_{d} \left(\cfrac{|\Omega\Delta B|}{|B|}\right)^2\right]}}. 
\end{equation}
where $B$ and $\gamma_{d}$ are as in Theorem \ref{main:theorem:section:brock}.
 Equality holds in \eqref{majoration:lambda2} {if} $\Omega$ is a ball.
\end{corollary}

Note that the method used for the {\Be Wentzell} eigenvalue problem also applies for the Laplace-Beltrami case and provides an upper bound for $\lambda_{1}^{LB}$ without any topological assumption on $\Omega$. 
\begin{theorem}\label{main:theorem:section:brockLB}
{\Be With the same notations as in Theorem \ref{main:theorem:section:brock}}, it holds
\begin{equation}\label{egaliteLB}
S^{LB}(\partial\Omega):=\sum_{i=1}^{d} \cfrac{1}{\lambda^{LB}_i(\partial\Omega)} {\Be \geq \cfrac{\displaystyle\int_{\partial\Om}|x|^2}{\Lambda[\Om]}}\geq \cfrac{d\omega_{d}^{-1/d}|\Omega|^{\frac{d+1}{d}}}{\Lambda[\Omega] }{\Rd \left[ 1+\gamma_{d} \left(\cfrac{|\Omega\Delta B|}{|B|}\right)^2\right]}.
\end{equation}
and 
\begin{equation}\label{majoration:lambda2LB}
\lambda_{1}^{LB}(\partial\Omega){\Be \leq d\frac{\Lambda[\Om]}{\displaystyle\int_{\partial\Om}|x|^2}}\leq   \cfrac{\Lambda[\Omega]}{\omega_{d}^{-1/d}|\Omega|^{\frac{d+1}{d}} {\Rd \left[ 1+\gamma_{d} \left(\cfrac{|\Omega\Delta B|}{|B|}\right)^2\right]}}.
\end{equation}
 Equality holds in \eqref{egaliteLB} and \eqref{majoration:lambda2LB} {if} $\Omega$ is a ball.
\end{theorem}

It is expected in this type of extremal eigenvalue problem that ball {\Rd are maximizers}. We are not able to fully justify the natural following conjecture:
\\
 
\noindent{\bf Conjecture:}  The ball maximizes the first non-trivial {\Be Wentzell}-Laplace eigenvalue among smooth open sets of given volume and which are homeomorphic to the ball. \\
  
The topological restriction is motivated by the limit case $\beta 	\rightarrow + \infty$ as we noticed before (see also Section \ref{remarques:preliminaires}).
{\Be In Section \ref{preuve:brock}, we observe that the intermediate bound in \eqref{majoration:lambda2} has both its numerator and denominator that are minimized by the ball, under volume constraint, so there is a competition. In Section \ref{section:tests:numeriques} we observe that in fact, the ball does not minimize this bound in general (see Figure \ref{sharpness:upperbound}). Therefore, we can not deduce from this bound the maximality of balls (though it might work for certain values of $\beta$ and the volume constraint). 
About the upper bound \eqref{majoration:lambda2}, we show that it is larger than $\lambda_{1,\beta}(B)$ for every $\beta>0$ (with equality  for the ball) and hence again does not implies {\Rd that balls are maximizing $\lambda_{1,\beta}$}}. 
To check if balls are {\Rd relevant candidates for maximizers} in our case, we then turn our attention to a shape sensitivity analysis of $\lambda_{1,\beta}$. 

{Therefore, we first wonder if} the ball {is} a critical shape in any dimension. With respect to shape sensitivity, the main difficulty is to handle multiple eigenvalues which leads to a nonsmooth dependency of $\lambda_{1,\beta}$ with respect to $\Omega$. However, for a fixed deformation field $\V\in W^{3,\infty}(\Om,\R^d)$, along the transport of $\Omega$ by $T_t=I+t\V$, we prove the existence of smooth branches of eigenvalues and eigenfunctions associated to the subspace generated by the group of eigenvalues and provide a characterization of the derivative along the branches: $\lambda_{1,\beta}$ is then the minimum value among these $d$ smooth branches. 

\begin{theorem}\label{Theoreme:gradient:cas}
We distinguish the case of simple and multiple eigenvalue. 
\begin{itemize}
\item If $\lambda{=\lambda_{k,\beta}(\Omega)}$ is a simple eigenvalue of the {\Be Wentzell} problem, then the  application $t \mapsto \lambda(t) {=\lambda_{k,\beta}(\Om_{t})}$ {(where $\Om_{t}=(I+t\V)(\Omega)$)} is differentiable and the derivative at $t=0$ is 
$$
\lambda'(0)=\int_{\partial\Omega}V_n\Big( \vert  \nablat u\vert ^2-\vert \partial_{n}u \vert ^2 -{ \lambda} H \vert  u_0\vert ^2+\beta  (H~I_{d}-2D^2b)\nablat u.\nablat u
\Big)~d\sigma.
$$
where $u$ is the normalized eigenfunction associated to $\lambda$, {$D^2b$ is the Hessian of the signed distance function (see \eqref{eq:signeddistance}), $H=\Tr(D^2b)$ is the mean curvature of $\partial\Om$, $I_{d}$ is the identity matrix of size $d$, and $V_{n}=\V\cdot {\bf n}_{\partial\Om}$ is the normal component of the deformation}. Moreover, the shape derivative $u'$ at $t=0$ of the eigenfunction satisfies
\begin{equation}
\left\{
\begin{array}{rcl}
\Delta u' &\!\!=& 0 \;\;\;\;\;\;\;\;\textrm{ in }\Omega, \\[3mm]
-\beta \Deltat u'+\partial_n{u'} - \lambda u' &=& \beta \Deltat (V_n \partial_n u) +\beta \divet\big(V_n(H I_{d}-2D^2b)\nablat u\big)\\
&&+ \divet(V_n\nablat u)  -\lambda' u+\lambda  V_n( \partial_nu+ H u) \hspace{1cm} \textrm{ on }\partial\Omega.
\end{array}
\right.
\end{equation}
\item { Let $\lambda$ be a multiple eigenvalue of order $m\ge 2$.  Let  $(u_k)_{k=1,\ldots,m}$ denote the eigenfunctions associated to $\lambda$. Then there exist $m$ functions $t\mapsto \lambda_{k,\beta}(t), k=1,\ldots,m$ defined in a neighborhood of 0 such that
\begin{itemize}
\item $\lambda_{k,\beta}(0)=\lambda$,
\item for every $t$ in a neighborhood of 0, $\lambda_{k,\beta}(t)$ is an eigenvalue of $\Om_{t}=(I+t\V)(\Om)$,
\item the functions $t\mapsto\lambda_{k,\beta}(t), k=1,\ldots,m$ admit derivatives and their values at 0 are the eigenvalues of the $m\times m$ matrix $M{=M_{\Om}(V_{n})}$ of entries $(M_{ij})$ defined by 
$$
{ M_{ij}=\displaystyle \int_{\partial\Omega} V_n\Big(\nablat u_i.\nablat u_j-\partial_n{}u_i \partial_n{}u_j -\lambda H u_iu_j+ \beta\left(HI_{d}-2D^2 b\right)\nablat u_i.\nablat u_j \Big)~d\sigma.
}$$
\end{itemize}}
\end{itemize}
\end{theorem}
Notice that in the notations above and contrary to \eqref{eq:spectrum}, the functions $\lambda_{k}(t)$ are no longer ordered.
As a byproduct of this result, notice that we can write the corresponding shape derivatives for the Steklov and Laplace-Beltrami eigenvalue problem (see Appendix \ref{app:SteLB}).
Another consequence of this result, regarding our conjecture, is that we  are  able to check that balls are critical shapes for $\lambda_{1,\beta}$ by computing the trace of the previously defined matrix $M=M_{B}$ (recall that $\lambda_{1,\beta}(B)$ is an eigenvalue of multiplicity $d$ the dimension). {\Be But first, we make a short remark about the notion of volume preserving deformation:}

\begin{remark}\label{rk:volume}
{\Rd In the next results and in many places in the paper, we will consider  volume preserving smooth deformations of domains, that is to say $\Om_{t}=T_{t}(\Om)$ where
$t\mapsto T_{t}$ satisfies:
\begin{itemize}
\item
$T_0=Id$,
\item
for every $t$ near 0, $T_t$ is a $W^{3,\infty}$-diffeomorphism from $\Omega$ onto its image $\Omega_t=T_t(\Omega)$,
\item
the application $t\mapsto T_t$ is real-analytic near $t=0$.
\item for every $t$ near 0, $|\Om_{t}|=|\Om|$.
\end{itemize}
More generally, it can be sufficient to assume that the volume is preserved at the first or the second order, depending on whether we are interested in first or second order conditions.
For example, if one considers $T_{t}=I+t\V$ the vector field $\V$ is said to be volume preserving at first order if it satisfies $\int_{\partial\Om}V_{n}d\sigma=0$ ; indeed for $\Om_{t}=(I+t\V)(\Om)$, we have $\frac{d}{dt}_{|t=0}|\Om_{t}|=\int_{\partial\Om}V_{n}d\sigma$.

When dealing with second order considerations as in {\Rd Theorem} \ref{conclusion:ordre:deux}, we need that the volume is preserved at the second order, so $T_{t}$ is volume preserving at second order if
$$
\cfrac{d^2}{dt^2}| \Omega_t|_{|t=0}=\int_{\partial \Om}\left(W+V_n\partial_{n} V_n +H V_n^2 \right)~d\sigma=0,
$$
where $\V=\frac{1}{t}(T_{t}-I)$, $V_{n}$ is the value at $t=0$ of $\V\cdot n_{\partial\Om_{t}}$, and $W$ denotes the derivative of $\V\cdot n_{\partial\Om_{t}}$ with respect to $t$ at $t=0$.
}\end{remark}

\begin{proposition}
\label{trace:matrice:derivees:premieres}
Any ball $B$ is a critical shape for $\lambda_{1,\beta}$ { with volume constraint}, in the sense that for every volume preserving deformations $\V$,
$$\Tr(M_{B}(V_{n}))={\Be \sum_{k=1}^d{\lambda'_{k,\beta}}(0)=}0,$$
{\Be where $(t\mapsto\lambda_{k,\beta}(t))_{k=1\ldots d}$ are defined in Theorem \ref{Theoreme:gradient:cas}.}\\
In particular, $0 \in \partial\lambda_{1,\beta}(B;V_{n}){\Rd:=[\inf_{i=1\cdots d} \lambda_{i,\beta}'(0), \sup_{i=1\cdots d} \lambda_{i,\beta}'(0)]}$ the directional subdifferential associated to the first non trivial eigenvalue.

Moreover, this subdifferential reduces to $\{0\}$ if $V_{n}$ is orthogonal to spherical harmonics of order two: in other words, in that case, the directional derivative exists in the usual sense and vanishes. 
\end{proposition}

{ Two situations can now occur: either the subdifferential in direction $V_{n}$ is not reduced to $\{0\}$ and then one can deduce from the previous statement that $B$ locally maximizes $\lambda_{1,\beta}$ along $t\mapsto B_{t}$ (see for example (c) and (d) in Figure \ref{Fig5}), or the subdifferential in direction $V_{n}$ is $\{0\}$ and then this first order shape calculus does not allow us to conclude that the ball is a local maximizer of $\lambda_{1,\beta}$. Hence, for the directions $V_{n}$ in $\mathcal{H}$ {\bf defined as the Hilbert space generated by spherical harmonics of order greater or equal to three}, we now consider the second order analysis to wonder if the ball satisfies the second order necessary condition of optimality, and obtain the following result in dimension two and three.}

\begin{theorem}
\label{conclusion:ordre:deux}
{ Let $B$ be a ball of radius $R$ in $\R^2$ or $\R^3$ and {\Rd$ t\mapsto B_{t}=T_{t}(B)$ a second order volume preserving deformation}. $\lambda_{1,\beta}(B)$ is an eigenvalue of multiplicity $d$ the dimension, and we denote $t\mapsto\lambda_{k,\beta}(t),~k=1,\ldots,d$ the branches obtained in Theorem \ref{Theoreme:gradient:cas}. 

Then the functions $t\mapsto \lambda_{k,\beta}(t), ~k=1,\ldots,d$  admit a second derivative and their values at 0 are the eigenvalues of the $d\times d$ matrix $E=E_{B}(V_{n})$ defined in Section \ref{section:analyze:ordre2}. Moreover, there exists \Be a nonnegative number $\mu(=\mu(\beta))$ independent of the radius $R$ such that \Bk
  $$
  \Tr(E_{B}(V_{n})){\Be = \sum_{k=1}^d{\lambda''_{k,\beta}}(0)}\leq - \Be \mu K(R) \Bk \int_{\partial B} \left(|\nablat V_{n}|^2+|V_{n}|^2\right) ~d\sigma =-\Bk \mu K(R)\Bk \|V_{n}\|_{\sH^1(\partial B)}^2.
  $$
  holds for or all $V_n \in \mathcal{H}$, \Be with  $K(R)=\frac{d}{R^{2+d}\omega_{d\Be -\Bk 1}}$\Bk.
  }
\end{theorem} 
As a consequence of Proposition \ref{trace:matrice:derivees:premieres} and {\Be Theorem }\ref{conclusion:ordre:deux}, we have the result:

\begin{corollary}\label{cor:final}
If $B$ is a ball in $\R^2$ or $\R^3$, and $t\mapsto T_{t}\in W^{3,\infty}(B,\R^d)$ a smooth \Rd(second order)\Bk volume preserving deformation, then
$$\lambda_{1,\beta}(B)\geq \lambda_{1,\beta}(T_{t}(B)), \;\;\;\;\textrm{ for }t\textrm{ small enough}.$$
\end{corollary}

\paragraph{Plan of the paper.}
The paper is organized as follows: in section \ref{cas:beta:positif}, we prove Theorem \ref{main:theorem:section:brock} by adapting the strategy of Brock and present some numerical tests to illustrate the sharpness of the upper bound. The first order shape analysis is presented in section \ref{section:analyze:ordre1}, while the second order shape analysis is presented in section \ref{section:analyze:ordre2}. The background material for shape calculus and the proofs of technical intermediary results are postponed to the annexes. 

\section{Upper bound for $\lambda_{1,\beta}$}
\label{cas:beta:positif}

\subsection{Preliminary remarks and results.}
\label{remarques:preliminaires}
 
Let us start by a few remarks on the proofs in the two limit cases $\beta \rightarrow + \infty$ (that is the Laplace-Beltrami eigenvalue problem), and $\beta=0$ (that is the Steklov eigenvalue problem).\\

\noindent{\bf On the Laplace-Beltrami case}:

\medskip

\noindent The case $d=2$ is trivial: it suffices to argue on each connected component of $\partial\Omega$. We introduce $\gamma:[0,L]$ a parametrization by the arclength of a connected component $\Gamma$ of $\partial\Omega$, then for any $u\in \sH^1(\partial\Omega)$, the Rayleigh quotient can be written as  
$$ \cfrac{\displaystyle\int_{\Gamma} |\nablat u|^2}{\displaystyle\int_{\Gamma} u^2}=\cfrac{\displaystyle\int_{0}^L [(u\circ \gamma)']^2}{\displaystyle\int_{0}^L(u\circ\gamma)^2}.$$
Hence, the $\lambda_{1}^{LB}(\Gamma)$ is nothing but the infimum of $\|u'\|_{\sL^2(0,L)}^2$ among {\Rd periodic} functions $u$ with $0$ mean value and $\|u\|_{\sL^2(0,L)}=1$, that is to say $4\pi^2/L^2$. It is a decreasing function of the length of the connected component of the boundary. Then, if $\Omega$ is simply connected, combined with the isoperimetric inequality, the previous computations leads to $\lambda_{1}^{LB}(\partial\Omega)\leq \lambda_{1}^{LB}(\partial B)$ where $B$ is a disk of same area than $\Omega$.

Moreover, if $\partial\Omega$ has more than one connected component, then $\lambda_{1}^{LB}=0$ since the multiplicity of $0$ as eigenvalue is at least the number of connected component. To check that claim, it suffices to check that the functions taking the value $1$ on one of the connected component and $0$ elsewhere are independent eigenfunctions associated to the eigenvalue $0$. We conclude that in dimension $2$, $\lambda_{1}^{LB}(\partial\Omega)\leq \lambda_{1}^{LB}(\partial B)$, where $B$ is a disk of same area than $\Omega$.

The case $d=3$ is more complex. There is a classical result of J. Hersch \cite{Hersch2}: if $\Omega\subset\R^3$ is homeomorphic to the ball, then 
\begin{equation}
\label{enonce:hersch}
\lambda_{1}^{LB}(\partial\Omega)\leq \lambda_{1}^{LB}(\partial B), \textrm{ for all }\Omega \textrm{ such that } |\partial\Omega|=|\partial B|.
\end{equation}
We first extend Hersch statement to domains of same volume by a classic homogenity argument.

\begin{lemma} \label{hersch:revisite}
If $\Omega\subset\R^3$ is homeomorphic to the ball, then  
$$\lambda_{1}^{LB}(\partial\Omega)\leq \lambda_{1}^{LB}(\partial B) \textrm{ if }|\Omega|=|B|.$$
\end{lemma}

\begin{proofof}{Lemma \ref{hersch:revisite}.} One easily checks that $\Omega \mapsto \lambda_{1}^{LB}(\partial\Omega)$ is homogeneous of degree $-2$, so $\Omega\mapsto \lambda_{1}^{LB}(\Omega) |\partial\Omega|^{2/(d-1)}$ is homogeneous of degree $0$. 
Then we get from Hersch's inequality \eqref{enonce:hersch}, that 
\begin{equation}
\label{hersch:pondere}
\lambda_{1}^{LB}(\partial\Omega) |\partial\Omega|^{\frac 2 {d-1}}\leq \lambda_{1}^{LB}(\partial B) |\partial B|^{\frac 2 {d-1}}, \textrm{ for all }\Omega \textrm{ such that } |\partial\Omega|=|\partial B|.
\end{equation}
Thanks to the invariance by translation of $\lambda_{1}^{LB}$ and the perimeter, and using the $0$-homogeneity of the previous product, we get that the previous inequality is in fact valid for any ball $B$ and any domain $\Om$.
We combine with  the isoperimetric inequality
$$\cfrac{|\partial B|^{\frac d {d-1}}}{|B|} \leq \cfrac{|\partial \Omega|^{\frac d {d-1}}}{|\Omega|}$$
to conclude.
\end{proofof}

\noindent{\bf \Be On the Steklov case}:

\medskip

\noindent {\Be In the general case $\beta\geq 0$, we will adapt the original Brock's proof; the main tool is} an isoperimetric inequality for the moment of inertia of the boundary $\partial\Omega$ with respect to the origin. The general form of  the weighted isoperimetric  inequality due to F. Betta, F. Brock, A. Mercaldo and M.R. Posteraro \cite{BBMP} is:
\begin{lemma}\label{Brock}
Let $\Omega \subset \mathbb{R}^d$ be an open set  and let $f$  be a continuous, nonnegative and nondecreasing  function defined on $[0,\infty]$.  Moreover, we suppose that 
$$
t\mapsto \displaystyle \left( f(t^{\frac{1}{d}})-f(0)\right)t^{1-\frac{1}{d}} \textit{~is convex for ~} t\ge 0
$$
Then 
\begin{equation}
\label{eq:Brock}
\int_{\partial \Omega} f(\vert  x \vert  ) d\sigma  \ge   f(R) \  |\partial B_{R}|,
\end{equation}
{\Rd where $B_{R}$ is the ball centered at the origin such that $|B_{R}|=|\Omega|$}.
\end{lemma}
Let us remark that the function $t\mapsto t^p$ satisfies the assumptions of the lemma as soon as $p\geq 1$ and in particular for $p=2$. In that case and in order to prove a refinement of Brock's inequality, L. Brasco, G. De Philippis and B. Ruffini established a qualitative refinement of this inequality (Theorem B of \cite{BrascoDePhilippisRuffini}):

\begin{lemma} \label{Brock:raffine}
There exists an explicit dimensional constant $\gamma_{d}$ such that for every bounded, open Lipschitz set $\Omega$ in $\R^d$,
\begin{equation}
\label{eq:Brock:raffinee}
\int_{\partial \Omega} |x|^2 d\sigma  \geq  R^2 \  |\partial B_{R}| \left[ 1+\gamma_{d} \left(\cfrac{|\Omega\Delta B_{R}|}{|B_{R}|}\right)^2\right],
\end{equation}
where $B_{R}$ is the ball centered at the origin such that $|B_{R}|=|\Omega|$ and $\gamma_{d}$ is the constant defined in \eqref{constante:Brasco}.
\end{lemma}

\noindent{\bf \Be On the Wentzell case}:

\medskip

\noindent An important remark for the sequel is the particular case when $\Omega$ is a ball $B_R$ of radius $R$. The eigenspace corresponding to $\lambda_{1,\beta}$ is $d$-dimensional: it consists to the restrictions on the sphere $S^{d-1}_R$ of the  linear functions in $\mathbb{R}^d$ spanned by the coordinates functions. It follows, from the theory of spherical harmonic functions that 
\begin{equation}
\label{valeurs:propres:sphere}
 \lambda_{1,\beta}(B_{R}) = \lambda_{2,\beta}(B_{R})=\ldots =\lambda_{d,\beta}B_{R})=\frac{(d-1)\beta +R}{ R^2}.
\end{equation}
The Laplace-Beltrami operator on $\partial B_{R}$ and the Steklov operator also are diagonal on the basis of spherical harmonics, hence
$$ \lambda_{1,\beta}(B_{R}) = \lambda_{1}^{S}(B_{R}) +  \beta \lambda_{1}^{LB}(\partial B_{R}),$$
and more generally the eigenvalue associated to spherical harmonics of order $l$ is
\begin{equation}
\label{valeurs:propres:sphere:ordre:l}
 \lambda_{(l)}(B_{R}) =\frac{l(l+d-2)\beta +R}{ R^2}.
\end{equation}
{ But, this situation is specific to the ball: indeed, in general we only have the inequality
$$\lambda_{1,\beta}(\Om)\geq \lambda_{1}^{S}(\Om)+\beta\lambda_{1}^{LB}(\Om).$$
{\Rd Moreover, we can easily prove that for any smooth $\Om$, $\displaystyle{\lim_{\beta\to\infty}\frac{1}{\beta}\lambda_{1,\beta}(\Om)=\lambda_{1}^{LB}(\Om)}$: indeed, we have a first trivial inequality $\frac{1}{\beta}\lambda_{1,\beta}(\Om)\geq\lambda_{1}^{LB}(\Om)$ for any $\beta\geq 0$, and using the variational formulation \eqref{eq:var}, we obtain
$\forall v\in \sH(\Om)$ with the addition condition $\int_{\partial\Om}v=0$,
$$\overline{\lim_{\beta\to\infty}}\frac{1}{\beta}\lambda_{1,\beta}(\Om)\leq \overline{\lim_{\beta\to\infty}}\frac{\frac{1}{\beta}\displaystyle\int_{\Om}|\nabla v|^2+\displaystyle\int_{\partial\Om}|\nabla_{\tau}v|^2}{\displaystyle\int_{\partial\Om}v^2}=\frac{\displaystyle\int_{\partial\Om}|\nabla_{\tau}v|^2}{\displaystyle\int_{\partial\Om}v^2}$$
which leads to the result.}

For example if $d=3$, combining Brock's inequality and Lemma \ref{hersch:revisite}, we obtain that the right-hand side in the previous inequality is maximized by the ball, among domains of given volume and homeomorphic to the ball. Unfortunately, this is not enough to obtain {\Rd that balls are maximizing} the {\Be Wentzell} eigenvalue.
}

{ So in order to obtain an estimate of $\lambda_{1,\beta}$, we look into the strategies used for the extremal problems, which are the Steklov ($\beta=0$) and the Laplace-Beltrami ($\beta\rightarrow +\infty$) cases. The strategies of Brock and Hersch for those cases are actually close but distinct: they use the coordinate functions as test functions in the Rayleigh quotient characterization of eigenvalues. In the case of the Laplace-Beltrami operator though, J. Hersch had an additional step: he first transports  the surface $\partial\Omega$ on the sphere by a conformal mapping, and use the conformal invariance of the Dirichlet energy for 2-dimensional surfaces. 
In the following, we choose to follow the ideas of Brock. This allows to obtain an estimate with no assumption on the topology or the dimension of the domain. 
Indeed, the above mentioned phenomenon of decoupling between the different connected components does not appear in the Steklov case, due to the volume term, and in fact Brock's result is valid for every (smooth enough) domain.}
The same volume term appears in the {\Be Wentzell} case and the approach of Brock is then the natural one. However, one expects from these topological considerations that it will not provide an optimal result.

\subsection{Proof of Theorem \ref{main:theorem:section:brock}}
\label{preuve:brock}

Our strategy to prove Theorem \ref{main:theorem:section:brock} is to use the following characterization for the inverse trace of eigenvalues (stated by J. Hersch in \cite{Hersch1} and proved by G. Hile and Z. Xu in \cite{HileXu})
\begin{equation}
\sum_{i=1}^{d}\displaystyle\frac{1}{\lambda_{i,\beta}}=\max_{v_{1},\cdots,v_{d}} \sum_{i=1}^{d}\frac{B(v_i,v_i)}{A_\beta(v_i,v_i)},
\end{equation}
where the functions $(v_i)_{i=1,\ldots,d}$ are non zero functions that are $B$-orthogonal to the constants and pairwise $A_{\beta}$-orthogonal. 

Before proving Theorem \ref{main:theorem:section:brock}, we now present some preliminary results.

\begin{lemma}
\label{lemme:matrice:normales}
{ The matrix $P[\Omega]$defined by  \eqref{definition:matrice:A}
is symmetric,  positive definite.  Its spectral radius $\Lambda[\Omega]$  satisfies
\begin{equation}
 \label{encadrement:Lambda}
(d-1)|\partial\Om|\geq \Lambda[\Omega]\geq  \cfrac{d-1}{d} |\partial\Om|.
\end{equation}
In particular, among sets of given volume, the spectral radius is minimal for the ball.}
\end{lemma}

\begin{proof} The matrix $P(\Omega)$ is symmetric by definition. For  $\mathbf{y}=(y_1,\cdots,y_{d}) \in \mathbb{R}^d$ with $\mathbf{y} \ne \mathbf{0}$, we check that 
$$ \sum_{i,j=1}^d y_{i}(\delta_{ij}-\n_i \n_j)y_j =  \mathbf{y}^T\mathbf{y}- (\mathbf{y}^T\n)^2 \ge 0
$$
by Cauchy-Schwarz inequality.  By integration over $\partial\Omega$, $P[\Omega]$ is positive semi\-definite.
Assume, by contradiction, that $P$ is not definite: then there is a vector $\bf{y}\ne \bf{0}$ such that 
$$0= \sum_{i,j=1}^d y_{i} \left(\int_{\partial\Omega} \left(\delta_{ij}-\n_i \n_j \right)\right)y_j =  \int_{\partial\Omega} \left(\mathbf{y}^T\mathbf{y}- (\mathbf{y}^T\n)^2 \right) .
$$
The equality case of Cauchy-Schwarz inequality $\mathbf{y}^T\mathbf{y}- (\mathbf{y}^T\n)^2=0$ is therefore satisfied everywhere on $\partial\Omega$, this holds if and only if $\mathbf{y}$ and $\n$ are colinear. Hence, $\n$ is constant on $\partial\Omega$ which contradicts the boundedness of $\Omega$.

The matrix $P[\Omega]$ has positive eigenvalues. Their sum is the trace $\Tr(P[\Omega])$, hence
{ $$
\Tr(P[\Om]) \geq\Lambda[\Omega]\geq \cfrac{\Tr(P[\Omega])}{d}\;\;\;\textrm{with}\;\;\Tr(P[\Om])= \sum_{i=1}^d \int_{\partial \Omega} (1-\n_{i}^2)=(d-1)|\partial\Omega|.$$}
Therefore 
$$
(d-1) |\partial\Omega|\geq\Lambda[\Omega]\geq  \cfrac{(d-1)}{d}\ |\partial \Omega| \geq \cfrac{(d-1)}{d}\ |\partial B|.$$
The last inequality is obtained by the usual isoperimetric inequality and assuming $B$ is a ball such that $|\Om|=|B|$. 
Let us compute $\Lambda[B]$. From the invariance by rotation of the ball, there exists a real number $a$ such that $P[B]=a I_{d}$. In others words, we have 
$$
\int_{\partial B} \n_i \n_j =0,~i\neq j
\textrm{ and } 
\int_{\partial B}(1-\n_i^2) =\int_{\partial B}(1-\n_1^2),~~i=1,\ldots,d.
$$
The real number $a$ is determined using the trace of the matrix: we obtain that $d \ \Lambda[B]=(d-1) |\partial B|,$ and so $\Lambda(\Om)\geq \Lambda(B)$.
\end{proof}

\begin{remark}\label{rk:cigare}
The inequalities in \eqref{encadrement:Lambda} are sharp. The lower bound is reached when $\Omega$ is a ball and the upper bound is the limit of the collapsing stadium $S_{\varepsilon}$ (union of a rectangle and two half-disks) of unit area and width $\varepsilon$   when $\varepsilon$ tends to $0$: one checks by an explicit elementary calculus that:
$$|\partial S_{\varepsilon}|= \cfrac{2}{\varepsilon}+\cfrac{\pi\varepsilon}{2} \textrm{ while }\Lambda[S_{\varepsilon}]=\cfrac{2}{\varepsilon}.$$ 
{\Rd This example is also useful to prove \eqref{eq:inf}: indeed, we easily prove
$$\int_{\partial S_{\varepsilon}}|x|^2\geq \frac{\alpha}{\varepsilon^3},$$
where $\alpha$ is a universal constant, so using \eqref{majoration:lambda2}, we obtain \eqref{eq:inf} for $d=2$ and $m=1$. The other cases can be handled similarly.}
\end{remark}

\begin{proofof}{Theorem \ref{main:theorem:section:brock}}
We first translate and rotate coordinates $x_i,~ i=1,2,\ldots d$ such that  
$$
{\Be \forall i\neq j\in\llbracket 1,d\rrbracket^2, \;\;\;\int_{\partial \Omega} x_i  =0
\textrm{ and }
\int_{\partial \Omega} x_i x_j  = 0.}
$$

We now construct { a family which is pairwise $A_{\beta}$-orthogonal, and $B$-orthogonal to $\R$}.
We  consider a collection of  a family of functions  $w_1,w_2,\ldots,w_{d}$  in the vector space spanned by the coordinates functions: there is a matrix $C$  such that  
$$
w_{i}=\sum_{j=1}^d c_{ij}x_j,~i\in\llbracket 1,d\rrbracket.
$$
Brock used directly the coordinate functions to deal with $A_{0}$. Here, we need an $A_{\beta}$-orthogonal family, hence the matrix $C$ will be chosen to that end.  Since the coordinates functions are $\sL^2$ orthogonal to the constants, each $w_i$ is $\sL^2$-orthogonal to the constants (that is to say  the eigenfunctions associated to the smallest eigenvalue $\lambda_0=0$). 

Let us compute $A_\beta(w_i,w_j)$. First, we  get
$\nabla w_{i} =  (c_{i1}, c_{i2}, \dots , c_{id})^T$ then  
\begin{equation*}
\int_{\Omega}\nabla w_i \cdot \nabla w_j  =  \int_{\Omega} \sum_{k,m=1}^d c_{ik}c_{jm} \ = \ |\Omega| \ (CC^T)_{ij}.
\end{equation*}
To compute the second term of the sum occurring in $A_\beta$,   we recall that 
$$\nablat w_{i} \cdot \nablat w_{j}= \nabla w_{i} \cdot \nabla w_{j}- (\nabla w_{i}\cdot \n)(\nabla w_{j}\cdot \n_{j}).$$
We therefore get 
\begin{eqnarray*}
\int_{\partial\Omega} \ \nablat w_{i}\cdot \nablat w_{i} &=& \int_{\partial\Omega} \left[ \sum_{k=1}^d c_{ik}\ c_{jk}- \left( \sum_{k=1}^d c_{ik} \n_{k}\right)\left( \sum_{k=1}^d c_{jk} \n_{k}\right)\right]\\
&=& \int_{\partial\Omega} \left[ \sum_{k=1}^d c_{ik}\ c_{jk}- \sum_{k,l=1}^d c_{ik}c_{jl} \n_{k} \n_{l} \right].
\end{eqnarray*}
We introduce  $P[\Omega]$ the matrix defined in \eqref{definition:matrice:A} to get 
$$
\int_{\partial \Omega} \nablat w_i \cdot \nablat w_j \  = \sum_{k,m} c_{ik} \ p_{km} \ c_{jm}
=(CP[\Omega]C^T)_{ij}.
$$
Gathering all the terms, it comes that 
\begin{equation}
A_\beta(w_i,w_j)=|\Omega|\ (CC^T)_{ij}\ +\ \beta(CP[\Omega]C^T)_{ij}
\end{equation}
Since $P[\Omega]$ is a real symmetric matrix, we can choose an orthogonal matrix $C$  such that $CP[\Omega]C^T$ is diagonal. Hence, $CC^T=I$ and finally $w_{i}$ and $w_{j}$ are $A_{\beta}$-orthogonal if $i\ne j$ while
\begin{equation}\label{majoration:Abeta}
A_\beta(w_i,w_i)=\ | \Omega | \  + \ \beta(CP[\Omega]C^T)_{ii} \leq  \ | \Omega| \  +\beta\Lambda[\Omega].
\end{equation}
and we can apply Hile and Xu's inequality (see \cite{HileXu}).

Since by assumption
$$\int_{\partial\Omega} x_i x_j   = 0$$ 
when $i\neq j$, it comes that 
$$
B(w_i,w_i)=\sum_{k=1}^d c_{ik}^2 \int_{\partial \Omega} x_k^2~
$$
and then 
$$
S(\Omega)=\displaystyle\sum_{i=1}^d \cfrac{1}{\lambda_{i,\beta}(\Omega)} \geq\cfrac{ \displaystyle\sum_{i=1}^d\sum_{k=1}^d c_{ik}^2 \displaystyle\int_{\partial \Omega} x_k^2  }{ |\Omega| +\beta\Lambda[\Omega]}
 = \cfrac{  \displaystyle\sum_{k=1}^d \left(\int_{\partial \Omega} x_k^2 \right) \sum_{i=1}^d c_{ik}^2}{| \Omega|+\beta\Lambda[\Omega]}
 = \cfrac{ \displaystyle \int_{\partial \Omega} |x |^2  }{|\Omega|+\beta\Lambda[\Omega]},$$
{\Be  which is the first part of the result.}
Then using first  the isoperimetric weighted inequality \eqref{eq:Brock} for $p=2$, we get 
$$\int_{\partial\Omega} |x|^2  \geq R^2 |\partial B_{R}|,$$ and so
$$
{\Be \cfrac{ \displaystyle \int_{\partial \Omega} |x |^2  }{|\Omega|+\beta\Lambda[\Omega]} \geq  \cfrac{R^2  |\partial B_{R}|}{|\Omega|+\beta\Lambda[\Omega]}= \cfrac{R^2 }{\cfrac{|B_{R}|}{ |\partial B_{R}|}+\cfrac{\beta\Lambda[\Omega]}{ |\partial B_{R}|}} .}
$$
If $\Omega=B_{R}$, we know that $d |B_{R}| = R |\partial B_{R}|$ and then 
$$\cfrac{R^2 }{\cfrac{|B_{R}|}{ |\partial B_{R}|}+\cfrac{\beta\Lambda[B_{R}]}{ |\partial B_{R}|}} = \cfrac{R^2}{\cfrac{R}{d}+\beta \cfrac{d-1}{d}}=\cfrac{d}{\lambda_{1,\beta}(B_{R})},$$
and prove the equality case. By the quantitative version of the isoperimetric inequality for the moment of inertia of $\partial\Omega$ with respect to the origin \eqref{eq:Brock:raffinee}, we {\Be also} get the precise version: 
$$
{\Be \cfrac{ \displaystyle \int_{\partial \Omega} |x |^2  }{|\Omega|+\beta\Lambda[\Omega]} \geq \cfrac{R^2 |\partial B_{R}|}{|\Om|+\beta\Lambda[\Omega]}{\Rd \left[ 1+\gamma_{d} \left(\cfrac{|\Omega\Delta B_{R}|}{|B_{R}|}\right)^2\right]}.}
$$
Using the definition of $R$ and $|\Omega|=|B_{R}|$, we obtain $R^2|\partial B_{R}|=d\omega_{d}^{-1/d}|\Om|^{\frac{d+1}{d}}$
and the desired inequality. \end{proofof}

\begin{proofof}{Corollary \ref{cor:bornesup}}
Since $\lambda_{1,\beta}(\Omega)\leq\lambda_{i,\beta}(\Omega)$ for $i=1,\dots, d$, we get
$$\lambda_{1,\beta}(\Om)\leq\cfrac{d}{S(\Omega)}{\Be \leq d\cfrac{|\Omega|+\beta\Lambda[\Omega]}{ \displaystyle \int_{\partial \Omega} |x |^2  }}\leq \cfrac{d}{{\Rd  1+\gamma_{d} \left(\cfrac{|\Omega\Delta B_{R}|}{|B_{R}|}\right)^2}} \ \cfrac{|\Omega|+\beta\Lambda[\Omega]}{d\omega_{d}^{-1/d}|\Om|^{\frac{d+1}{d}}}.$$ 
\end{proofof}

\begin{proofof}{Theorem \ref{main:theorem:section:brockLB}}
It is a direct adaptation of the previous proof to the Laplace-Beltrami case: it suffices to replace the bilinear form $A_{\beta}(u,v)$ by $A(u,v)= \int_{\Omega}\nabla u. \nabla v$.  Then Equation \eqref{majoration:Abeta} becomes
$A(w_i,w_i)=(CP[\Omega]C^T)_{ii} \leq \Lambda[\Omega]$ and the conclusion follows.
\end{proofof}

\subsection{On the sharpness of the upper bounds.}
\label{section:tests:numeriques}

\paragraph{Testing the sharpness.} 

Let us denote $M_{1}(\Omega)$ the upper bound \eqref{majoration:lambda2}. In order to emphasize the improvement to the inequality of Brasco, De Philippis and Ruffini, we also plot the rougher upper bound 
$$M_{3}(\Omega)=\cfrac{|\Omega|+\beta\Lambda[\Omega]}{{\omega_{d}^{-1/d}|\Om|^{\frac{d+1}{d}}}}=d\cfrac{|\Omega|+\beta\Lambda[\Omega]}{R^2|\partial B_{R}|}.$$
It is clear from the bound of $\Lambda[\Omega]$ stated in \eqref{encadrement:Lambda} that
$$\lambda_{1,\beta}(B_{R})=M_{1}(B_{R}) \leq M_{2}(\Omega)= \cfrac{d}{\left( 1+\gamma_{d} \cfrac{|\Omega\Delta B_{R}|}{|B_{R}|}\right)^2} \ \cfrac{|\Omega|+\beta\Lambda[\Omega]}{R^2|\partial B_{R}|}. $$
\Be We also plot the shaper bound 
$$M_{1}(\Omega)=d\cfrac{|\Omega|+\beta\Lambda[\Omega]}{ \displaystyle \int_{\partial \Omega} |x |^2  }.$$
\Bk
This inequality means {\Rd that proving that balls are maximizers} would be strictly better tha\Be n\Bk \ \eqref{majoration:lambda2}. Let us illustrate this fact with some numerical illustrations. We compute $\lambda_{1,\beta}(\Omega)$ and $M_{i}(\Omega)\;(i=1,2)$ for several parametrized families of plane domains  when $\beta=1$. In Figure \ref{Fig1_1}, we present the case of ellipses of area $\pi$ (their semiaxis are $e^t$ and $e^{-t}$, $t$ is in abscissa)  while in Figure \ref{Fig1_2} and \ref{Fig1_3} we present the 
case of the star-shaped domains $\Omega_{t}$ defined in polar coordinate by $r(\theta)=a(t)(2+cos(k\theta))$ where $a(t)$ is a constant chosen such that $|\Omega_{t}|=\pi$.
\begin{figure}[!htH]\label{sharpness:upperbound}
\begin{center}
\subfigure[Ellipses of area $\pi$, ]{\label{Fig1_1} 
\includegraphics[width=0.35\textwidth]{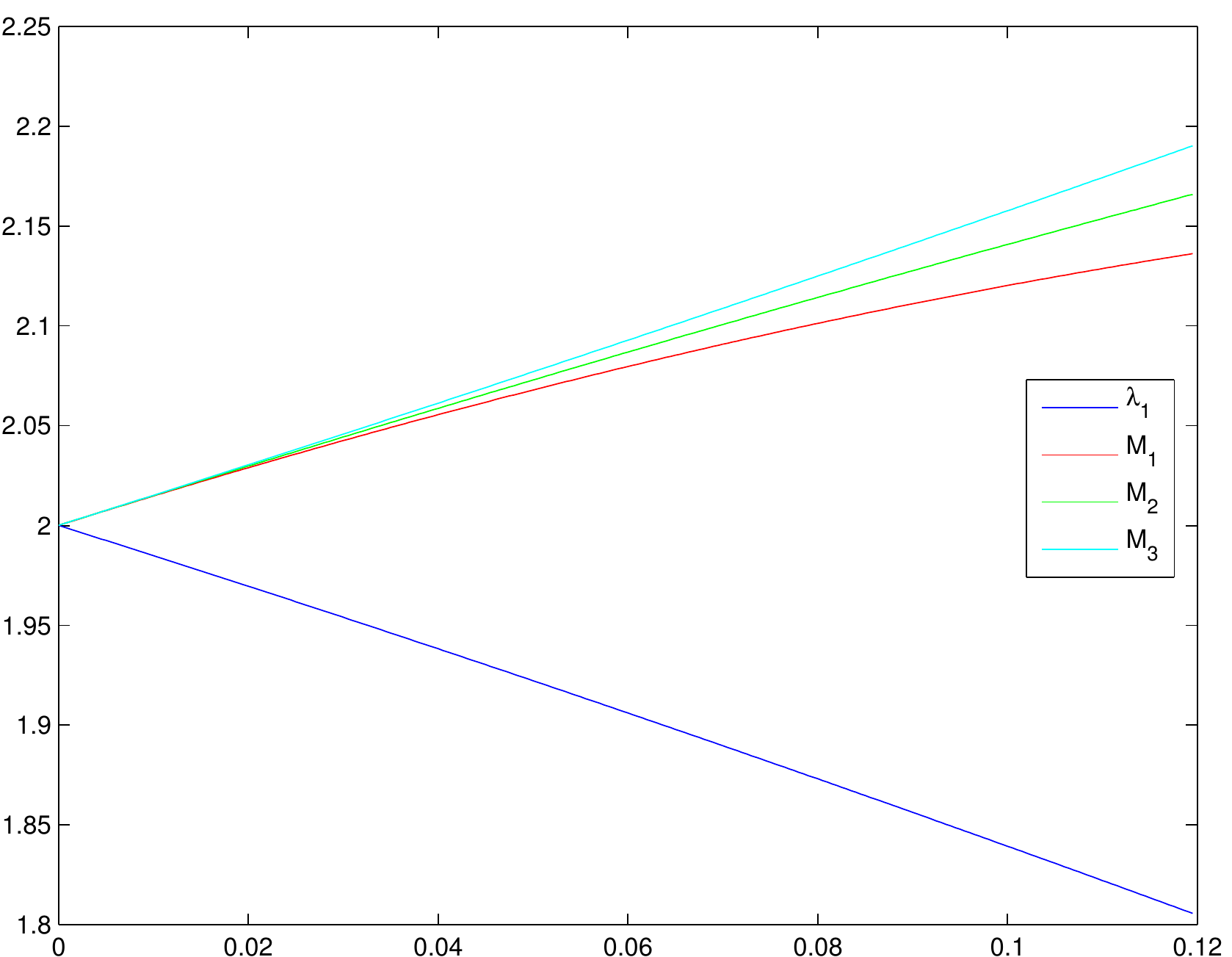}
}
\subfigure[5 branches star-shaped domains]{\label{Fig1_2}
 \includegraphics[width=0.35\textwidth]{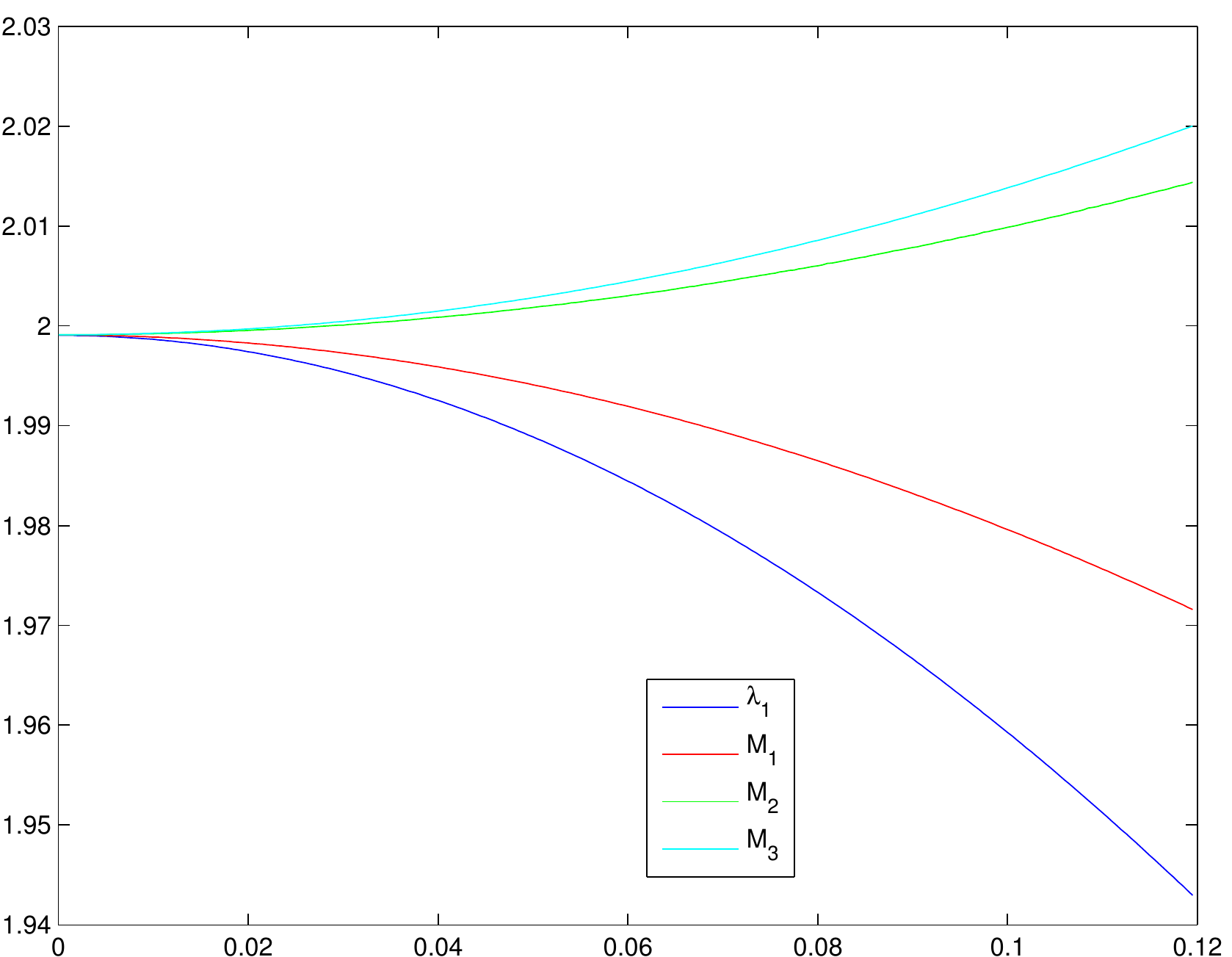}
}
\subfigure[11 branches star-shaped domains]{\label{Fig1_3}
 \includegraphics[width=0.35\textwidth]{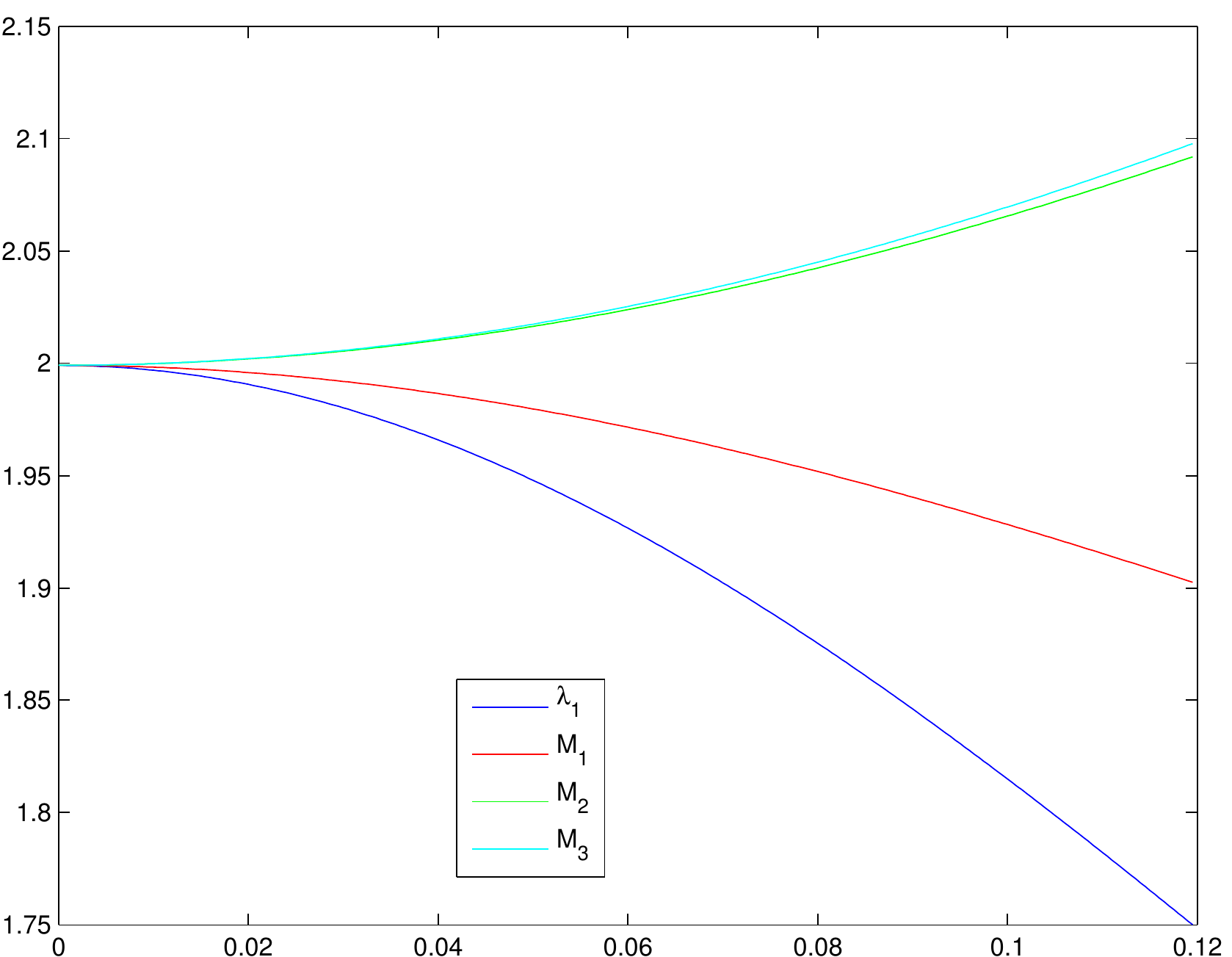}
}
\subfigure[17 branches star-shaped domains]{\label{Fig1_4}
 \includegraphics[width=0.35\textwidth]{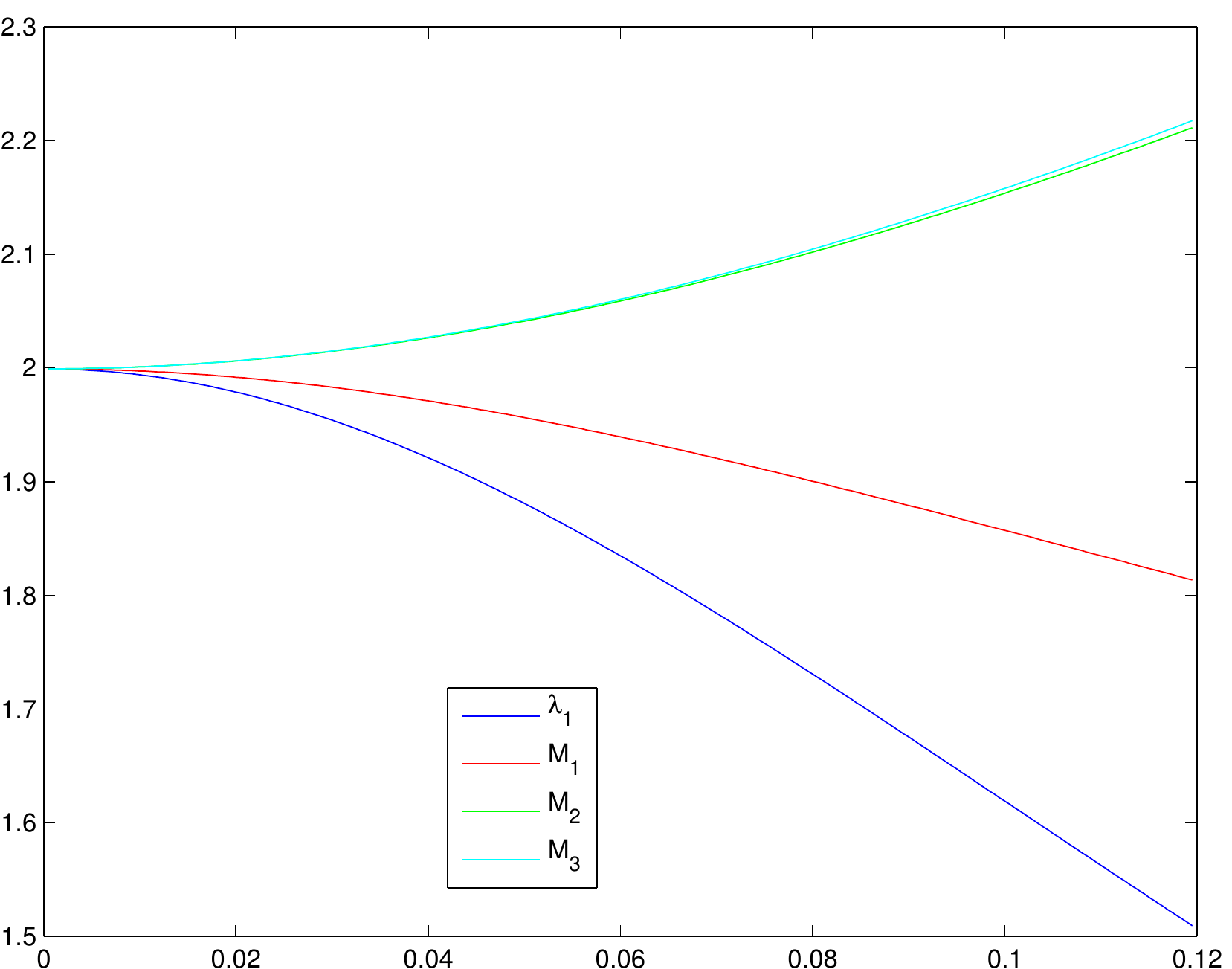}
}
\caption{Comparison of $\lambda_{1,\beta}(\Omega)$ and $M_{i}(\Omega) $. Here $\lambda_{1,\beta}(B_{1})=2$.}
\end{center}
\end{figure}

\noindent From these graphs, it seems that the upper bounds $M_{i}(\Omega)$ lack of precision when $\Omega$ is far from a ball and that {\Rd the maximality of balls is possible} and would improve the upper bound { given in Corollary \ref{cor:bornesup}}.  

\paragraph{Some numerical tests.} { It is natural to wonder if {\Rd the ball have the largest $\lambda_{1,\beta}$ among all the domains of same volume that are homeomorphic to the ball. This question cannot be solved with estimate \eqref{majoration:lambda2}, as shows Figure \ref{Fig1_1}}. { Therefore, to conclude} this section, we would like to present some numerical experiments in favor of such property.

Let us start by computing the value of $\lambda_{1,\beta}(\Omega)$ when $\Omega$ is an ellipse of fixed volume. We present here the results of our numerical computations for $\beta \in \{ 0.1,1,5,10\}$ when $|\Omega|=\pi$. Then when the volume of $\Omega$ is $4\pi$. In both figures, the abscissa stands for the eccentricity of the ellipse.  It seems that the ball maximizes $\lambda_{1,\beta}$ among ellipses of fixed area. 

\begin{figure}[!htH]
\begin{center}
\subfigure[$|\Omega|=\pi$]{\label{Fig2_1} 
\includegraphics[width=0.4\textwidth]{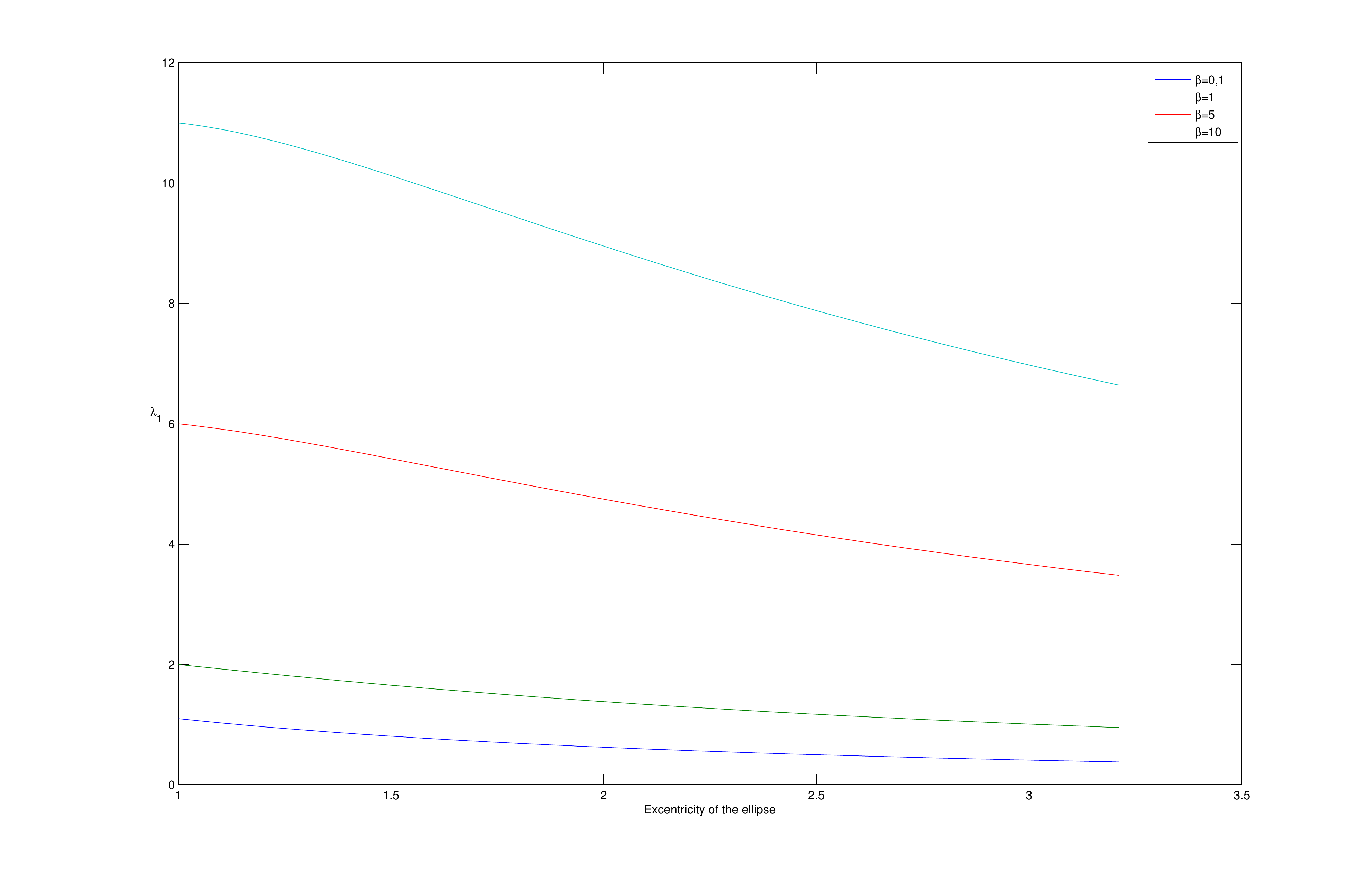}
}
\subfigure[$|\Omega|=4\pi$]{\label{Fig2_2}
 \includegraphics[width=0.4\textwidth]{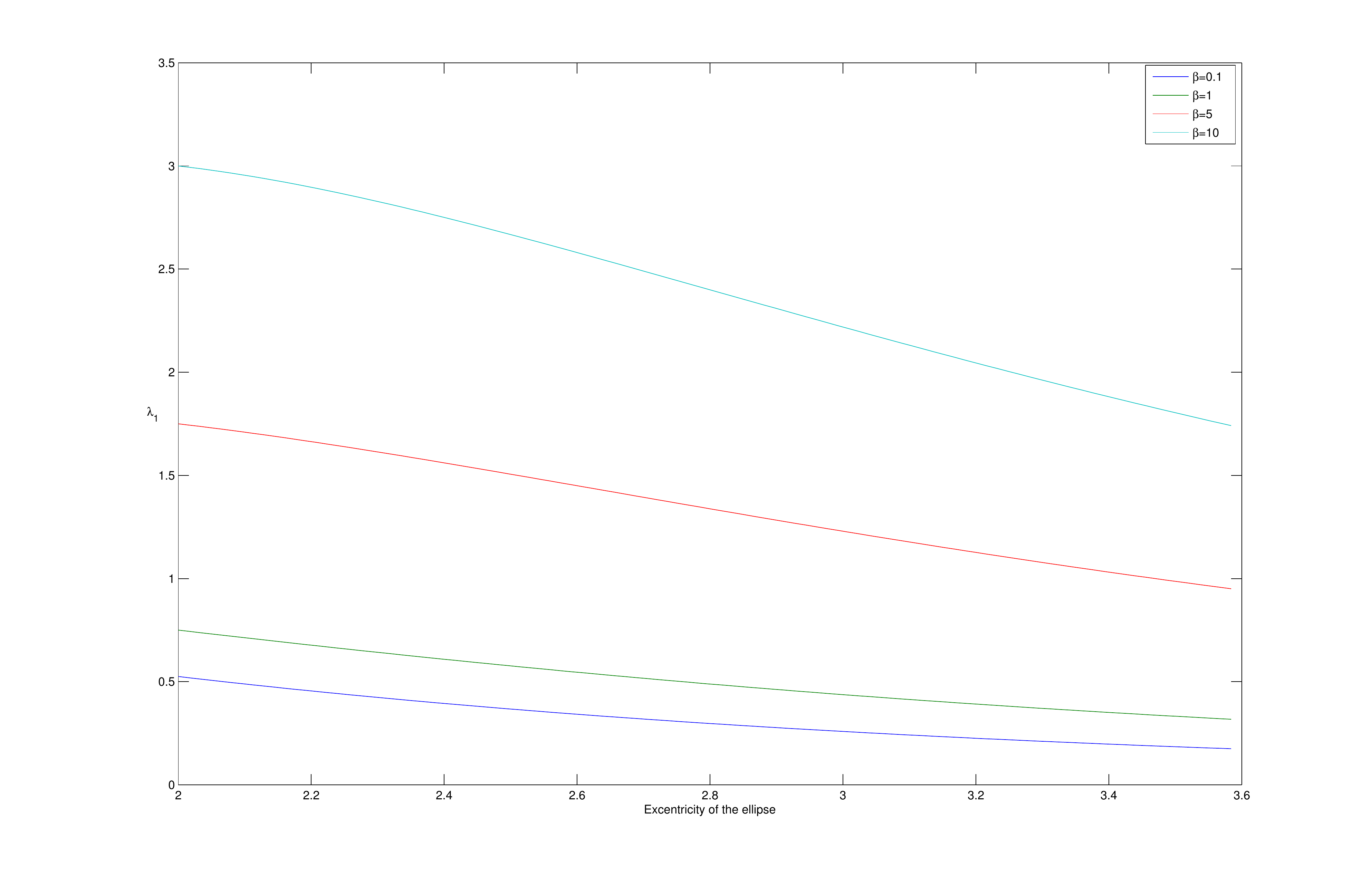}
}
\caption{$\lambda_{1,\beta}(\Omega)$ when $\Omega$ is an ellipse of volume $|\Omega|$}
\end{center}\end{figure}

Let us show some computations in dimension three. We consider families of ellipsoids with semi-axes defined by $\left(\exp(\alpha_{i}t)\right)_{i=1,2,3}$ where $\alpha_{1}+\alpha_{2}+\alpha_{3}=0$ to insure the volume constraint. The ball $B$ corresponds to $t=0$. We remind that in this case, $\lambda_{1,\beta}(B)$ has multiplicity $3$ at the sphere, we then have plotted the three corresponding eigenvalues in two cases: first for the family such that $ \alpha=(2,-0.8,-1.2)$ in Figure \ref{Fig3_1}, then for $\alpha=(2,-1,-1)$ in Figure \ref{Fig3_2}. In the last case, the defined ellipsoids are of revolution and we observe that in this particular case $\lambda_{3,\beta} \approx \lambda_{4,\beta}$. One can wonder if it is really the case.  
\begin{figure}[!htH]
\begin{center}
\subfigure[$\alpha=(2,-0.8,-1.2)$]{\label{Fig3_1} 
\includegraphics[width=0.35\textwidth]{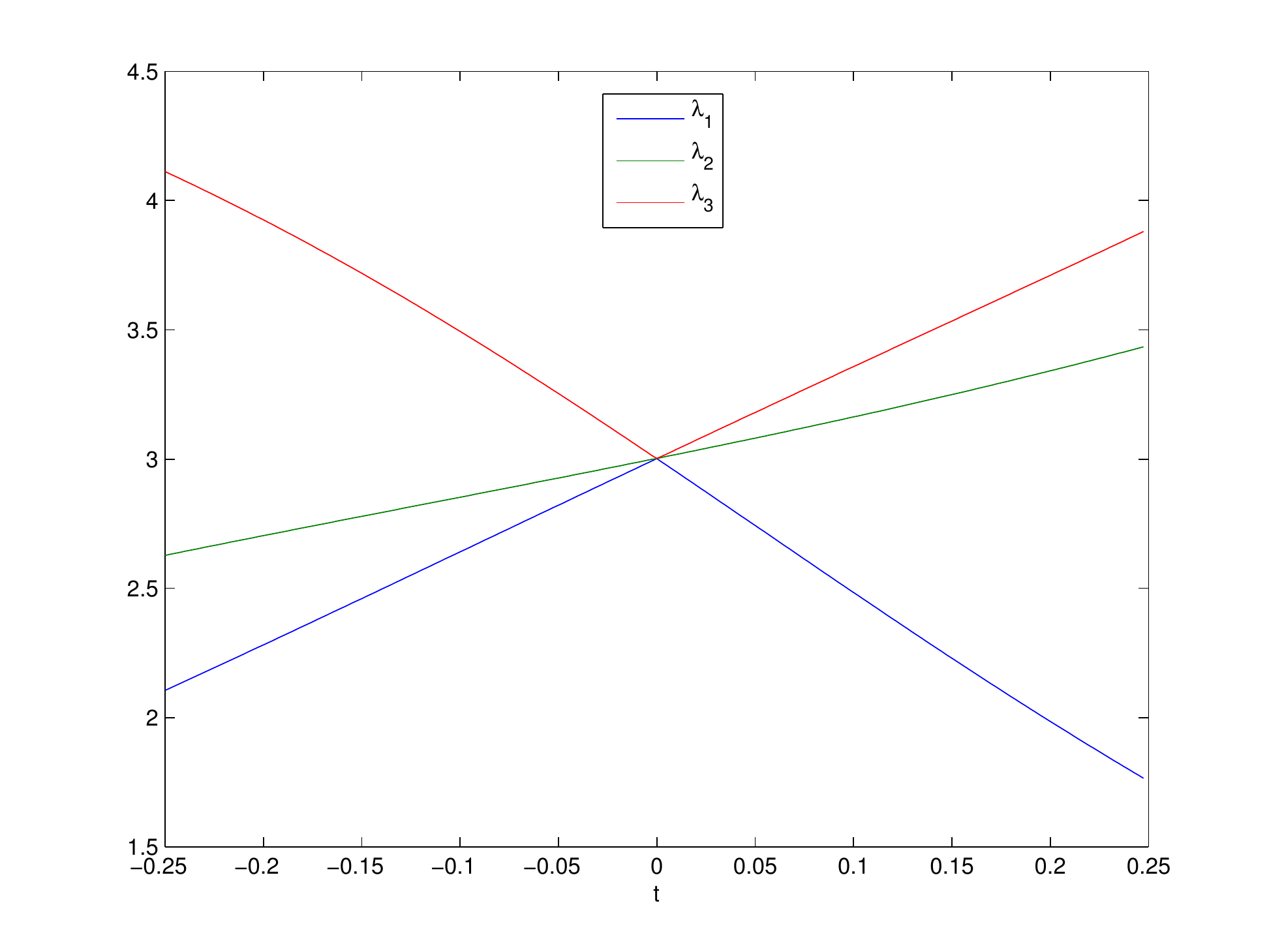}
}
\subfigure[$\alpha=(2,-1,-1)$]{\label{Fig3_2}
\includegraphics[width=0.35\textwidth]{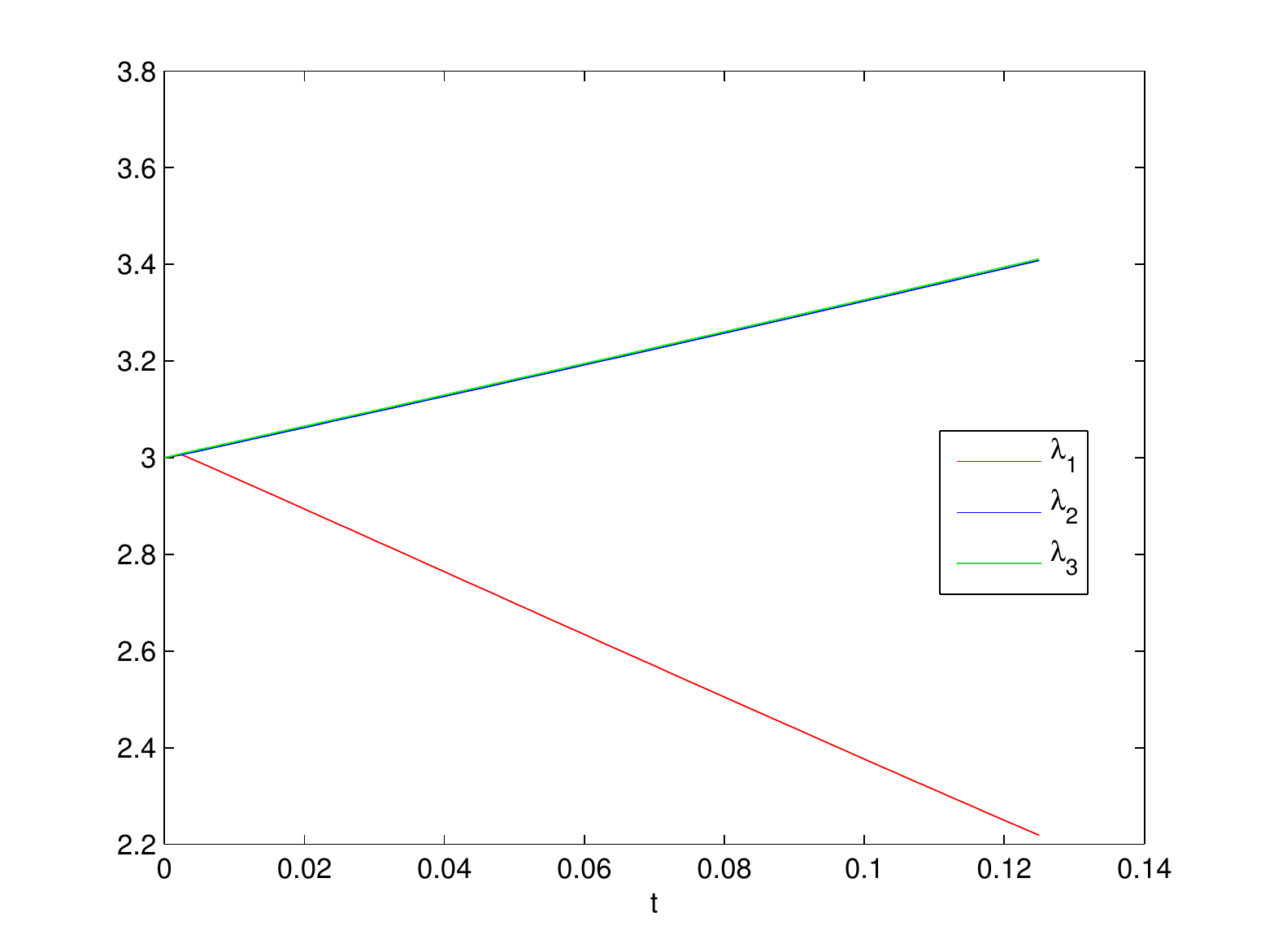}
}
\caption{$(\lambda_{1,\beta}(\Omega_{t}),\lambda_{2,\beta}(\Omega_{t}),\lambda_{3,\beta}(\Omega_{t}))$ when $\Omega_{t}$ is a parametrized  ellipsoid of volume $4\pi/3$}
\end{center}\end{figure}

Let $E(a,b)$ be an ellipsoid of volume $4\pi/3$ where $a$ is the larger semiaxis and b the middle one. We now show in Figure \ref{fig:ellipsoid} the surfaces $z=\lambda_{i,\beta}(E(a,b))$ where $i=1,2,3$. The pictures have been obtained by interpolation after the computations of the eigenvalues on $2700$ ellipsoids. Again one can attest that the ball seems to maximize $\lambda_{1,\beta}$ among ellipsoids.

\begin{figure}[!htH]
\begin{center}
\subfigure[$\lambda_{1,\beta}(E(a,b))$]{\label{Fig4_1} 
\includegraphics[width=0.3\textwidth]{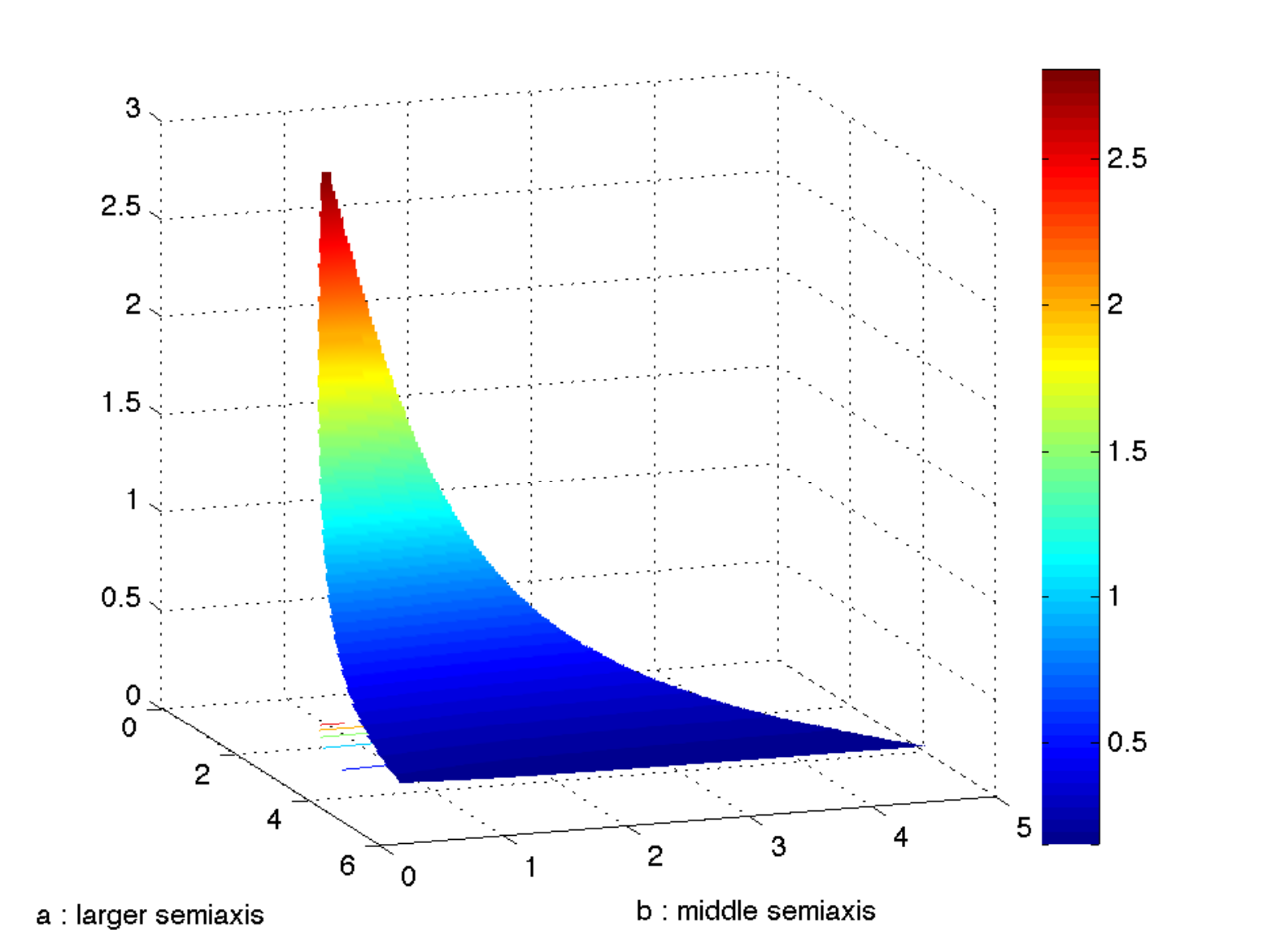}
}
\subfigure[$\lambda_{2,\beta}(E(a,b))$]{\label{Fig4_2}
\includegraphics[width=0.3\textwidth]{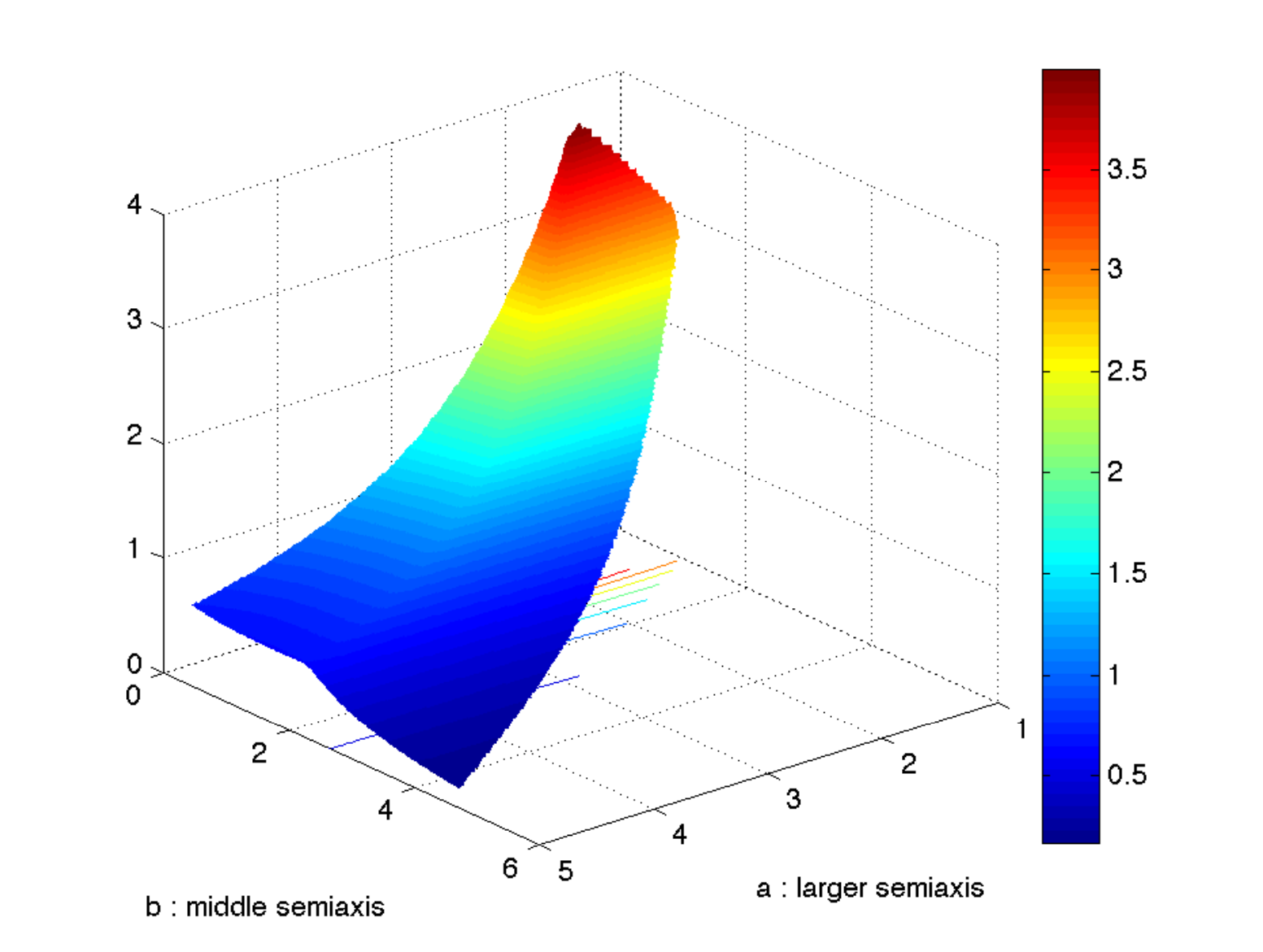}
}
\subfigure[$\lambda_{3,\beta}(E(a,b))$]{\label{Fig4_3}
\includegraphics[width=0.3\textwidth]{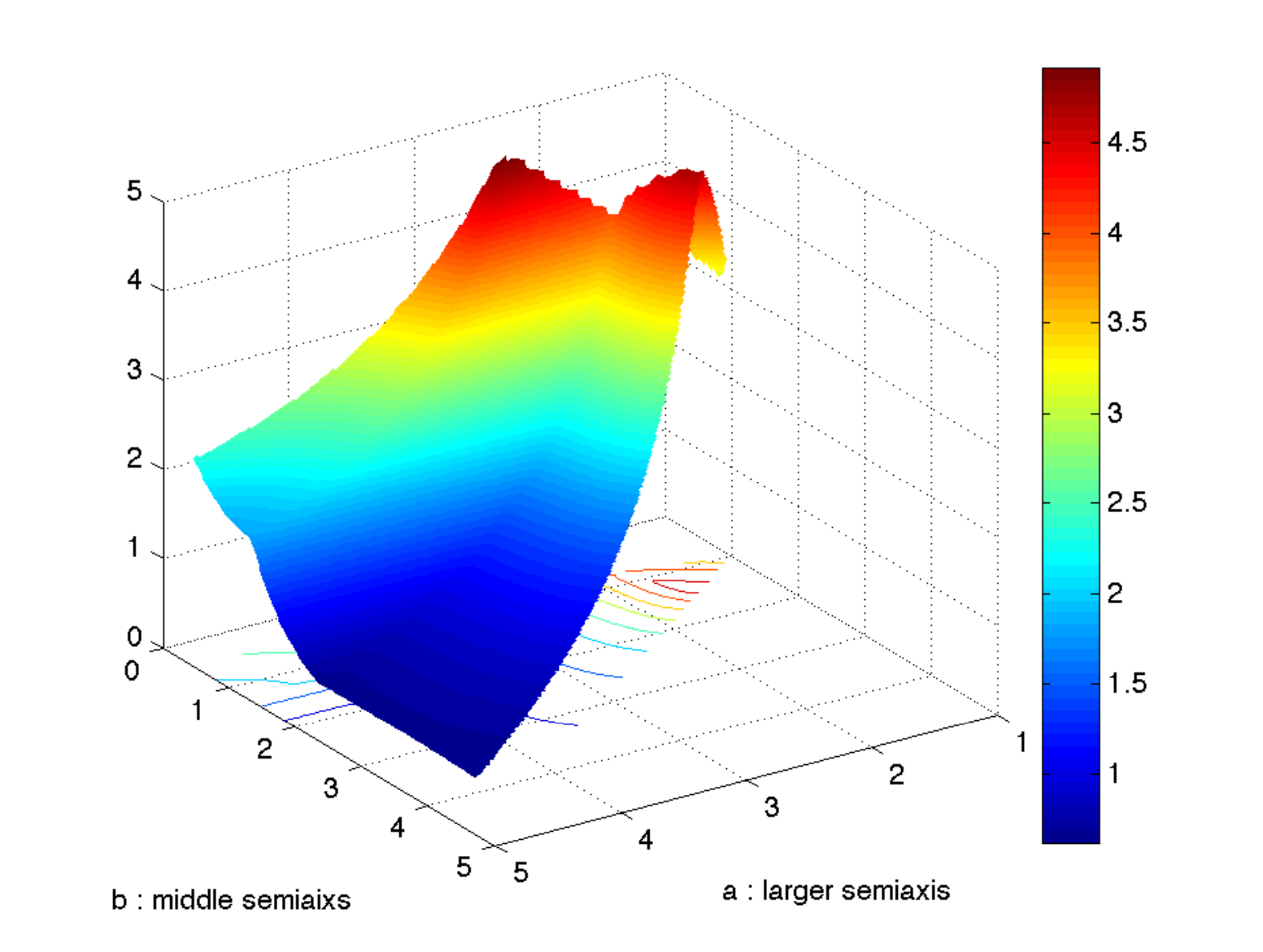}
}

\caption{$(\lambda_{1,\beta}(\Omega),\lambda_{2,\beta}(\Omega),\lambda_{3,\beta}(\Omega))$ when $\Omega=E(a,b)$ is an ellipsoid of volume $4\pi/3$}
\label{fig:ellipsoid}\end{center}\end{figure}

\section{First order shape calculus }
\label{section:analyze:ordre1}

In order to go one step further, we adopt a shape optimization point of view and prove in this section that the ball is a critical point. The main difficulty here is that the eigenvalue $\lambda_{1,\beta}(B)$ has multiplicity the dimension of the ambient space. We need some technical material on shape derivative and tangential calculus on manifold to justify the results stated in this section;  to simplify the reading of this work, we postpone these reminders in Appendix \ref{app:diffcal}.

{ Let us emphasize that from this point we do not make the assumption $\beta\geq0$, and therefore all the results of this section and the following are valid for any $\beta\in\R$.} {\Be Though from now on we drop the notation $\beta$ in $\lambda_{1,\beta}$ since there is no possible confusion anymore.}

\subsection{Notations and  preliminary result for shape deformation}

{ We adopt the formalism of Hadamard's shape calculus and consider the map $t\mapsto T_t=I+t\V$ where $\V\in W^{3,\infty}(\Om,\R^d)$ and  $t$ is small enough.
We denote by $$\Omega_t=T_t(\Omega)=\{ x+t\V(x), x\in \Om\}.
$$
\begin{remark}\label{rk:deformation}
More generally the results and computations from this section are valid if $t\mapsto T_{t}$ satisfies:
\begin{itemize}
\item
$T_0=Id$,
\item
for every $t$ near 0, $T_t$ is a $W^{3,\infty}$-diffeomorphism from $\Omega$ onto its image $\Omega_t=T_t(\Omega)$.
\item
The application $t\mapsto T_t$ is real-analytic near $t=0$.
\end{itemize}
\end{remark}}
We need to introduce the surface jacobian $\omega_t$  defined as 
$$
\omega_t(x)=\textrm{det}(DT_t(x))\parallel (DT_t(x)^{T})^{-1}\mathbf{n}(x)\parallel,
$$
and the functions
$$
A_t(x)=(DT_t(x))^{-1} (DT_t(x)^T)^{-1}
, \;\;\;\tilde{A}_{t}(x)=\textrm{det}(DT_{t}(x)) A_{t}(x), \;\;\;\;C_t(x)=\omega_t(x)A_t(x).
$$
We have to study the transport of the considered eigenvalue problem on the deformed domain $\Omega_t$. \Rd To that end, we first rewrite the deformed equation on the fixed domain $\Omega$ and its boundary $\partial\Omega$: we have to describe how are transported the Laplace-Beltrami and the Dirichlet-to-Neumann operators. \Bk

\paragraph{Transport of the Dirichlet-to-Neumann map.}

Let us consider  the Dirichlet-to-Neu\-mann operator defined on its natural space $\D_{t}:  H^{1/2}(\partial\Omega_t)\rightarrow H^{-1/2}(\partial\Omega_t)$. It maps a function $\phi_t$ in $H^{1/2}(\partial\Omega_{t})$ onto the normal derivative of its harmonic expansion in $\Omega_{t}$, that is to say 
$\D_{t}(\phi_t)=\partial_{\n_{t}} u_t$, where $u^t$ solves the boundary values problem:
\begin{equation}
\label{definition:Dirichlet:to:Neumann}
\left\{
\begin{array}{rcll}
-\Delta u_t&=&0 &\textrm{ in } \Omega_{t},\\
 u_t&=& \phi_t &\textrm{ on }\partial \Omega_{t}.
\end{array}
\right.
\end{equation}

To compute the quantity $\mathcal{D}_{t}$ such that $\mathcal{D}_{t}(\phi_t\circ T_{t})= [\D_{t}(\phi_t)] \circ T_{t}$ , we transport the boundary value problem \eqref{definition:Dirichlet:to:Neumann}  back on the domain $\Omega$. \Rd In others words, $\mathcal{D}_{t}$ makes the following diagram commutative:
\begin{center}
\begin{tikzpicture}
  \matrix (m) [matrix of math nodes,row sep=3em,column sep=4em,minimum width=2em]
  {
     \sH^{1/2}(\partial\Omega_{t}) &  \sH^{1/2}(\partial\Omega) \\
      \sH^{-1/2}(\partial\Omega_{t}) &  \sH^{-1/2}(\partial\Omega)  \\};
  \path[-stealth]
    (m-1-1) edge node [left] {$\D_{t}$} (m-2-1)
            edge node [above] {$T_{t}$} (m-1-2)
    (m-2-1.east|-m-2-2) edge node [above] {$T_{t}$} (m-2-2)
    (m-1-2) edge node [right] {$\mathcal{D}_{t}$} (m-2-2);
\end{tikzpicture}
\end{center}

\noindent To be more precise, we have  the following result proved in \cite{DK2012}.
\begin{lemma}\label{transport:Dirichlet:2:Neumann} 
Given $\psi\in H^{1/2}(\partial\Omega)$, we denote $v^t$ the solution of the boundary value problem 
\begin{equation}
\left\{
\begin{array}{rcll}
-\dive(\tilde{A}_{t}\nabla  v^{t})&=&0 &\textrm{ in } \Omega,\\
 v^{t}&=& \psi &\textrm{ on }\partial \Omega.
\end{array}
\right.
\end{equation}
and then define  $\mathcal{D}_{t}\psi\in H^{-1/2}(\partial\Omega)$ as:
$$
\mathcal{D}_{t}\psi: f \in H^{1/2}(\partial\Omega) \mapsto \int_{\Omega} \tilde{A}_{t}(x)\nabla v^{t}(x)\cdot\nabla E(f)(x)\dd x,$$
where $E$ is a continuous extension operator from $H^{1/2}(\partial\Omega)$ to $H^1(\Omega)$.
Then the relation  
\begin{equation}
\label{eq:transport:Dirichlet:2:Neumann:cas:trace}
{{(\D_{t}\varphi)\circ T_{t} }}= \mathcal{D}_{t}\left[\varphi \circ T_{t} \right]
\end{equation}
holds for all functions $\varphi \in H^{1/2}(\Omega_t)$.
\end{lemma}

\Bk\noindent Setting $u^{t}=u_t\circ T_{t}$, we check from the variational formulation,  that  the function $u^{t}$ is the unique solution of the transported boundary value problem:
\begin{equation}
\label{definition:Dirichlet:to:Neumann:transporte}
\left\{
\begin{array}{rcll}
-\dive(\tilde{A}_{t}\nabla  u^{t})&=&0 &\textrm{ in } \Omega,\\
 u^{t}&=& \phi_t\circ T_t &\textrm{ on }\partial \Omega.
\end{array}
\right.
\end{equation}
Hence, setting $y=T_t(x),~x\in \Omega$ we get formally
\begin{equation*}
\D_{t}(\phi_t)(y) = \nabla u_t(y).\n_{t}(y) 
= (DT_{t}(x)^T)^{-1}\nabla v_{t}(x). \frac{(DT_{t}(x)^T)^{-1}\n(x)}{\| (DT_{t}(x)^T)^{-1}\n(x)\|}= \frac{A_{t}(x)\n(x).\nabla u^{t}(x)}{\| (DT_{t}(x)^T)^{-1} \n(x)\|}.
\end{equation*}

Here again, we can give a sense to the co-normal derivative $A_{t}\n.\nabla u^t$  thanks to  the boundary value problem \eqref{definition:Dirichlet:to:Neumann:transporte}: this quantity  is defined in a weak sense as the previous Dirichlet-to-Neumann  operator $\D_{t}$.

\paragraph{Transport of the Laplace-Beltrami operator.}

We recall now  the expression of the transported Laplace-Beltrami operator, relying on the relation 
\begin{equation}
\label{eq:transport:Laplace:Beltrami:cas:surface}
\forall \varphi \in H^{{2}}(\partial\Omega_{t}), \;\;{ {(\Delta_{\tau} \varphi)\circ T_{t}}}= \frac{1}{\omega_{t}(x)} \divet\left( C_{t}(x) \nablat(\varphi\circ T_{t})(x)\right) \textrm{ on } \partial\Omega.
\end{equation}
Let us denote  by    $\mathcal{L}_{t} $ the operator defined as 
\begin{align}
\label{definition:L(h)}
\mathcal{L}_{t}& \left[\varphi \circ T_{t} \right](x)= \nonumber \\ 
 &\cfrac{1}{\omega_{t}(x)} ~\divet\left\{ C_{t}(x) \nabla_{} \left[\varphi \circ T_{t} \right](x) - \cfrac{C_{t}(x)\nabla \left[\varphi \circ T_{t} \right](x).\n(x) }{A_{t}(x)\n(x).\n(x)} A_{t}(x)\n(x) \right\}
\end{align}
for  $\varphi \in H^{5/2}( \Omega_{t})$.  In  \cite{DK2012},  we show the following lemma:
\begin{lemma}
\label{transport:Laplace:Beltrami:cas:trace}
The identity  
\begin{equation}
\label{eq:transport:Laplace:Beltrami:cas:trace}
\left[\Delta_{\tau} {\Rd \varphi}\right] \circ T_{t} = \mathcal{L}_{t}\left[\varphi \circ T_{t} \right]
\end{equation}
holds  for all functions $\varphi$ belonging to  $H^{5/2}( \Omega_{t})$.
\end{lemma}

\subsection{Regularity of the eigenfunctions and eigenvalues with respect to the parameter}

The section is a slight variation of a theorem due  to  Ortega and Zuazua on  the existence and regularity of eigenvalues and associated eigenfunctions in the case of Stokes system \cite{ZuazuaOrtega}. The difficulty comes from the possible multiple eigenvalues. The main result is, for a fixed deformation field $\V\in W^{3,\infty}(\Om,\R^d)$, the existence of smooth branches of eigenvalue. In other words, the eigenvalues are not regular when sorted in the increasing order, but can be locally relabeled around the multiple point in order to remain smooth. The restriction is that this labeling depends on the deformation field $\V$ hence one cannot hope to prove Fr\'echet-differentiability.

\begin{theorem}\label{analyticity}
Let $\Omega$ be an open smooth bounded domain of $\mathbb{R}^d$. Assume that $\lambda$ is an  eigenvalue of multiplicity $m$ of the {\Be Wentzell}-Laplace operator. We suppose that  $T_t=I+t \V$
for some $\V \in W^{3,\infty}(\Omega,\mathbb{R})^d$ and denote $\Omega_t=T_t(\Omega)$. Then there exists $m$ real-valued continuous functions $t\mapsto \lambda_i(t),~i=1,2,\ldots,m$ { and $m$ functions $t\mapsto u_{i}^t\in H^\frac{5}{2}(\Om)$} such that the following properties hold 
\begin{enumerate}
\item
$\lambda_{i}(0)=\lambda,~i=1,\ldots,m$,
\item
the functions $t \mapsto \lambda_i(t)$ and $t\mapsto u_{i}^{t}, ~i=1,2, \ldots,m$ are  analytic in a neighborhood of $t=0$.
\item
{ The functions $u_{i,t}$ defined \Rd by \Bk$u_{i,t}\circ T_t=u_{i}^{t}$ are {\Be normalized } eigenfunctions associated to $\lambda_{i}(t)$ on the moving domain $\Omega_t$. If one considers $K$ compact subset such that $K\subset \Om_{t}$ for all $t$ small enough, then $t\mapsto {u_{i,t}}_{|K}$ is also an analytic function of $t$ in a neighborhood of $t=0$. 
}
\item  Let $I\subset \mathbb{R}$ be an interval such that $\overline I$ contains only the eigenvalue $\lambda$ of the {\Be Wentzell} problem of multiplicity $m$.  Then there exists a neighborhood of $t=0$ such that $\lambda_{i}(t)$ $i=1,\ldots,m$ are the only eigenvalues of $\Om_{t}$ which belongs to $I$.
\end{enumerate}
\end{theorem}
\begin{proof}
Let  $\lambda$ be an eigenvalue of multiplicity $m$ and let  $u_1,\ldots,u_m$ the orthonormal eigenfunctions associated to $\lambda$.
Let  $(\lambda(t),u_t)$  be an eigenpair satisfying 
\begin{equation*}\label{eq:Ventcelt}
({P}_t)
\left\{\begin{array}{cll}
\displaystyle -\Delta u_t &=0&\textrm{ in }\Omega_t,\\
\displaystyle -\beta \Deltat u_t+\partial_{n_t}u_t&=\lambda(t) u_t&\textrm{ on }\partial\Omega_t.
\end{array}\right.
\end{equation*}
Setting $u^t=u_t\circ T_t$, Lemma \ref{transport:Dirichlet:2:Neumann} (transport of the Dirichlet-to-Neumann map) and \ref{transport:Laplace:Beltrami:cas:trace} (transport of the Laplace-Beltrami operator) show that the system $(P_t)$  above is equivalent to  the following equation set on the boundary    
\begin{equation}
(-\beta \mathcal{L}_{t} + \mathcal{D}_{t}) u^{t} =\lambda(t)\omega_t u^{t}\textrm{ on }\partial\Omega.
\end{equation}
Consider the operator $S(t)$ defined \Rd on \Bk $\sH^{3/2}(\partial\Omega)$ by
\begin{equation}
\label{operatorS}
v\mapsto S(t) v=
-\beta \mathcal{L}_{t}v + \mathcal{D}_{t}v  
 \end{equation}
From their expressions computed for example in \cite[\Rd Section 5-2\Bk]{HenrotPierre} and the regularity assumption on $T_{t}$, all the operators $C_t$, $A_t$ and $\omega_t$  are analytic in a neighborhood of $t=0$ .  Since ${ \det(DT_t)} >0$ for $t$ small enough,  we deduce that all the expressions involved in $\mathcal{C}_t $, $\mathcal{L}_{t} $ and  $\mathcal{D}_{t}  $ are analytic in a neighborhood of $t=0$. This enables us to conclude that $S(t)$ is also analytic in a neighborhood of zero.
\par
\noindent
To show that the eigenvalues and the corresponding eigenfunctions are analytic in a neighborhood of zero, we apply the Lyapunov-Schmidt  reduction in order to treat a problem on a finite dimensional space, namely the kernel of $S(0)- \lambda I$.  To that end, we rewrite the problem $(P_t)$   on the fixed domain $\partial \Omega$   as 
$$ S(t)(u^t)-\lambda(t)\omega_t u^t=0.$$
From the decomposition
$$
(S(0)-\lambda)(u^t)=\Big[(S(0)-S(t))+\left[(\lambda(t)-\lambda)\omega_t +\lambda(\omega_t-1)\right]\Big]u^t,
$$
 $u^t$ is solution of the equation
\begin{equation} \label{equation:introductionW}
(S(0)-\lambda )(u^t)= W(t,\lambda(t)-\lambda)u^t,
\end{equation}
where we have set $R(t)=S(0)-S(t)+\lambda(\omega_t-1)$ and $W(t,\alpha)=R(t)+\alpha \omega_{t}I$.
From the Lyapunov-Schmidt Theorem {\ (see \cite[Lemma 3-2, p. 999]{ZuazuaOrtega})}, we obtain that  $S(0)-\lambda$ has a right inverse  operator denoted by $K$. Hence  the equation above implies that 
$ u^t = K    W(t,\lambda(t)-\lambda) u^t + \psi_t $ where $\psi_t \in \textrm{~Ker~}(S(0)-\lambda )$, i.e $\psi_t=\sum_{k=1}^m c_k(t) \phi_k$ where $(\phi_k)$ is a basis of $\textrm{~Ker~}(S(0)-\lambda)$.  
Notice that $I-KW(t,\lambda(t)-\lambda)$ is invertible on $\mathrm{Ker}(S(0)-\lambda I)$, the inverse of his operator restricted to this kernel will be denoted by $(I-KW(t,\lambda(t)-\lambda))^{-1}$ so that 
$$u^t=(I-KW(t,\lambda(t)-\lambda))^{-1} \psi_{t}.$$
From \eqref{equation:introductionW}, $ W(t,\lambda(t)-\lambda) u^t$ belongs to $\mathrm{Im}(S(0)-\lambda )= \mathrm{Ker}^\perp(S(0)-\lambda )$ since $S(0)$ is a Fredholm selfadjoint operator, and then
\begin{equation}\label{vector_c}
\sum_{k=1}^m c_k(t)\langle W(t,\lambda(t)-\lambda)(I-KW(t,\lambda(t)-\lambda))^{-1}\phi_k , \phi_i\rangle=0,~~i=1,2,\ldots,m,
\end{equation}
{\Rd where $\langle\cdot,\cdot\rangle$ denote the scalar product of $L^2(\partial\Om)$.}
This shows that  a  vector of coefficients $C=(c_j)_{j=1,\ldots,m}\neq 0$ is  a solution if and only if the determinant of the $m\times m$ matrix $M(t,\lambda(t)-\lambda)$  with entries 
$$M(t,\alpha)_{i,j}=\langle W(t,\alpha)(I-KW(t,\alpha))^{-1}\phi_j , \phi_i\rangle$$ satisfies $$\det{(M(t,\lambda(t)-\lambda))}=0.$$Hence $\lambda(t)$ is an eigenvalue of our problem if and only if  $\det{(M(t,\lambda(t)-\lambda))}=0$.  Note that $t\mapsto M(t,\lambda(t))$ is analytic around $t=0$. 

For small values of $t$ the operator  $(I-K{W}(t,\alpha))^{-1} $ is well defined since  $I-KW(0,0)=I$ and $t\mapsto (I-K{W}(t,\alpha))^{-1} $ is analytic around $t=0$.  On the other hand,  if $\det{M(t,\alpha)}=0$ then (\ref{vector_c}) has a nontrivial solution $c_1(t),\ldots,c_m(t)$ and this means that $\lambda(t)=\lambda+\alpha$ is an eigenvalue of $(P_t)$. 
\par

We focus now on $\det{M(t,\alpha)}$ for $\alpha\in\mathbb{R}$. From the fact that ${W}(0,\alpha)=\alpha I$, it comes that for sufficiently small values of $\alpha$, the operator  
$ I-K W(0,\alpha)$ is invertible on $\mathrm{Ker}(S(0)-\lambda I)$ and from the Von Neumann expansion we write 
$$
\langle W(0,\alpha)(I-K  W(0,\alpha))^{-1}\phi_i,\phi_j\rangle=\alpha\Big[ \delta_{ij}+\sum_{k=1}^\infty \alpha^{k}\langle K^k\phi_i,\phi_j\rangle\Big];
$$
hence 
$$
\det{(M(0,\alpha))}=\alpha^m +\sum_{i=1}^\infty \beta_{i}\alpha^{m+i}=\alpha^m(1+\sum_{i=1}^\infty \beta_{i}\alpha^{i}).
$$
Since $\det{(M(0,\alpha))} \neq 0$ is the restriction on $t=0$ of $\det{(M(t,\alpha))}$, we deduce from the Weierstrass preparation theorem    that there is neighborhood of  $(0,0)$ such that $\det{(M(t,\alpha))}$ is uniquely representable as 
$$
\det{(M(t,\alpha))}= P_m (t,\alpha) h(t,\alpha)
$$
where $$P_m(t,\alpha)=\alpha^m+\sum_{k=1}^m a_k(t)\alpha^{m-k}$$ and where $$h(t,\alpha)\neq 0.$$ Furthermore, the coefficients $a_k(t),~k=1,\ldots,m$ are real  and  analytic in a neighborhood of $t=0$. Then $\det{(M(t,\alpha))}=0$ if and only if $P_m(t,\alpha)=0$. If $\alpha_k(t),~k=1,\ldots,m$ are the real roots of the polynomial, we take $\lambda_{1}(t)=\lambda+\alpha_1(t)$ if $\alpha_1(t)$ is not identically equal to zero. 

We now have to find the $(m-1)$ other branches $\lambda_i(t)$ and the corresponding eigenfunction $u_{i,t}$ for $i=2,\ldots,m$.  
We use the idea of the deflation method by considering the operator 
$$
S_2(t)=S(t)-\lambda_1 P_1(t)
$$
where $P_1$ is the orthogonal projection on the subspace  spanned by $u_{1,t} $.  At $ t=0$, we obtain 
$$
S_2(0)u_j=S(0)u_j-\lambda \delta_{1j} u_j 
$$
in other terms  $S_2(0)u_j=\lambda u_j,~j=2,\ldots,m$ while $S_2(0)u_1=0$. This shows  that $\lambda$ is an eigenvalue of multiplicity $m-1$ of $S_2(0)$ with eigenvalues $u_{2},\ldots,u_{m}$.  One can show that these functions  are the only linearly independent eigenfunctions associated to $\lambda$. Now  we can apply the same recipe used before to the operator $S_2$ instead of $S$. We then get a branch $\lambda_2(t)$ such that $t\mapsto \lambda_2(t)$ is analytic in a neighborhood of $t=0$.  Iterating the process, we get at the end the $m-$ branches $\lambda_i(t),~i=1,\ldots,m$ such that each branch is analytic in a neighborhood of $t=0$ and $m$ corresponding eigenfunctions forming  an orthonormal set  of functions in $H^{\frac{3}{2}}(\partial\Om_t)$. 
\par
\noindent
The proof of the  last item follows the same lines than the proof of Ortega and Zuazua for the Stokes system, see \cite{ZuazuaOrtega}. 
\end{proof}

{\Rd \begin{theorem}\label{th:u'}
With the notations of Theorem \ref{analyticity}, if $t\mapsto (\lambda(t),u_{t})$ is one of the smooth eigenpair path $(\lambda_{i}(t),u_{i,t})$ of $\Om_{t}$ for  the {\Be Wentzell} problem, then the shape derivative $u'={\left(\partial_{t}u_{t}\right)}_{|t=0}$ of the eigenfunction satisfies
\begin{align}\label{shape_derivative}
\Delta u' =&\; 0 \textrm{ in }\Omega, \nonumber\\
-\beta \Deltat u'+\partial_n{u'} - \lambda u' =&\; \beta \Deltat (V_n \partial_n u) -\beta \divet\big(V_n(2D^2b-H I_{d})\nablat u\big)\nonumber \\
&\:\:\;+ \divet(V_n\nablat u)  -\lambda'(0) u+\lambda  V_n( \partial_nu+ H u) \textrm{ on }\partial\Omega.
\end{align}
\end{theorem}}
\begin{proof}
The fact that $u'$ is harmonic inside the domain is trivial. To derive the boundary condition satisfied by $u'$, we use a test function $\phi_t$ defined on $\partial\Omega_{t}$ with $\partial_{n}\phi_{t}=0$ as used in the proof of Lemma \ref{transport:Dirichlet:2:Neumann} and \ref{transport:Laplace:Beltrami:cas:trace} in \cite{DK2012}.  We get the following weak formulation valid for all $t$ small enough:
$$
\displaystyle \int_{\partial\Omega_t}\beta \nablat u(t,x).\nablat \phi_t~d\sigma_t+\int_{\partial\Omega_t} \partial_{n_t}{}u(t,x)~ \phi_t~d\sigma_t-\lambda(t) \int_{\partial\Omega_t} u(t,x) \phi_t~d\sigma_t =0. 
$$
We take the derivative with respect to $t$ and get at $t=0$: 
\begin{equation*}
{\Rd \beta}\frac{d}{dt}\Big(\restriction{ \int_{\partial\Omega_t} \nablat u(t,x) .\nablat \phi_t~d\sigma_t \Big)}{t=0}  +\frac{d}{dt}\Big(\restriction{\int_{\partial\Omega_t} \partial_{n_t}{u(t,x)} \phi_t~d\sigma_t \Big)}{t=0}=\frac{d}{dt}\Big(\restriction{\lambda(t) \int_{\partial\Omega_t}  u(t,x)~\phi_t(x) ~d\sigma_t\Big)}{t=0}.
\end{equation*}
{From \cite{DesaintZolesio} and \cite{CaubetDambrineKateb}, we get 
\begin{equation*}\displaystyle\frac{d}{dt}\Big(\restriction{\displaystyle \int_{\partial\Omega_t} \nablat u(t,x) . \nablat \phi_t~d\sigma_t\Big)}{t=0}=  \int_{\partial\Omega} \Big( -\Deltat u' -\Deltat(V_n \partial_{n}u ) +\divet \left((2D^2b - H I_{d})\nablat u\right)\Big) \phi ~d\sigma.
\end{equation*}
}
After some lengthy but straightforward computations we also obtain 
\begin{equation*}
\cfrac{d}{dt}\restriction{\Big(\int_{\Omega_t} \partial_{n_t}{}u ~\phi_t ~d\sigma_t\Big)}{t=0}
=\displaystyle \int_{\partial\Omega} \partial_n{} u'{}~\phi~d\sigma - \int_{\partial\Omega} \nablat V_n.\nablat u{ } \phi~d\sigma+ \int_{\partial\Omega}V_n \Big(\partial_n u+H u \Big)~\phi~d\sigma
\end{equation*}
and 
\begin{align*}
\frac{d}{dt}\restriction{\Big(\displaystyle \int_{\partial\Omega_t} \lambda(t) u_t \phi_t~d\sigma_t\Big)}{t=0} =&
\displaystyle \lambda'(0)\int_{\partial\Omega} u ~\phi~d\sigma \\
&+ \lambda \int_{\partial\Omega} u' \phi ~d\sigma+ \lambda\int_{\partial\Omega} \partial_{n}{}u~ \phi~d\sigma + \lambda
\int_{\partial\Omega}H{}u \phi~d\sigma.
\end{align*}
To end the proof of this second point, it suffices to gather the relations.
\end{proof}

\subsection{Shape derivative of simple eigenvalues of the {\Be Wentzell}-Laplace  problem}

Let $\lambda$ be a simple  eigenvalue of the {\Be Wentzell}-Laplace equation (\ref{Steklov_Ventcel}) and let  $u$ be the corresponding normalized eigenfunction. We give in this subsection the explicit formula for  the shape derivative of the  eigenvalue of the {\Be Wentzell}-Laplace operator associated to \eqref{Steklov_Ventcel}. 

On $\Omega_t=(I+t\V)(\Omega)$ with $t$ small,  there is  a unique eigenvalue $\lambda(t)$ near $\lambda$ which is an analytic function with respect of the parameter $t$. The associated eigenfunction { $u_{t}(x)=u(t,x)$} is solution of the problem \eqref{Steklov_Ventcel}. The shape derivative denoted  $u'$ is the partial derivative $\partial_{t}u(t,x)$ evaluated at $t=0$ {\Rd and solves \eqref{shape_derivative}}.  Let us deduce the analytic expression of $\lambda'(0)$:

{\Rd \begin{theorem}
	\label{derivee:vp:simple}
If $(\lambda,u)$ is an eigenpair (with $u$ normalized) for  the {\Be Wentzell} problem with the additional assumption that $\lambda$ is simple then  
the  application $t \rightarrow \lambda(t) $ is  analytic and its derivative at $t=0$ is 
$$
\lambda'(0)=\int_{\partial\Omega}V_n\Big( \vert  \nablat u\vert ^2-\vert \partial_{n}{}u \vert ^2 -\lambda H \vert  u\vert ^2+\beta  (H~I_{d}-2D^2b)\nablat u.\nablat u
\Big)~d\sigma.
$$
\end{theorem}
}
\begin{proof}
We start with the result of Theorem \ref{th:u'}. Let us  multiply the two sides of  \eqref{shape_derivative} the boundary condition satisfied by $u'$ by the eigenfunction $u$ and integrate over the boundary $\partial\Omega$: 
\begin{eqnarray*}
&0=&\displaystyle\int_{\partial\Omega}  v'(-\beta \Deltat u +\partial_n{}u-\lambda u)~d\sigma+\int_{\partial\Omega} V_n \partial_n{}u(-\beta \Deltat u)~d\sigma \\[3pt]
&&+\displaystyle\int_{\partial\Omega}\beta V_n(H I_{d}-2D^2b)\nablat u.\nablat u~d\sigma +\displaystyle\int_{\partial\Omega} V_n |\nablat u\vert ^2-\lambda'(0) \int_{\partial\Omega}\vert  u\vert ^2\\&&-\lambda\int_{\partial\Omega}V_n\Big(  u\partial_n{u}  +H \vert  u\vert ^2\Big)~d\sigma.
\end{eqnarray*}
Using the boundary condition satisfied by the eigenfunction: $-\beta \Deltat u +\partial_n u-\lambda u=0$, it follows that 
\begin{eqnarray*}
0=\displaystyle\int_{\partial\Omega} V_n \partial_n{}u(\lambda u-\partial_n{}u)~d\sigma 
+\int_{\partial\Omega}\beta V_n(H I_{d}-2D^2b)\nablat u.\nablat u)~d\sigma\\[3pt]
+\displaystyle\int_{\partial\Omega} V_n \nablat \vert   u\vert ^2-\lambda'(0) \int_{\partial\Omega}\vert  u\vert ^2-\lambda\int_{\partial\Omega}V_n\Big(  u\partial_n{u}  +H \vert  u\vert ^2\Big)~d\sigma.
\end{eqnarray*}
and the normalization condition  $\displaystyle\int_{\partial\Omega} u^2~d\sigma=1 $ implies 
\begin{eqnarray*}
\lambda'(0)=-\displaystyle\int_{\partial\Omega} V_n \vert  \partial_n{}u\vert ^2~d\sigma 
+\int_{\partial\Omega}\beta V_n(H I_{d}-2D^2b)\nablat u.\nablat u~d\sigma\\[3pt]
+\displaystyle\int_{\partial\Omega} V_n \vert  \nablat  u\vert ^2-\lambda \int_{\partial\Omega}V_n ~ H \vert  u\vert ^2~d\sigma.
\end{eqnarray*}
\end{proof}

\subsection{Shape derivative of multiple  eigenvalues of the {\Be Wentzell}-Laplace  problem}

\subsubsection{The general result}

We suppose that $\lambda$ is an eigenvalue of multiplicity $m$. For smooth deformation $t\mapsto\Omega_t$, there will be  $m$ eigenvalues close to $\lambda$ (counting their multiplicities) for small values of $t$. We know that such a  multiple eigenvalue is no longer differentiable in the classical sense. We are then led to compute the directional derivative of $t\mapsto \lambda_{i}(t)$ at $t=0$ where $\lambda_{i}(t), {j=1,\ldots,m}$ are given by Theorem \ref{analyticity}.
This is the second part of Theorem \ref{Theoreme:gradient:cas} that we recall here:

\begin{theorem} 
\label{theorem:gradient:vp:multiple}
Let $\lambda$ be a multiple eigenvalue of order $m\ge 2$. {\Be Then each $t\mapsto \lambda_i(t)$ for $i\in\llbracket 1,d\rrbracket$ given by Theorem \ref{analyticity} has a derivative near 0, and the values of $(\lambda_i'(0))_{i\in\llbracket1,d\rrbracket}$ are the eigenvalues} of the matrix $M(V_{n})=(M_{jk})_{1\le j,k\le m}$ defined by 
\begin{equation}\label{eq:M}
M_{jk}=\displaystyle \int_{\partial\Omega} V_n\Big(\nablat u_j.\nablat u_k-\partial_n{}u_j \partial_n{}u_k -\lambda H u_ju_k+ \beta\left(HI_{d}-2D^2 b\right)\nablat u_j.\nablat u_k \Big)~d\sigma.
\end{equation}
\end{theorem}

\begin{proofof}{Theorem \ref{theorem:gradient:vp:multiple}}
Let  $t\mapsto(u(t,x),\lambda(t)=\lambda(\Omega_t))$ a smooth path of eigenpair of the Laplace-{\Be Wentzell} problem, so that it satisfies
\begin{equation*}
\left\{
\begin{array}{rlll}
\Delta u(t,x)&=&0&~\textrm{~~in~~}  \Omega_t\\
-\beta \Deltat u(t,x)+\partial_n{} u(t,x)&=&\lambda(t) u(t,x)&~\textrm{~~on~}~\partial\Omega_t.
\end{array}
\right.
\end{equation*}
 We have proved that  $u'=\partial_{t} u(0,x)$ is harmonic in $\Omega$ and  satisfies  the boundary condition (\ref{shape_derivative}) on $\partial\Omega$. 
We use the decomposition of $u=u(0,x)$ as 
 \begin{equation*}
 u=\sum_{j=1}^m c_j u_j
 \end{equation*}
for some $c=(c_1,c_2,\ldots,c_m)^T\neq 0$. Multiplying the two sides equation of  $ (\ref{shape_derivative}) $ by $u_k$, we get  after some integration by parts the eigenvalue equation 
$$
\lambda'(0)c= Mc
$$
where $M=(M_{jk})_{1\le i,j\le m}$ is defined by \eqref{eq:M}.
From this, we deduce that the set of derivatives $(\lambda_{i}'(0))_{i\in\llbracket 1,d\rrbracket}$ is exactly the set of eigenvalues of the matrix $M$, which achieves the proof of Theorem \ref{theorem:gradient:vp:multiple}.
\end{proofof}

\subsubsection{The case of balls} 

We consider now the case where the domain is a ball of  radius $R$. The problem is invariant under translation. In order to remove the invariance, we fix the center of mass of the  boundary of the domain,  as in Section \ref{cas:beta:positif}. 

{\Rd
The coordinates functions $x_i$ are  eigenfunctions of the {\Be Wentzell}-Laplace operator, so we get 
$$
\lambda =\cfrac{\beta(d-1)+R}{R^2}, \textrm{ and }u_i(x)=\cfrac{x_i}{\parallel x_i\parallel_{L^2(\partial B_{R})}}=\frac{x_{i}}{\sqrt{\omega_{d}R^{d+1}}}.$$ 

\begin{corollary}
	\label{cor:derivees:premieres:disque}
Let $\Omega=B_{R}$ be a ball of radius $R$, $\lambda_{1}$ its first non-trivial eigenvalue, which is of multiplicity $d$. The shape derivatives of the maps $t\mapsto\lambda_{i}(t)$, $i=1,\ldots, d$ given by Theorem \ref{analyticity} are the eigenvalues of the matrix $M_{B_{R}}(V_{n})=(M_{jk})_{j,k=1,\ldots,d}$ defined by
\begin{equation}\label{expression:Mij}
 M_{jk}=\frac{\delta_{jk}}{\omega_{d}R^{d+1}}\left(1+\beta\frac{d-3}{R}\right)\int_{\partial B_{R}}V_{n}-C(d,R)\displaystyle \int_{\partial B_{R}} V_n~x_{j}x_{k}~d\sigma
\end{equation}
where $C(d,R)=\frac{(d+1)(1+\beta\frac{d-2}{R})}{\omega_{d}R^{d+3}}$.
\end{corollary}

\begin{proofof}{Corollary \ref{cor:derivees:premieres:disque}}
We use \eqref{eq:M}. 
On one hand we check the geometric quantities:
$$H=\frac{d-1}{R}, \;\;\;\;D^2b(x)=\frac{1}{R}I_{d}-\frac{1}{R^3}(x_{i}x_{j})_{i,j}
$$
so since $\nablat u_{j}, \nablat u_{k}$ are in the tangent space of $\partial B_{R}$, we obtain that
$$(H I_{d}-2D^2b(x))\nablat u_{j}.\nablat u_{k}=\frac{d-3}{R}\nablat u_{j}.\nablat u_{k}
$$
and on the other hand:
$$
\partial_{n}u_{j}=\frac{x_{i}}{R\sqrt{\omega_{d}R^{d+1}}}\;\;\;\;\;\;\nablat u_j.\nablat u_k =\cfrac{1}{\omega_{d}R^{1+d}}\left(\delta_{jk}-\cfrac{x_{j}x_{k}}{R^{2}}\right)
$$ 
Therefore, the matrix $M=M_{B_{R}}$ has the following  entries 
\begin{eqnarray*}
M_{jk}&=&\displaystyle \frac{1}{\omega_{d}R^{d+1}}\int_{\partial B_{R}} V_n\Big[ \left(\delta_{jk}-\frac{x_{j}x_{k}}{R^2}\right)-\frac{x_{j}x_{k}}{R^2}-\lambda\frac{d-1}{R}x_{j}x_{k}+\beta\frac{d-3}{R} \left(\delta_{jk}-\frac{x_{j}x_{k}}{R^2}\right)\Big]~d\sigma\\[3pt]
&=&\frac{\delta_{jk}}{\omega_{d}R^{d+1}}\left(1+\beta\frac{d-3}{R}\right)\int_{\partial B_{R}}V_{n}-\left[\frac{d+1+\beta\frac{(d-1)^2+d-3}{R}}{\omega_{d}R^{d+3}}\right]\displaystyle \int_{\partial B_{R}} V_nx_{j}x_{k}~d\sigma.
\end{eqnarray*}
This leads to the result since $(d-1)^2+d-3=(d+1)(d-2)$.


\end{proofof}
From this formula, we deduce a first interesting result:
\begin{proposition}\label{cas-nullite-M}
If $V$ is a volume preserving deformation, then the following statements are equivalent:
\begin{description}
\item[(i)] $V_{n}$ is orthogonal (in $\sL^2(\partial B_{R})$) to homogeneous harmonic polynomials of degree $2$, 
\item[(ii)] $M_{B_{R}}(V_{n})=0$. 
\end{description}
\end{proposition}
\begin{proofof}{Proposition \ref{cas-nullite-M}}
We denote $\mathcal{H}_{2}$ the space of homogeneous harmonic polynomials of degree 2 (therefore we use here a slightly different notation than in Section 4). Let us suppose that $M(V_n)=0$; this means that $\displaystyle \int_{\partial B_{R}} V_n~x_{j} x_{k}~d\sigma=0,$  for all $ j,k=1,\ldots,d$, and in particular $V_{n}$ is orthogonal to $\mathcal{H}_{2}$.\\
If we assume now that $V_{n}$ is orthogonal to $\mathcal{H}_2$, using that
$$
\mathcal{H}_2=\text{span}\left\{  x_{j}x_{k},~j\neq k\in\{1,\ldots,d\}, \;\;x_{1}^2-x_{j}^2, j=2,\ldots,d\right\}.
$$
and moreover that
$\int_{\partial B_{R}}V_{n}=0$, we obtain
$$d\int_{\partial B_{R}}V_{n}x_{1}^2=\sum_{j=2}^d\int_{\partial B_{R}}V_{n}(x_{1}^2-x_{j}^2)+\int_{\partial B_{R}}\sum_{j=1}^dx_{j}^2=0,$$
and therefore
$$\int_{\partial B_{R}}V_{n}x_{j}^2=\int_{\partial B_{R}}V_{n}(x_{j}^2-x_{1}^2)=0,$$
which concludes the proof.
\par
\noindent
\end{proofof}
}
\par
\noindent
In the case where ${\Be M_{B_{R}}}(V_{n})\ne 0$, we compute the trace of the matrix ${\Be M_{B_{R}}}(V_{n})$ to obtain information on its eigenvalues.
 
\begin{proposition}
\label{trace:matrice:derivees:premieresbis}
When $\Omega$ is a ball of radius $R$, then 
\begin{equation}
\text{Tr}(M_{B_{R}}(V_n))=0
\end{equation}
for all volume preserving deformations.
\end{proposition}

\begin{proofof}{Proposition \ref{trace:matrice:derivees:premieresbis}} It comes that
{\Be $$
\text{Tr}(M_{B_{R}}(V_n))=-{\Rd C(d,R)}\ \displaystyle \int_{\partial B_{R}} \sum_{j=1}^d x_{j}^2 ~V_n ~d\sigma=-{\Rd C(d,R)}\ \sum_{j=1}^d x_{j}^2\ \displaystyle \int_{\partial B_{R}}  ~V_n ~d\sigma=0$$
since we are concerned  with deformations preserving the volume.}
\end{proofof}
As a consequence of Proposition \ref{cas-nullite-M} and Proposition \ref{trace:matrice:derivees:premieresbis}, there is the following alternative: either the only eigenvalue of $M(V_{n})$ is $0$, or $M(V_{n})$ has at least one nonnegative and one nonpositive eigenvalue. \Rd Each $t\mapsto \lambda_i(t)$ given by Theorem \ref{analyticity} has a directional derivative at $t=0$ denoted by $\lambda_{i}'(0)$. We then define, as usual \cite{Clarcke}, $ \partial \lambda_{1}$ the subgradient of $\lambda_{1}$ by $ \partial \lambda_{1}=[\inf_{i=1\cdots d} \lambda_{i}'(0), \sup_{i=1\cdots d} \lambda_{i}'(0)]$. With this notation, $0 \in \partial \lambda_{1}$ and we say the ball is a critical shape. 
\Bk 
\subsection{Numerical illustrations}

In order to illustrate Proposition \ref{trace:matrice:derivees:premieresbis}, we consider the two dimensional case and consider perturbations of the disk given in polar coordinates by 
$$\Rd \rho_{t}(\theta)=R+t f(\theta)  $$
where \Rd $f$ has zero mean value. \Bk
\begin{figure}[!ht]
\begin{center}
\subfigure[$f(\theta)=\cos(\theta)$]{\label{Fig5_1} 
\includegraphics[width=4.5cm]{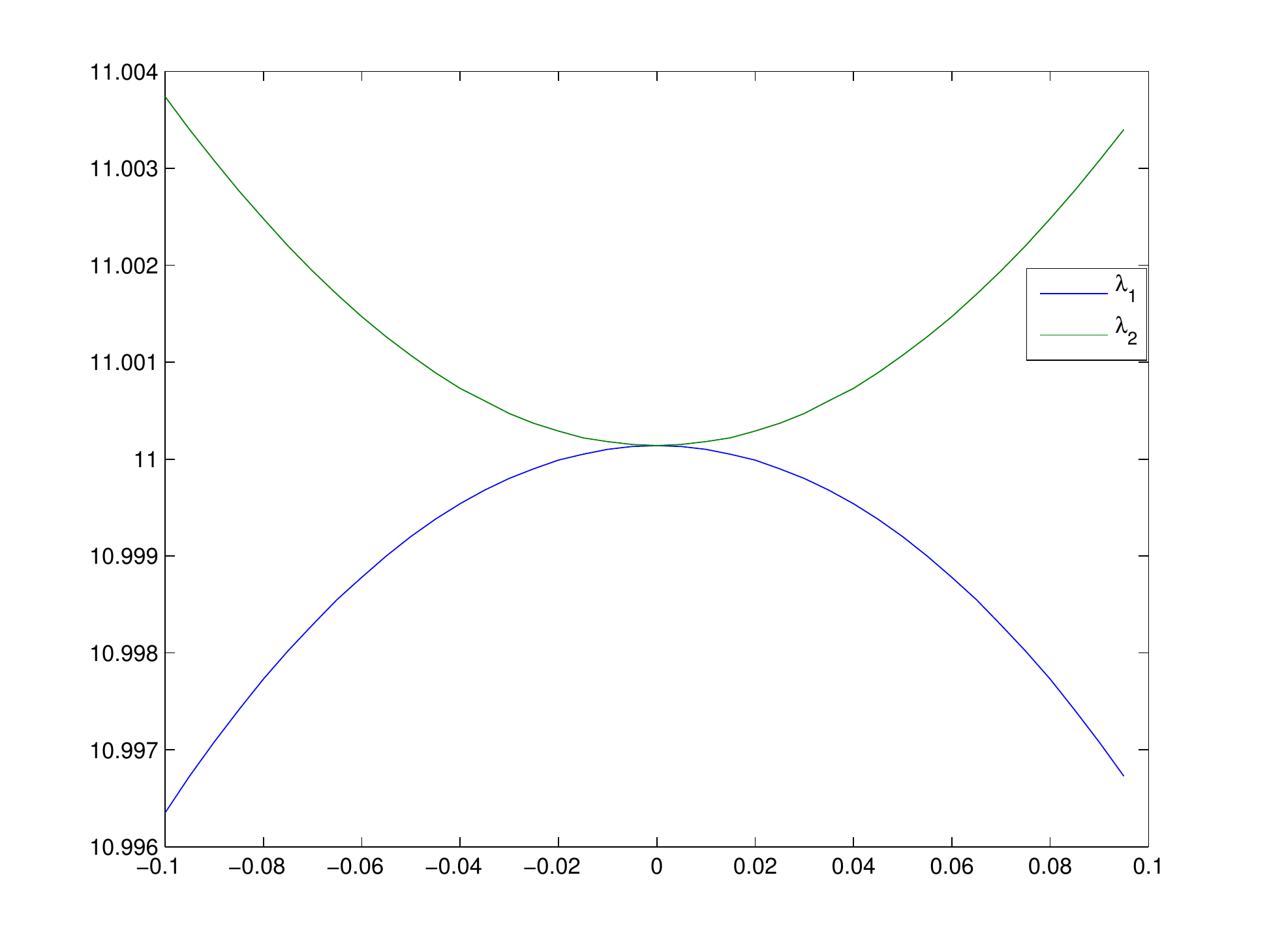}
}
\subfigure[$f(\theta)=\sin(\theta)$]{\label{Fig5_1bis} 
\includegraphics[width=4.5cm]{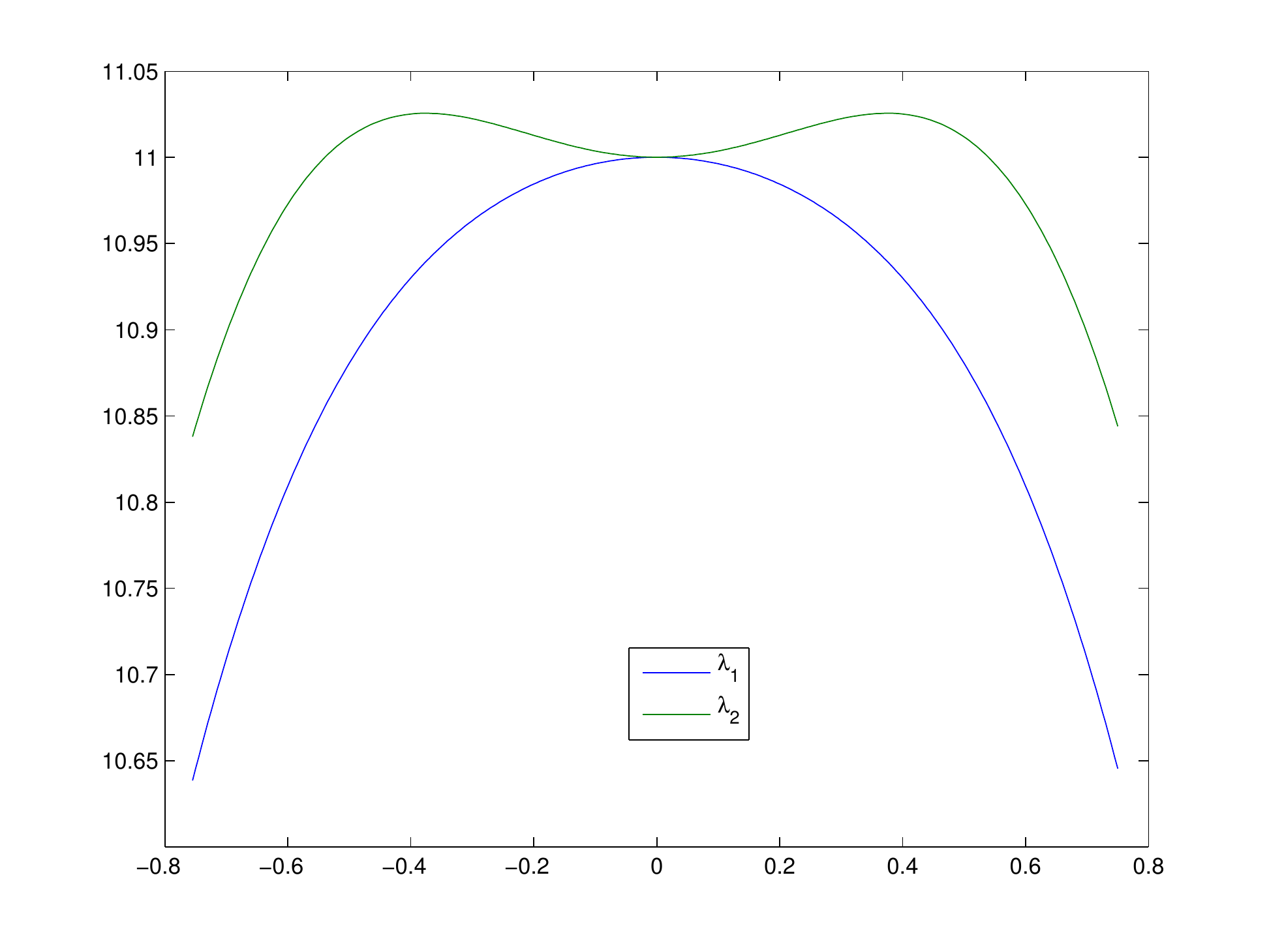}
}
\subfigure[$f(\theta)=\cos(2\theta)$]{\label{Fig5_2} 
\includegraphics[width=4.5cm]{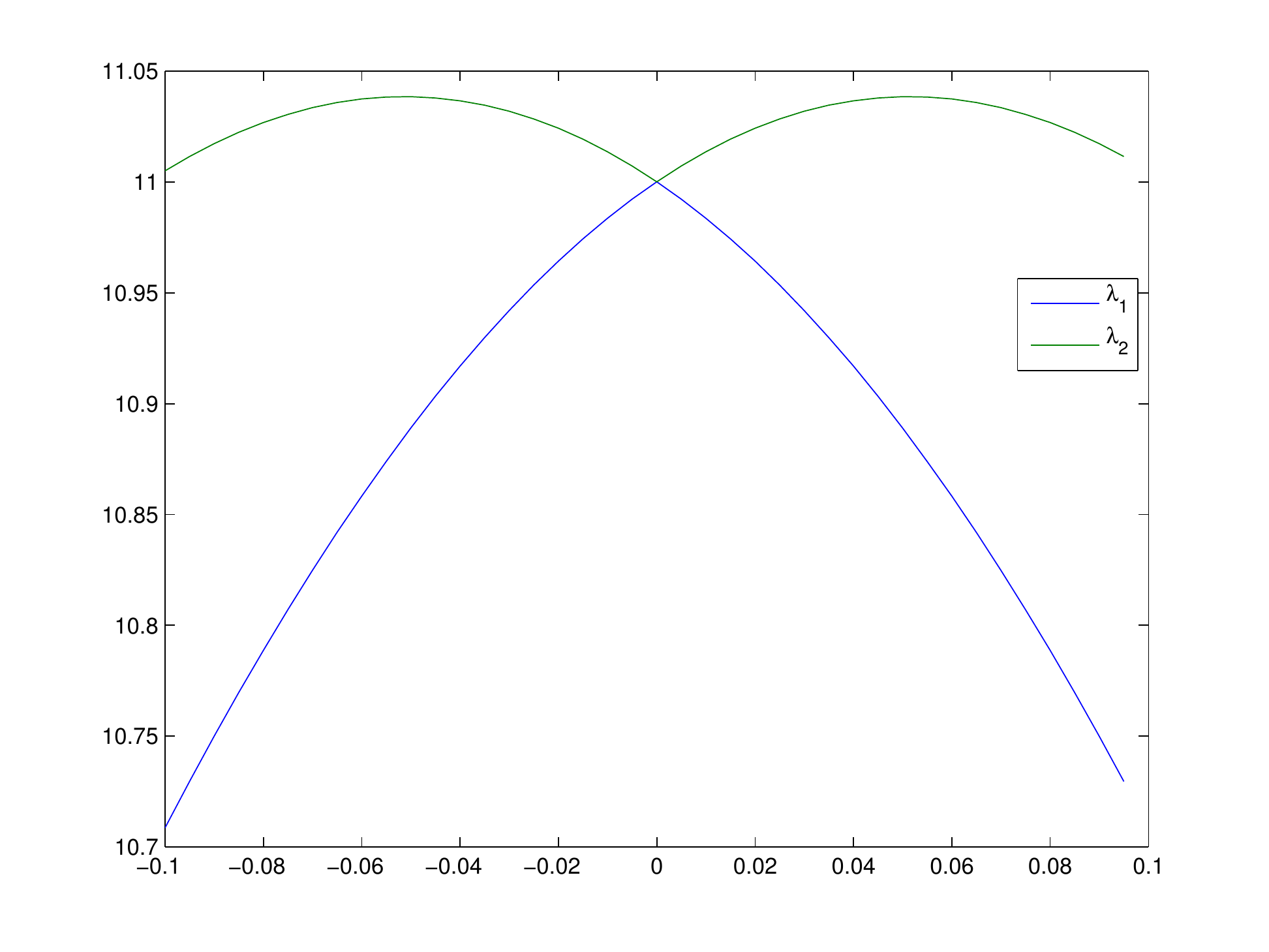}
}
\subfigure[$f(\theta)=\sin(2\theta)$]{\label{Fig5_2bis} 
\includegraphics[width=4.5cm]{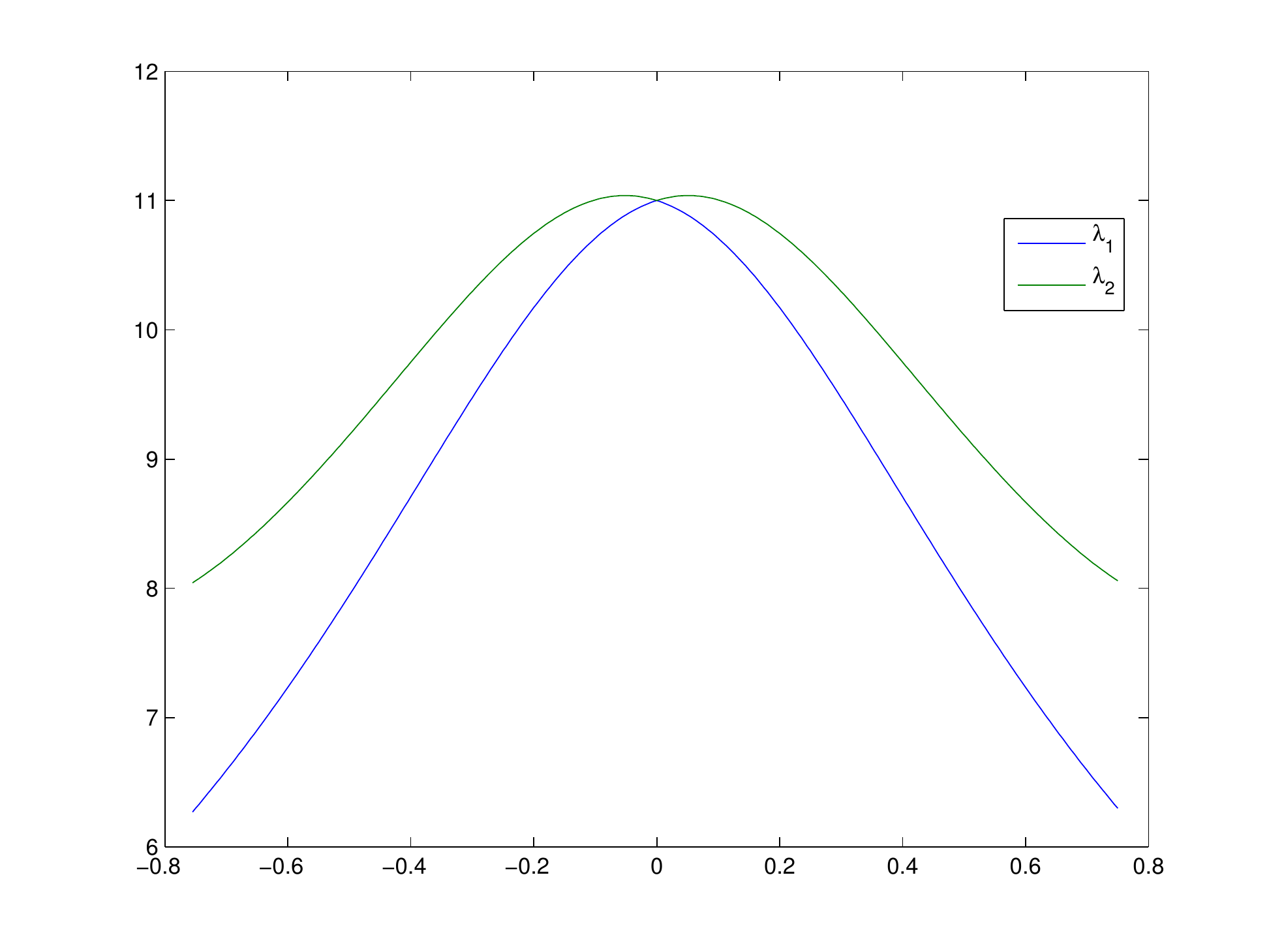}
}
\subfigure[$f(\theta)=\cos(3\theta)$]{\label{Fig5_3} 
\includegraphics[width=4.5cm]{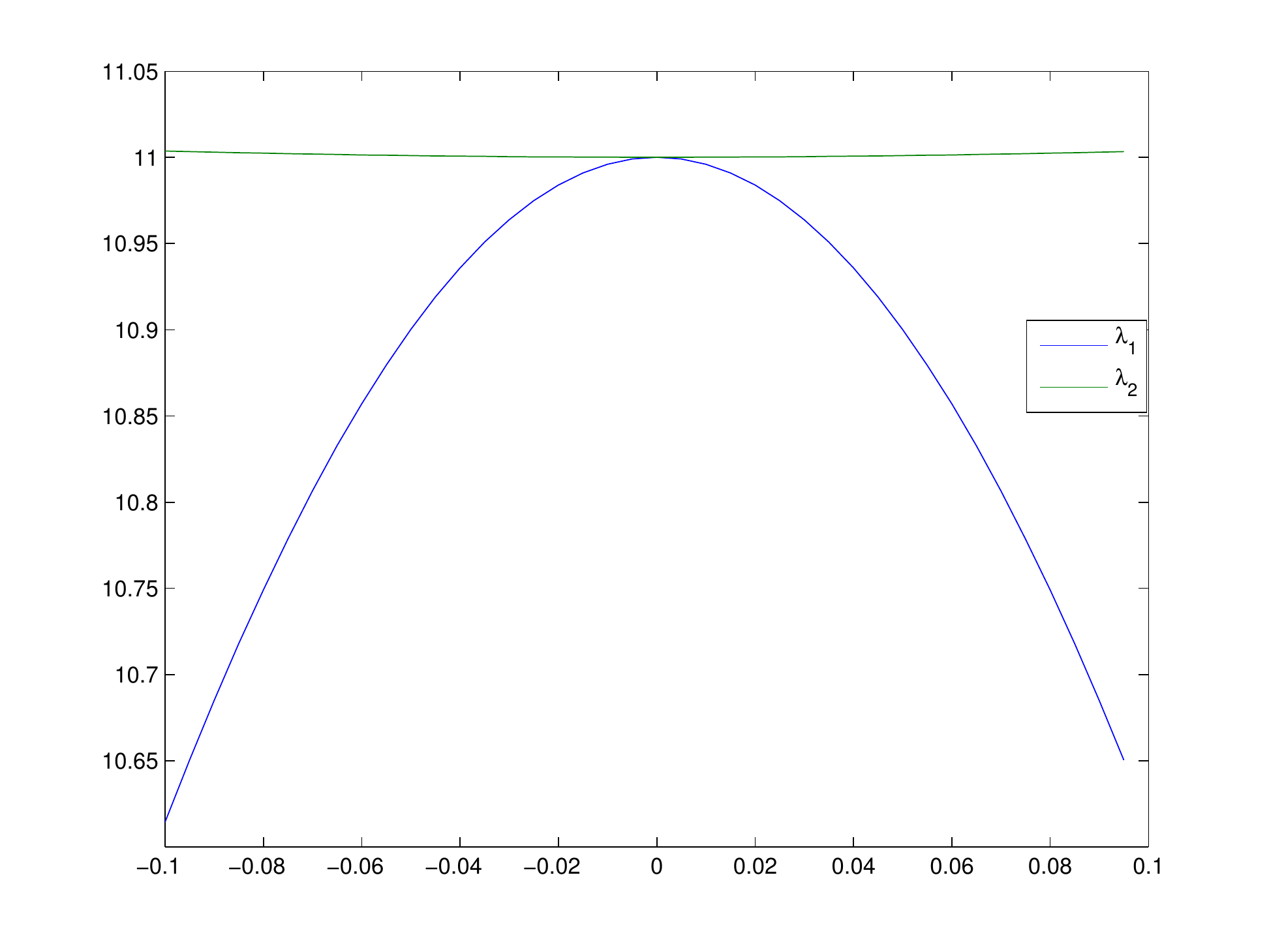}
}
\subfigure[$f(\theta)=\cos(4\theta)$]{\label{Fig5_4} 
\includegraphics[width=4.5cm]{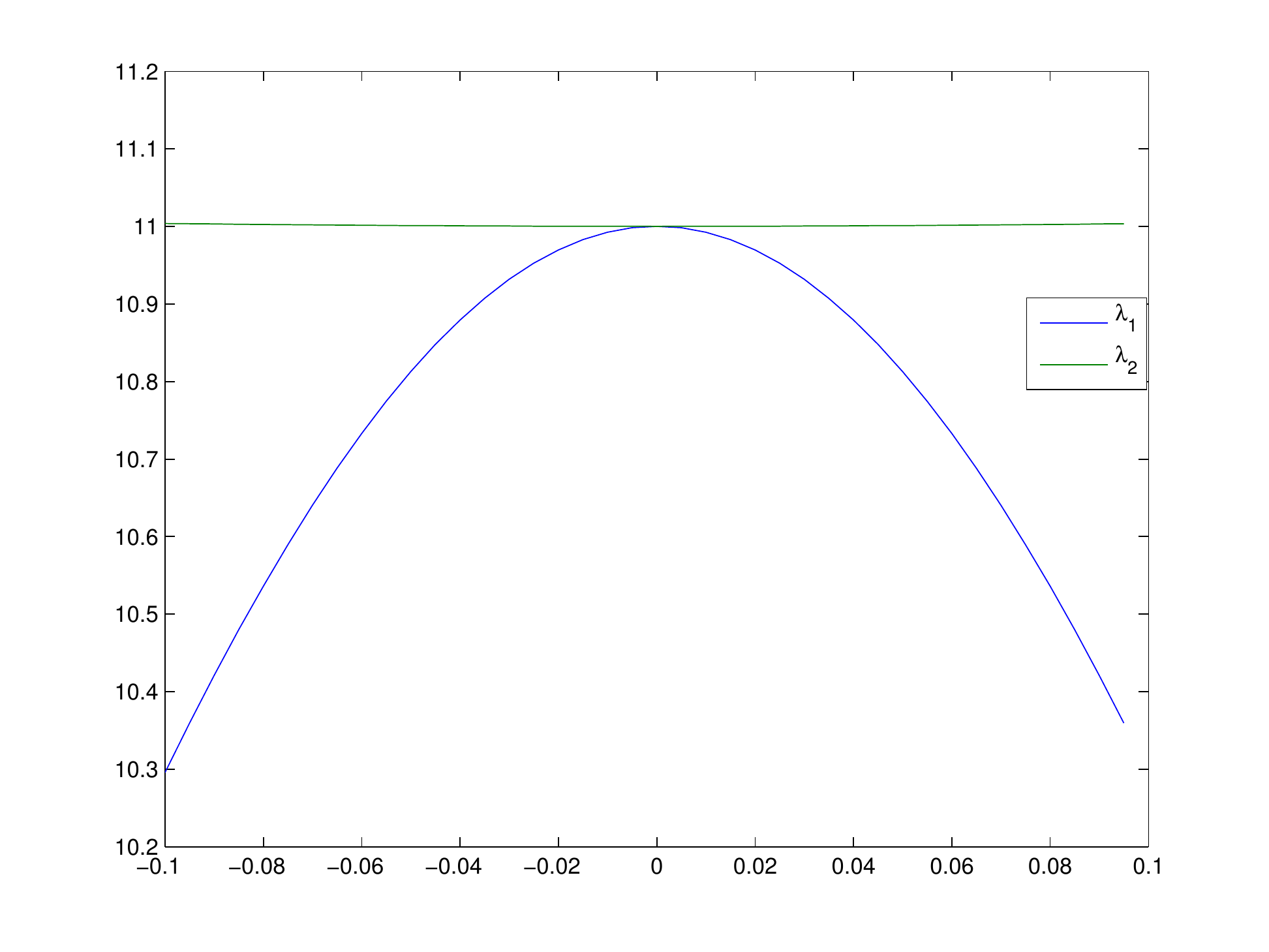}
}
\caption{$\lambda_{1}(\Omega)$ and $\lambda_{2}(\Omega)$ in the direction of $f(\theta)$ - $|B_{R}|=\pi$, $\beta=10$.}
	\label{Fig5}
\end{center}\end{figure}

In Figure \ref{Fig5}, the computations are made in the case $R=1$ and $\beta=10$, the deformation parameter $t$ appears in the abscissa.

In both collection of figures, we can see the derivatives of the second and third eigenvalues vanish at the ball in every case except when $f(\theta)=\cos(2\theta)$,  where the regular lines cross, leading to a really non differentiable second eigenvalue. This is coherent with Proposition \ref{cas-nullite-M}. 
 Let us explicit the case $V_n= R^2 \cos{2\theta}$, where we are led to compute the eigenvalues of the following  symmetric matrix  
$$
M=-\displaystyle\frac{3}{\pi R} \left(
\begin{array}{clr}
&\displaystyle\int_0^{2\pi} \cos{2 \theta }\cos^2{\theta} ~d\theta  &0\\
&0 &\displaystyle \int_0^{2\pi} \cos{2\theta}~\sin^2{\theta}~d\theta
\end{array}
\right)
$$
whose eigenvalues are $\alpha_1=-\displaystyle\frac{3}{2 R}$ and $\alpha_2=\displaystyle\frac{3}{2 R}$.

\section{Testing if the ball is a local maximum for $\lambda_{1}$: second order arguments}
\label{section:analyze:ordre2}

We know that any ball is a critical point for volume preserving deformations. Therefore, if the subgradient $\partial\lambda_{1}(B;V_{n}) \ne	\{0\}$,  then the ball is a local maximizer. It remains to deal with the case where all the eigenvalues of $M_{B}(V_{n})$ are $0$; this case corresponds to $V_{n}$ orthogonal to the harmonics of order two. Then, we aim at proving that the second derivative of $\lambda_{1}$ along at least one of the smooth branches is nonpositive. 

The necessary order two conditions of optimality are:  the second derivative \Rd of the Lagrangian \Bk should be non positive on the  subspace  orthogonal to the space generated by the gradient of the volume constraint.
We compute:
\begin{equation}
\label{gradient:contraintes}
\text{Vol}'(0)=\int_{\partial B_{R}} V_{n} 
\end{equation}
Hence $\text{Vol}'(0)=0$ if and only if $V_{n}\in (\mathcal{H}_{0})^{\perp}$  
where $\mathcal{H}_{k}$ denotes the linear space of spherical harmonics of order $k$.
\Bk
Due to the previous remarks, we hence consider deformation field in the hilbertian space $\mathcal{H}$ spanned by all the spherical harmonics of order \Be$l\in I=\N\setminus\{0,2\}$ \Bk. \Rd The normal component of such a field is orthogonal to spherical harmonics of order $0$ and $2$.\Bk

The goal of this section is to present the different steps for the computations. We will characterize the matrix $E$ whose eigenvalues are the second order derivatives of the smooth branches of eigenvalues.  It turns out that this computation is hard even in the case of a ball. Nevertheless, the computation of $\Tr(E)$ is much simpler than the individual computations of the entries. In order to prove that the ball is a local maximum of $\lambda_{1}$, it suffices to  prove that its trace is nonpositive: therefore at least one smooth branch of eigenvalues has a nonpositive second order derivative.  

\Rd In this section, we consider deformations preserving the volume at second order and not only at first order. Hence, we cannot consider deformation $T_{t}$ of type $I+tV$ with $V$ independent of $t$ and introduce deformations 
$S_{t}$ that are the flow at time $t$ of a vector field $V$ (see also Remark \ref{rk:volume}). Notice that $S_{t}=I+tV+{o}(t)$ so that  $T_{t} - S_{t}={o}(t)$ and first order shape derivatives are unchanged. In particular, one has
$$ \cfrac{d^2}{dt^2} \text{Vol}(S_{t}(\Omega))= \int_{\partial \Omega(t)} \left( \cfrac{\partial }{\partial t} (V_{n(t)})+ V_{n(t)} \cfrac{\partial }{\partial n(t)}(V_{n(t)})+ H V_{n(t)}^2 \right) \ d\sigma$$
and the volume preservation at second order means that  
\begin{equation}\label{volume:preserving:second:order}
\left(\cfrac{d^2}{dt^2} \text{Vol}(S_{t}(\Om))\right)_{|t=0}= \int_{\partial \Omega} \left( \cfrac{\partial }{\partial t} (V_{n(t)})+ V_{n(t)} \cfrac{\partial }{\partial n(t)}(V_{n(t)})+ H V_{n(t)}^2 \right)_{|t=0} \ d\sigma=0.
\end{equation}

\Bk
\subsection{Construction of the matrix $E$ of the second derivatives}\label{ssect:E''}
Let  $(u(t,x),\lambda(t)=\lambda(\Omega_t))$ be  an eigenpair of the Laplace-{\Be Wentzell} problem, that is to say solves
\begin{equation*}
\left\{
\begin{array}{rlll}
\Delta u(t,x)&=&0&~\textrm{~~in~~}  \Omega_t\\
-\beta \Deltat u(t,x)+\partial_n{} u(t,x)&=&\lambda(t) u(t,x)&~\textrm{~~on~}~\partial\Omega_t
\end{array}
\right.
\end{equation*} 
We use the decomposition of $u=u(0,x)$ in the basis of eigenfunctions:
 \begin{equation*}
 u=\sum_{j=1}^d c_j u_j
 \end{equation*}
for some $c_1,c_2,\ldots,c_d$ not all zero. We have shown that the vector   $c=(c_1,c_2,\ldots,c_d)^T$  is solution of 
$$
\lambda'(0)c= M(V_n)c
$$
where the matrix $M(V_n)=(M_{jk})_{1\le i,j\le d}$ is defined by \eqref{eq:M}. 
\par
\noindent
To compute the second derivative at $t=0$, one has  to compute the first shape derivative $u'(x)=u'(0,x)$.  Fredholm's alternative insures the existence of a unique harmonic function $\tilde  u_j$ orthogonal to the 
eigenfunctions $u_1,u_2,\ldots,u_d$ and satisfying on $\partial \Om $ the boundary condition 
\begin{eqnarray}
-\beta \Deltat  \tilde{u}_j+\partial_n \tilde{u}_j-\lambda \tilde{u}_j&=&\beta\Big[\Deltat[V_n\partial_n  u_j]+\divet[V_n(H I_{d}-2D^2b)\cdot\nablat u_j]\Big]
\nonumber \\
&&+\divet[V_n\nablat  u_j]+\lambda' u_j+\lambda V_n(\partial_n u_j+H u_j).
\end{eqnarray}
It follows that 
\begin{equation}\label{decomposition}
u'=\sum_{j=1}^d \tilde{c}_j  u_j+\sum_{j=1}^d c_j  \tilde{u}_j
\end{equation}
for some $c_j,\tilde{c}_j$ when $ j=1,\ldots,d$.   We point out that the $(c_j)$ are the same coefficients as the decomposition of  $u$ in the basis $( u_j)$. 
\par
\noindent
The strategy is straightforward : we have to consider the equation satisfied by $u'$ on the boundary $\partial \Omega$ and take its shape derivative again. A first look to the second derivative shows that we will encounter three operators : 
\begin{itemize} 
\item 
the first contains only $u''$ and its expression is  the following
$$
E^{(0)}=  -\beta \Delta u''+\partial_n u''-\lambda u''
$$
\item concerning the term in $u'$ and $\lambda'=0$ we have  
$$
\begin{array}{lll}
E^{(1)}&=&-2\beta \Deltat (V_n\partial_n u')-2\divet(  V_n(I+\beta \mathcal{A}) \nablat u')\\[5pt]
&& -2\Big[\lambda'  u'+\lambda V_n ( \partial_nu'+Hu') \Big]
\end{array}
$$
where $\mathcal{A}=H I -2D^2 b$ is the deviatoric part of the curvature tensor. 
\item The remaining  term is $E^{(2)} $  contains only $u$; we give a more explicit expression below.
\end{itemize}
Green-Riemann identity tells us that   $\langle E^{(0)},u_i\rangle=\langle u'',  -\beta \Deltat u_i+\partial_n u_i -\lambda u_i\rangle=0,~i=1,\ldots,d$. This means that  the term $E^{(0)}$ will have no influence in the determination of the second derivative of the eigenvalue.  We will  focus  only on $E^{(1)}$ and $E^{(2)}$. 
\\

\noindent
{\bf Construction of $E^{(2)}$}:

The computations are very technical.   We need first  to  use  a test function $\phi$ which is the restriction of  a test function $\Phi$ defined on a tubular neighborhood of the boundary such that its normal derivative on $\partial\Omega$  is zero. This kind of extension is well discussed in the book \cite{DelfourZolesio} of Delfour-Zolesio.   Taking the shape derivative  of the boundary condition \eqref{shape_derivative} (in the multiple case) we need to compute 
\begin{eqnarray*}
\restriction{\left( \frac{d}{dt}\int_{\partial \Omega_t} ~V_n{}\nablat u .\nablat \phi ~d\sigma_t\right)}{t=0}&=&\langle A^{(1)}u',\phi\rangle+ \langle A^{(2)}u,\phi\rangle,\\
\beta  \restriction{\left(\frac{d}{dt}  \int_{\partial \Omega_t}\mathcal{A}(t)V_n \nablat u. \nablat \phi~d\sigma_t\right)}{t=0}&=&\langle B^{(1)}u',\phi\rangle+ \langle B^{(2)}u,\phi\rangle,\\
-\displaystyle\frac{d}{dt}\restriction{\left( \displaystyle\int_{\partial \Omega_t} ~\Big[\lambda' u+ \lambda(u'+V_n \partial_n u +V_n Hu)\Big]\phi~d\sigma_{t} \right)}{t=0}&=&\langle C^{(0)}u'',\phi\rangle+ \langle C^{(1)}u',\phi\rangle+\langle C^{(2)}u,\phi\rangle,\\
\beta \restriction{\left( \frac{d}{dt}\int_{\partial \Omega_t} ~\nablat(V_n{}\nablat \partial_n u) .\nablat \phi d\sigma_t\right)}{t=0}&=&\langle D^{(1)}u',\phi\rangle+ \langle D^{(2)}u,\phi\rangle.\\
\end{eqnarray*}

The remaining  $E^{(2)} $ containing only $u$  is then given by 
$$
E^{(2)} = A^{(2)}u+B^{(2)}u+C^{(2)}u+D^{(2)}u.
$$

For an operator $L$ involved in $E^{(i)},~i=1,2,3$ we denote by  $(L_{ij})_{i,j=1,\ldots,d}$ the matrix of $L$ in the basis of the eigenvalues. After calculations (see also Remark \ref{remarque:simplification:du:noyau} in the Appendix), we get the following linear equation 
$$
(\lambda'' I-E)c + 2(-M(V_{n})+\lambda' I) \tilde c=0
$$
(corresponding to the second derivation) together with
$$
(-M(V_{n})+\lambda' I) c=0.
$$
(corresponding to the first derivation) where the matrix $E=(E_{ij})$ is split  into $E=E^{(1)}+E^{(2)}$ where the terms involving $u'$ are gathered in $E^{(1)}$ and the terms involving $u$ are gathered in $E^{(2)}$.
\subsection{Computation of the trace }
Since the direct computations of the eigenvalues are difficult, we restrict ourselves to the cases $d=2$ or $d=3$, and we will focus on the trace of $E$ and prove that $\Tr(E)$ is nonpositive. We start with the trace of $E^{(2)}$:
\begin{lemma}\label{lem:E2}
Assume $d\in\{2,3\}$. With $K(R)=\frac{d}{R^{2+d}\omega_{d\Be -\Bk 1}}$, we have 
\begin{equation}\label{trace:E2}
\Tr{(E^{(2)})} = -(d\beta+R) R K(R) \int_{\partial B_{R}} |\nablat V_{n}|^2 d \sigma - K(R) \int_{\partial B_{R}} V_{n}^2 d\sigma. 
\end{equation} 
for all deformations preserving volume and such that   $V_n$ is  orthogonal to spherical harmonics of order two.
\end{lemma}
\begin{proof}
The computation of $E^{(2)}$ is done in the Appendix \ref{app:E2}, and to obtain the result, we sum all the traces given by Lemmas \ref{tracematriced}, \ref{trace:B},  \ref{Trace:A} and \ref{trace:C}.
\end{proof}  
\par
\noindent
Concerning $\Tr{(E^{(1)})}$, we start with the following Lemma which is  straightforward (see also Remark \ref{remarque:simplification:du:noyau}):
\begin{lemma}
We have
\begin{equation}\label{trace:E1}
\Tr{(E^{(1)})}=2\displaystyle \int_{\partial\Omega} V_n\sum_{j=1}^d\Big(-\partial_n \tilde u_j \partial_n u_j -H \lambda \tilde u_j u_j+ (I+\beta\left(HI_{d}-2D^2 b)\right)\nablat \tilde u_j.\nablat u_j \Big)~d\sigma.
\end{equation}
holds for all deformations preserving volumes such that   $V_n$ is  orthogonal to spherical harmonics of order two.
\end{lemma}
From this result we deduce the following, which is proved in Appendix \ref{app:u'}:
\begin{proposition}\label{trace:E1:dimension:3}
Assume $d=3$ and set $\alpha=\cfrac{\beta}{R}$. We denote $Y^{m}_l,×m=-l,\ldots,m$  any spherical harmonic of order $l\Be \in I \Bk$. If 
$$V_{n}=\sum_{\Be l\in I \Bk} R^l \Big(\sum_{m=-l}^l v_{l,m}Y_{l}^m\Big),$$ then 
$$\Tr(E^{(1)})=\ - \ \Be K(R) \Bk \left( \ \sum_{\Be l\in I \Bk} \left[A_{l,\alpha}+ B_{l,\alpha}\right] \ R^{2l\Be +1 \Bk}\  \sum_{m=-l}^l |v_{l,m}|^2\right)$$
where
$$
A_{l,\alpha}\ =\ \frac{l}{2l+1}\ \frac{l+2}{l-2}\ (4\alpha+2l)\ \frac{1+\alpha(3-l)}{1+\alpha(l+1)} 
\textrm{ and }
B_{l,\alpha}\ =\ \frac{l+1}{l}\ \frac{l-1}{l}\ (4\alpha+2)\ \frac{1+\alpha(4+l)}{1+\alpha(3+l)}.$$
\end{proposition}

Since $\Tr{E}=\Tr{(E^{(1)}})+ \Tr{(E^{(2)}})$, we will then deduce the following result 
\begin{proposition}\label{trace_E}
Assume $d\in\{2,3\}$. Then there exists a  nonnegative constant $\mu$ such that  $$
  \Tr(E)\leq -K(R) \mu \int_{\partial B_{R}} |\nablat V_{n}|^2+|V_{n}|^2 ~d\sigma.
  $$
holds for  all preserving volume deformations such that $V_n$ is orthogonal {\Be to $\mathcal{H}$}.
\end{proposition}

\begin{proof} 
We distinguish the case $d=2$ and $d=3$.
\paragraph{The case  $d=2$.}
Let us compute the  trace of the matrix $E$. 
Gathering all the results of Lemma \ref{lem:E2} with the computations of Appendix \ref{app:u'} concerning the trace of the different matrices involved in the matrix $E$, we obtain the following formula: when  $$V_n=\sum_{\Be l \in I \Bk} \Be \cfrac{R^l}{\sqrt{\pi}}\Bk\left(v_1^{(l)}\cos{l \theta}+v_2^{(l)}\sin{l \theta}\right),~ \Be l \in I \Bk,$$ we have 
\begin{equation}\label{eq:trE2d}
\Tr(E)=-K(R)\sum_{\Be l \in I \Bk}G(\alpha,l) \ (l^2+1)\   \Be R\Bk^{2l+1} \big((v_1^{(l)})^2+(v_1^{(l)})^2\big) .
\end{equation}
where 
$$G(\alpha,l)=\cfrac{(l^2-1)}{2(1+l^2)} \ \cfrac{2 + l^2 + 2 \alpha^2 (l-2) l^2 + 
   \alpha (l-2)(l^2+2))} {(l-2 ) l (1 + \alpha l)}.$$
\Rd Let us remark that $G(\alpha,1)=0$. This could have been guessed since the Wentzell eigenvalues are translation invariance: we recall that, denoting $\text{Bar}$ the center of mass of the boundary, we have 
$$\text{Bar}'(0)=\int_{\partial B_{R}} x V_{n}$$
so that deformations orthogonal to spherical harmonics of order $1$ preserve at first order the center of mass. \Bk    
A close look to the fraction $G$ shows that it has no pole for $\alpha>0$ and $l\geq 3$, that it is nonnegative for $l>2$ and that $G(l,\alpha) 	\rightarrow 1$ when $l \rightarrow +\infty$; then there is a nonnegative constant $\mu$ such that for all $l\geq 3$, $\mu\leq G(l,\alpha)$. This gives 
  $$
  \Tr(E)\leq -K(R) \mu\int_{\partial B_{R}} |\nablat V_{n}|^2+|V_{n}|^2 ~d\sigma.
  $$

\paragraph{The case $d=3$.}
The strategy is the same, and we use again Lemma \ref{lem:E2} and the detailed computations from Section \ref{app:d3}: we get for $\Be l \in I \Bk$:
$$V_{n}=\sum_{\Be l \in I \Bk} \Be R\Bk^l \sum_{p=- l}^{\Be l\Bk} v_{p}^{(l)} Y_{l}^p,$$
$$
\Tr(E)=- K(R) \sum_{\Be l \in I \Bk}\ F(\alpha,l)\   (l(l+1)+1) \Be R\Bk^{2l+1} \sum_{p=-l}^{l} (v_p^{(l)})^2.
$$
where $F(\alpha,l)$ is the fraction 
$$F(\alpha,l)=\cfrac{(l-1) \displaystyle\sum_{m=0}^3 P_{m}(l)\alpha^m} {(l(l+1)+1) \ l\ (1+\alpha(l+1)) \ (2l+1)\ (l-2) \ (1+\alpha(l+3)) },$$
and where the polynomial $P_{m}$ are defined as
\begin{eqnarray*}
P_{0} (X)&=& 2X^4+5X^3+16X^2-8,\\
P_{1}(X) &=& 4 X^5+18X^4+40X^3+68X^2-28X-56,\\
P_{2} (X)&=& 2X^6+21X^5+42X^4+35X^3+16X-112,\\
P_{3}(X) & = & 8X^6+18X^5+24X^4-68X^3-144X^2-112X-64. 
\end{eqnarray*}
\Be Let us remark that $F(\alpha,1)=0$ for the same reason than in dimension two. \Bk
By Descartes's rule of signs, the polynomials  $P_{m}$ have at most one positive root. Since $P_{m}(0)<0$ and $P_{m}(2)>0$ for $m = 0,\dots 3$, $P_{m}$ has exactly one positive root which is in $[0,2]$. Since $l>2$,  there exists a nonnegative constant $\mu$ such that for all $k\geq 3$, $\mu\leq F(k,\alpha)$ and 
  $$
  \Tr(E)\leq - K(R)\  \mu \ \int_{\partial B_{R}} |\nablat V_{n}|^2+|V_{n}|^2 ~d\sigma.
  $$
\end{proof}
\appendix 

\section{Some classical results on tangential differential calculus}\label{app:diffcal}
We recall some facts about tangential operators acting on functions defined on $\partial \Omega$.  The formulas involve the extensions of  functions and the differential calculus becomes easier since we will use the classical euclidean differential calculus in a neighborhood of  $\partial \Omega$.  The canonical extension will be provided thanks to the oriented  distance  and the orthogonal projection on the tangent plane.  For more details, the interested reader will consult the book \cite{DelfourZolesio} of M. Delfour and J.P. Zolesio from which we borrowed the necessary material.
\subsection{Notations and definitions. Preliminary results}
We recall some essential notations and definitions that are needed for the computations of  shape derivatives. Given a smooth function $f:\partial\Omega\mapsto \mathbb{R}$, we define its tangential gradient $\nablat$ as 
\begin{equation}
\nablat{}f=\nabla \tilde{f}-\nabla \tilde {f}.\textbf{n}~\textbf{n}
\end{equation}
where $\tilde{f}$ is any extension of $f$ in a tubular neighborhood of $\partial\Omega$. An  extension is easily obtained when $\partial\Omega$ is smooth. The tangential gradient does not depends on the extension. 
\par
\noindent
It is also useful to define the tangential gradient as the normal projection of $\nabla \tilde f$ to the tangent hyperplane of $\partial\Om$; in other words
$$
\nablat f=\nabla \tilde  f -n \otimes n \nabla \tilde f, ~\text {on } \partial\Om.
$$
We also need the definition of the tangential divergence : for a tensor $v$, we define the surface divergence as 
$$
\divet u=\Tr(\nablat u)
$$
For regular functions we define the surface Laplacian or Laplace-Beltrami operator as
$$
\Delta_\tau f := \divet (\nablat f).
$$
We recall the definition of the oriented distance $b_{\partial\Omega}$: 
\begin{equation}\label{eq:signeddistance}
b_{\partial\Omega}(x)=\left\{
\begin{array}{lll}
{}d_\Omega(x)& & \textrm{~for~} x \in \mathbb{R}^d\backslash\overline{\Omega}\\
-d_\Omega(x)& & \textrm{~for~} x \in \Omega,
\end{array}
\right.
\end{equation}
where the notation $d_\Omega$ stands for the  distance function for a subset $\Omega \subset \mathbb{R}^d$:
$$
d_\Omega(x)=\textrm{inf}_{y\in \Omega}\vert x-y\vert
$$
We shall sometimes write $b$ instead of $b_{\partial\Omega}$; its gradient is an extension of the normal vector field $\textbf{n}$ in a neighborhood of $\partial\Omega$. 
\par
\noindent
Let $D^2 b$ be the Weingarten operator with entries $(\nablat)_i~n_j$ where $n_j$ is the $j-$th component of $\textbf{n}$.  The normal {\Rd vector }is known to be in the kernel of $D^2b$, {\Rd while the other eigenfunctions are tangential with the corresponding eigenvalues given by the principal curvatures of $\partial\Om$}. \par
\noindent
Let $\kappa_i,~i=1,\ldots,d-1$ be the non zero eigenvalues of $D^2b$.  We define the mean curvature $H$  as 
\begin{equation}
H=\sum_{i=1}^{d-1}\kappa_i=\Tr(D^2 b)=\Delta b, \textrm{~on~} \partial\Om.
\end{equation}
\par
\noindent
An important result about the normal derivative of this quantities is: 
\begin{proposition}
Suppose that  the boundary $\partial\Omega$ is of class $C^3$. Then the normal derivative of the mean curvature $H$ is 
\begin{equation}
\partial_n H=-\sum_{i=1}^{d-1} \kappa_i^2.
\end{equation}
\end{proposition}
Other  known  identities: we denote $\mathbf{x}$ the identity function. We have 
$$
\begin{array}{lll}
-\Delta_\tau \mathbf{x}&=&H \mathbf{n}\\
\divet \mathbf{n}&=&H  \mathbf{n}
\end{array}
$$
\noindent{\bf Tangential integral formula:}
Given two functions $f$ (scalar) and $\mathbf{v}$ smooth enough, we have 
$$
\int_{\partial\Om} f \divet \mathbf{v} + \int_{\partial\Om} \nablat f.  v=\int_{\partial\Om}  H f~ \mathbf{v}.{\mathbf{n}}
$$
\noindent{\bf Shape derivative of the main curvature $H$ and of the normal $\bf{n}$} in the direction of a  velocity $\V$: 
\begin{proposition}\label{derivees:normale:H}
Let a  surface $\partial\Omega$ be of class $C^2$. The shape derivatives of the normal $\textbf{n}$ and of the mean curvature $H$ in the direction of the velocity vector $\V$ are 
\begin{equation}
\begin{array}{lll}
\textbf{n}'&=&-\nablat V_n\\
H'&=&-\Deltat V_n
\end{array}
\end{equation}
where $V_n=\langle \V,\n\rangle $ denotes the normal component of the vector deformation $\bf{V}$. 
\end{proposition}

\subsection{A commutation lemma}
Here   $f$ and $g$ are two smooth functions defined on $\mathcal{U}$ a neighborhood of $\partial\Omega$;  the notation $b$ stands for the oriented distance. Recall that its gradient is an extension of the normal field $\mathbf{n}$ on $\partial\Omega$.

\begin{proposition}
\label{commutateur}
We have 
\begin{equation}
\displaystyle \partial_\mathbf{n} (\nablat f.\nablat g)+  2(D^2b \nablat f). \nabla g= \nablat (\partial_\mathbf{n}f).\nablat g
+ \nablat (\partial_\mathbf{n}g).\nablat f
\end{equation}
\end{proposition}
\begin{proof}  A straightforward \Rd computation \Bk gives
$$
 \partial_\mathbf{n} (\nabla f.\nabla  g)=( D^2 f\nabla g).\mathbf{n}+(D^2 g \nabla f).\mathbf{n}
$$
and 
$$
\begin{array}{lll}
\nabla (\partial_\mathbf{n}f).\nabla  g
&=&\nabla (\nabla f.\mathbf{n}).\nabla g\\
&=&(D^2f \mathbf{n}).\nabla g +(D^2b \nabla f).\nabla g
\end{array}
$$
hence 
$$
\begin{array}{lll}
\nabla  (\partial_\mathbf{n}f).\nabla  g
+ \nabla (\partial_\mathbf{n}g).\nabla f &=&2(D^2b \nabla f).\nabla g +(D^2f \mathbf{n}).\nabla g +(D^2 g~ \mathbf{n}).\nabla f \\
&=&~~2(D^2b \nabla f).\nabla g +\partial_\mathbf{n} (\nabla f.\nabla  g)
\end{array}
$$
We use now the decomposition of $\nabla$ into its normal and tangential components and the well known identity $D^2b\mathbf{ n}.{\mathbf{n}}=0$. We get 
\begin{align}
\nablat  (\partial_\mathbf{n}f).\nablat  g
+ \nablat (\partial_\mathbf{n}g).\nablat f&+\displaystyle \frac{\partial^2 f}{\partial n^2}\displaystyle \frac{\partial g}{\partial n}
+\displaystyle \frac{\partial^2 g}{\partial n^2}\displaystyle \frac{\partial f}{\partial n}  \nonumber= \\
&2(D^2b \nablat f).\nablat g +\partial_\mathbf{n} (\nablat f.\nablat  g)+ \displaystyle \frac{\partial^2 f}{\partial n^2}\displaystyle \frac{\partial g}{\partial n}+\displaystyle \frac{\partial^2 g}{\partial n^2}\displaystyle \frac{\partial f}{\partial n}
\end{align}
hence 
$$
\nablat  (\partial_\mathbf{n}f).\nablat  g
+ \nablat (\partial_\mathbf{n}g).\nablat f=2(D^2b \nablat f).\nablat g +\partial_\mathbf{n} (\nablat f.\nablat  g)
$$ 
\end{proof}

\section{Spherical harmonics}\label{app:harmonics}

In order to  explicit the shape hessian under consideration, a useful tool is the surface spherical harmonics defined as the restriction to the surface of the unit sphere of harmonic polynomials {\Be in the special case  $d=3$}. We recall here facts from \cite[pages 139-141]{SteinWeiss}. 
{\Rd Spherical Harmonics are defined as restrictions of homogeneous harmonic polynomials  to the unit sphere. The spherical harmonics are said of order $k$ when the harmonic homogeneous polynomial is of degree $k$}. 
We  denote by {\Rd $\mathcal{H}_{k}$ the space of spherical harmonics of degree $k$. We show that is also } the eigenspace of the Laplace-Beltrami operator on the unit sphere associated with the eigenvalue ${\Rd k(k+1)}$. Its dimension is $${\Be d_{k}=2k+1}.
$$ 
Let $\Be (Y^{l}_{k})_{-k\leq l \leq k}$ be an orthonormal basis of $\mathcal{H}_{k}$ with respect to the $\sL^2(\partial {\Rd B_{1})}$ scalar product. The $(\mathcal{H}_{k})_{k \in \mathbb{N}}$ spans a vector space dense in $\sL^{2}({\Rd\partial B_{1}})$ and the family {\Be $(Y^{l}_{k})_{k\in \mathbb{N}, -k\leq l \leq k}$ is a Hilbert basis of $\sL^2(\partial B_{1})$. To be more precise, if $f \in \sL^2(\partial B_1)$, then there exists a unique representation $$f =\sum_{k=0}^\infty \mathbf{Y}_k$$ where the series converge to $f$ in the $\sL^2$ norm and 
$$\mathbf{Y}_k=\sum_{l=-k}^k b^{l}_k Y_{k}^l \in \mathcal{H}_k $$
 \par
 \noindent
  {\Be If $x=(x_1,x_2,x_3)\in\mathbb{R}^3$, it is natural to use on a sphere the spherical coordinates $(r,\theta,\phi)$ where $r$ is the radius and $\theta$ and $\phi$ are the Euler angles.  The  spherical harmonic $Y_k^l$ is defined with the Euler angles $(\theta,\phi)$ as 
$$
Y_k^l=(-1)^l \sqrt{\Big[ \cfrac{k+\frac{1}{2}}{2\pi} \cfrac{(k-l)!}{(k+l)!} \Big]}e^{il\phi }\mathbb{P}_{k}^{l}(\cos{\theta}),~-k\le l\le k.
$$
where the  polynomial $\mathbb{P}_k^l$ is the associated Legendre polynomial. The formula giving the explicit form of these polynomials can be found in the book of Nedelec  \cite[page 24]{Nedelec}.  

When $k\neq k'$, we have also the orthogonality property 
$$\int_{\partial B_1} \mathbf{Y}_k\mathbf{Y}_{k'}  d\sigma=0$$ when $\mathbf{Y}_k\in \mathcal{H}_k$ and $\mathbf{Y}_{k'}\in \mathcal{H}_{k'}$.
An homogeneity argument shows that any  function $\varphi$ in $L^2(\partial B_{R})$ can be decomposed as the Fourier series:
$$\varphi(x)= \sum_{k=0}^{\infty}{\Rd R^k  \left( \sum_{l=-k}^{k} \alpha_{k,l}(\varphi) Y^{l}_{k}\left(\frac{x}{\mid x\mid}\right)\right)}, \ \ \textrm{ for }|x|=R.$$
Then, by construction, the function $u$ defined by
$$u(x) = \sum_{k=0}^{\infty} {\Rd  |x|^k\left( \sum_{l=1}^{d_{k}} \alpha_{k,l}(\varphi) Y^{l}_{k}\left(\frac{x}{|x|}\right)\right)}, \ \ \textrm{ for }|x|\leq R,$$}
is harmonic in $B_{R}$ and satisfies $u=\varphi$ on $\partial B_{R}$. 

We recall now some results about the  {\Be{integration of three spherical harmonics}}, they will enable us to estimate  $\Tr(E)$ in dimension three. {\Be{When we integrate three spherical harmonics, we use coefficients called Clebsch-Gordon coefficients or Wigner-$3j$ coefficients}}.  {\Be{The Wigner-$3j$ coefficients  are mostly used;  they  are related to Clebsch-Gordon coefficients via some known formula that the interested reader will find  in the book of Cohen-Tannoudji and al \cite[Tome 2, Annex B]{CohenTannoudji}}}.
\par
\noindent
{\Be{The first general result concerns the product of two  spherical harmonics; it  is given by  the following proposition}}
\begin{proposition}
\label{produit:deux:harmoniques}
Given $l_{1},l_{2}>0$ two natural integers and $-l_{1}\leq m_{1} \leq l_{1}$,  $-l_{2}\leq m_{2} \leq l_{2}$, we have
\begin{align*}
Y_{l_{1}}^{m_{1}}&Y_{l_{2}}^{m_{2}} =\\(-1)^{m_{1}+m_{2}}& \sum_{L=|l_{1}-l_{2}|}^{l_{1}+l_{2}} \sqrt{\cfrac{(2l_{1}+1)(2l_{2}+1)(2L+1)}{4\pi}} 
\begin{pmatrix} l_{1} &l_{2} &L\\ 0&0&0 \end{pmatrix}
\begin{pmatrix} l_{1} &l_{2} &L\\ m_{1}&m_{2}&-m_{1}-m_{2} \end{pmatrix}
Y_{L}^{m_{1}+m_{2}},
\end{align*}
 where  $\begin{pmatrix} l_{1} &l_{2} &L\\ 0&0&0 \end{pmatrix}$ and $
\begin{pmatrix} l_{1} &l_{2} &L\\ m_{1}&m_{2}&-m_{1}-m_{2} \end{pmatrix}$ are the Wigner-$3j$ symbols. 
\end{proposition}
The second  result concerns the integration of three spherical harmonics.  
\begin{proposition}
\label{inetrage:produit:trois:harmoniques}
We have:
\begin{align*}
\int_{\partial B_{1}} Y_{l_{1}}^{m_{1}} Y_{l_{2}}^{m_{2}}Y_{l_{3}}^{m_{3}} =
 \sqrt{\cfrac{(2l_{1}+1)(2l_{2}+1)(2l_{3}+1)}{4\pi}} 
\begin{pmatrix} l_{1} &l_{2} &l_3\\ 0&0&0 \end{pmatrix}
\begin{pmatrix} l_{1} &l_{2} &l_3\\ m_{1}&m_{2}&m_{3} \end{pmatrix}.
\end{align*}
\end{proposition}
\par
\noindent
\medskip
\par
\noindent

\Bk
In particular it holds
\begin{proposition}
\label{inetrage:produit:trois:harmoniques:bis}
Let $l$ be  a natural integer and $m$ an integer. We have \Rd
\begin{enumerate}
\item If $-l\le m\le l $ then 
\begin{align*}
\int_{\partial B_{1}} Y^{m}_{l} Y_{0}^{0}\overline{Y_{l}^{m}} =
 \sqrt{\cfrac{1}{4\pi}},
 \end{align*}
 and
 \begin{align*}
\int_{\partial B_{1}} Y^{m}_{l} Y_{1}^{1}\overline{Y^{m+1}_{l-1}} =-
 \sqrt{\cfrac{3}{8\pi}} \ \sqrt{\cfrac{(l-m)(l-m-1)}{(2l+1)(2l-1)}}.
\end{align*}
\item If   $-l-1\le m\le l+1$ then 
\begin{align*}
\int_{\partial B_{1}} Y_{l}^{m} Y_{1}^{0}\overline{Y^{m}_{l+1}} =
 \sqrt{\cfrac{3}{4\pi}} \ \sqrt{\cfrac{(l+m+1)(l-m+1)}{(2l+1)(2l+3)}},
\end{align*}
\item If $-l-2\le m\le l$ then 
\begin{align*}
\int_{\partial B_{1}} Y^{m}_{l} Y_{1}^{1}\overline{Y^{m+1}_{l+1}} =
 \sqrt{\cfrac{3}{8\pi}} \ \sqrt{\cfrac{(l+m+1)(l+m+2)}{(2l+1)(2l+3)}},
\end{align*}
\end{enumerate}
\Bk
\end{proposition}

\section{Intermediate results for the second shape  derivative matrix }\label{app:E2}

We need to construct the matrix associated to the second shape derivative. To that end, we have to compute the explicit formula for all the shape  derivatives of order one  involved  in the formula giving $\lambda'$ (see Theorem \ref{derivee:vp:simple}). In this appendix, we focus on the term $E^{(2)}$ introduced in Section \ref{ssect:E''}.
Since these computations are very technical, we only give the main line  and the used arguments, omitting a couple of details.  In the following lines, we denote by $H(t)$ the mean curvature associated to the boundary of $\Omega_{t}$ and $\mathcal{A}(t)$ the deviatoric part defined on $\partial\Omega_{t}$ as $$\mathcal{A}(t)=H(t) I -2D^2 b(t)$$ (see \cite{DelfourZolesio}  for the terminology). 

In order to deal with the weak formulation on the boundary $\partial\Om_t$, we will make use of a test function $\phi$ which is the restriction of a test function $\Phi$ defined on a tubular neighborhood of the boundary such that its normal derivative is zero. This kind of extension is well discussed in the book \cite{DelfourZolesio} of Delfour-Zolesio.   

In this differentiation, nineteen terms arise and we introduce some notations to study them separately. For all function test $\phi \in H^1(\partial\Om)$, we will need in the sequel the following quantities:
\begin{eqnarray*}
A(u,u',\phi)&=&\restriction{\left( \frac{d}{dt}\int_{\partial \Omega_t} ~V_n{}\nablat u .\nablat \phi d\sigma_t\right)}{t=0},\\
B(u,u',\phi)&=&\beta  \restriction{\left(\frac{d}{dt}  \int_{\partial \Omega_t}\mathcal{A}(t)V_n \nablat u. \nablat \phi~d\sigma_t\right)}{t=0},\\
C(u,u',u'',\phi)&=&-\displaystyle\frac{d}{dt}\restriction{\left( \displaystyle\int_{\partial \Omega_t} ~\Big[\lambda' u+ \lambda(u'+V_n \partial_n u +V_n Hu)\Big]\phi~d\sigma_{t} \right)}{t=0},\\
D(u,u',\phi)&=&\beta \restriction{\left( \frac{d}{dt}\int_{\partial \Omega_t} ~\nablat(V_n \partial_n u) .\nablat \phi \ d\sigma_t\right)}{t=0}.\\
\end{eqnarray*}
We will now study independently each term $A,B,C$ and $D$, when $\Om=B_{R}\subset\R^2$ or $\R^3$, and $t\mapsto\Om_{t}$ is volume preserving.

\paragraph{Study of $D(u,u',\phi)$.}

First, {\Rd we denote 
$$W= \frac{d}{dt} \left(\V\cdot \n_{\Omega_{t}}\right)_{|t=0}.$$}From the derivative formula of boundary integrals, we know that we have to compute three  main terms:  the first corresponding to the shape derivative, the second concerns the normal derivative of the integrand and the third is related to the term related to   the mean curvature $H$.  
\Be The first term is
\begin{align*}
\beta &\restriction{\left( \int_{\partial \Omega_t} ~\frac{d}{dt}\left[ \nablat(V_n \partial_n u) .\nablat \phi\right]\ d\sigma_t\right)}{t=0}\\
&~~~~~~~~~~=\beta \Big(\displaystyle\int_{\partial B_R}  \nablat(V_n.\partial_n u'-V_n \nablat u.\nablat V_n) .\nablat \phi~d\sigma+ 
 \int_{\partial B_R}  \nablat(V'_n.\partial_n u) .\nablat \phi~d\sigma         \Big)\\
&~~~~~~~~~~~~~~~+\beta \int_{\partial B_R} \partial_n(V_n\partial_n u)\nablat V_n.\nablat \phi~d\sigma\\
&~~~~~~~~~~=-\beta \int_{\partial B_R}  \Deltat(V_n.\partial_n u') \phi~d\sigma+ \beta
\int_{\partial B_R}  \nablat(V'_n.\partial_n u) .\nablat \phi~d\sigma       \\
&~~~~~~~~~~~~~~~+\beta \int_{\partial B_R} \partial_n (V_n\partial_n u) \nablat V_n.\nablat \phi~d\sigma+\beta \int_{\partial B_R}  \Deltat \Big(  V_n\nablat u .\nablat V_n\Big) \phi ~d\sigma.
\end{align*}
\Bk
The third term is 
$$
\beta\displaystyle\int_{\partial B_R}HV_n \nablat(V_n{}  \partial_n u).\nablat \phi~d\sigma =-\beta \int_{\partial B_{R}} \divet \Big(H V_n \nablat(V_n\partial_n u)\Big)  \phi ~d\sigma.
$$
We focus now on \Be the second term. \Bk 
We have 
\begin{align*}
\beta& \int_{\partial B_R} V_n\partial_n[\nablat (V_n\partial_n u).\nablat\phi]~d\sigma\\
&= \beta \int_{\partial B_R} V_n\nablat[ \partial_n(V_n\partial_n u)].\nablat\phi~d\sigma -2\beta\int_{\partial B_R} V_n (D^2b \nablat [V_n\partial_n u]).\nablat\phi~d\sigma\\
&=\beta\left( \int_{\partial B_R} V_n\nablat\left[ \partial_n V_n \partial_n u)\right].\nablat\phi~d\sigma -2 \int_{\partial B_R} V_n (D^2b \nablat [V_n\partial_n u]).\nablat\phi~d\sigma\right)
\\
&=-\beta \int_{\partial B_R}\divet [ V_n \nablat [\partial_n u~ \partial_n V_n] -2  V_n D^2b \nablat[V_n\partial_n u]]  \phi~d\sigma.
\end{align*}
We expand  $D(u,\phi)$ into  a  sum 
$\langle D^{(1)} u',\phi \rangle +\langle D^{(2)} u,\phi\rangle$. For $D^{(2)}$, we will set $D^{(2)} =\sum_{k=1}^3 D^{(2,k)}$
where
\begin{eqnarray*}
\langle D^{(1)} u',\phi \rangle &=& 
\beta \int_{\partial B_R}  \nablat[V_n.\partial_n u'] .\nablat \phi~d\sigma
=-\beta \int_{\partial B_R}\Deltat [V_n\partial_n u'] \phi~d\sigma \\
\langle D^{(2,1)} u,\phi \rangle &=&\beta\Big[\int_{\partial B_R}  (-\Deltat [\Rd W\Bk\partial_n u]\  \phi~d\sigma-\int_{\partial B_{R}}\divet [V_n\partial_n V_n \nablat [\partial_n u]]\ \phi~d\sigma, \nonumber 
\\
&& - \int_{\partial B_{R}} \divet [H V_n \nablat(V_n\partial_n u) ] \ \phi  ~d\sigma\Big],\\
\langle D^{(2,2)} u,\phi \rangle &=&-\beta\int_{\partial B_R} \divet[\partial_n u ~\partial_n V_n \nablat V_n]  \phi~d\sigma+\beta \displaystyle \int_{\partial B_R}  \Deltat [ V_n\partial_nu \nablat V_n] \phi ~d\sigma, \\
\langle D^{(2,3)} u,\phi \rangle &=&2\beta \int_{\partial B_R}  \divet[ V_n D^2b \cdot \nablat[V_n\partial_n u] ]  \phi~d\sigma.
\end{eqnarray*}
We denote $D^{(1)}$ and  $D^{(2,k)},~k=1,2,3$ the matrices whose elements are defined by 
$$
D^{(1)}_{ij}=\langle D^{(1)}\tilde  u_i, u_j\rangle,
\textrm{ and }D^{(2,k)}_{ij}=\langle D^{(2,k)} u_i, u_j\rangle,~i,j=1,2,\ldots,d.
$$
We give a result concerning the traces of the matrices.

\begin{lemma}\label{tracematriced}
We have 
$$
\Tr(D^{(2,1)})=\Tr(D^{(2,2)})=0
\textrm{ and }
\Tr(D^{(2,3)})=-\cfrac{2\beta (d-1)K(R)}{R}\int_{\partial B_{R}} V_n^2~~d\sigma ,
$$
with the normalization constant $K(R)= \cfrac{d}{R^{2+d} \omega_{d}}.$ 
\end{lemma}

\begin{proofof}{Lemma \ref{tracematriced}}
We have 
\begin{eqnarray}\label{snd_snderiv1}
\Tr(D^{(2,1)} )&=&\beta\Big[\int_{\partial B_R}  -\Deltat (\Rd W \Bk \sum_{i=1}^d\partial_n  u_i)  u_i~d\sigma 
-\int_{\partial B_R
} \divet (V_n\partial_n V_n \sum_{i=1}^d\nablat (\partial_n  u_i)) u_i~d\sigma \nonumber 
\\
&& - \int_{\partial B_R}\Be \sum_{i=1}^d\Bk \divet \left(H V_n \nablat(V_n\partial_n  u_i) \Be\right)\Bk u_i ~d\sigma\Big]\\
&=&\beta \int_{\partial B_R} V_n' (d-1)\sum_{i=1}^d| \partial_n  u_i) |^2~\cfrac{d\sigma}{R}+ \beta  \int_{\partial B_R} V_n \partial_n V_n \sum_{i=1}^d  \nablat (\partial_n   u_i)\Be\cdot\Bk \nablat  u_i ~\partial \sigma \nonumber\\
&&+\beta \displaystyle \int_{\partial B_R}HV_n^2 \sum_{i=1}^d \nablat (\partial_n   u_i)\Be\cdot\Bk \nablat  u_i ~d\sigma +\beta \displaystyle \int_{\partial B_R}H\sum_{i=1}^d 
 \partial_n  u_i V_n \nablat V_n\Be\cdot\Bk \nablat  u_i.~d \sigma \nonumber
\end{eqnarray}
Combining the two facts (coming from algebraic properties of spherical harmonics, see Appendix \ref{app:harmonics}),
\begin{equation}
\label{definition:K(R)}
(d-1)\sum_{i=1}^d\cfrac{| \partial_n  u_i |^2}{R}= \sum_{i=1}^d \nablat (\partial_n u_i )\cdot\nablat  u_i=\cfrac{d(d-1)}{R^{2+d}\omega_{d}} =(d-1)K(R).
\end{equation}
and 
\begin{equation}
\label{prop:harmoniques}
\int_{\partial B_R}  V_{n} \sum_{i=1}^d \partial_n  u_i \nablat V_n\cdot \nablat  u_i=0,
\end{equation}
 we get 
$$
\Tr(D^{(2,1)})=(d-1)\sum_{i=1}^d | \partial_n  u_i  |^2 \int_{\partial B_R}\Big(\Rd W\Bk+V_n\partial_{n} V_n +HV_n^2\Big)~\cfrac{d\sigma}{R}.
$$Since we assumed the deformation \Rd to be volume preserving up to the second order \eqref{volume:preserving:second:order}, we have \Bk
$\Tr(D^{(2,1)})=0.$
The same strategy applies for $\Tr(D^{(2,2)})$. 
\par
\noindent
We focus now on $\Tr(D^{(2,3)})$.  We first expand the second term in the definition of $D^{(4)}$: 
\begin{align*}
\Tr(D^{(2,3)})\ = \ & \beta \sum_{i=1}^d   \int_{\partial B_R}  V_n\partial_n  u_i \ \nablat[\partial_nV_n].\nablat  u_i -2  V_n \partial_n u_{i} \ D^2b \nablat V_n\cdot\nablat  u_i~ d\sigma\\
 &-\beta \sum_{i=1}^d  \int_{\partial B_R} 2V_n^2\ D^2b\nablat(\partial_n   u_i) \cdot\nablat u_i ~d\sigma.
\end{align*}
We follow the same argument thanks to the relations \eqref{definition:K(R)}-\eqref{prop:harmoniques} and the  fact  
$$
 \sum_{i=1}^d  D^2b \nablat(\partial_n  u_i).\nablat u_i =\cfrac{(d-1)K(R)}{R}
$$
on the sphere. Recall that on the sphere $D^2 b=I_{d}/R$ when restricted to the tangent space.
\end{proofof}

\paragraph{Study of $B(u,u',\phi)$. } In the same manner, we begin to compute the derivative of the integrand:
\begin{align*}
\frac{d}{dt}\restriction{\Big( \mathcal{A}(t)V_n \nablat u. \nablat \phi\Big)}{t=0}&= \\\mathcal{A}'V_n \nablat u. \nablat \phi&+\mathcal{A}V'_n \nablat u. \nablat \phi+\mathcal{A}V_n \nablat u'. \nablat \phi-\mathcal{A}V_n\partial_n u \nablat V_n.\nablat \phi.
\end{align*}
Denote  $\mathcal{A}=(a_{ij})_{1\le i,j\le d}$  and   $\tilde{\mathcal{A}}=(\partial_n a_{ij})_{1\le i,j\le d}$. Thanks to lemma \ref{commutateur}, we get
\begin{align*}
 V_n\partial_n &\Big(V_n\mathcal{A} \cdot \nablat u. \nablat\phi\Big)=V_n^2 \tilde{\mathcal{A}} \cdot \nablat u. \nablat\phi+ V_n\partial_n V_n \mathcal{A}\cdot  \nablat u. \nablat\phi+ V_n^2 \mathcal{A}\partial_n\Big[  \nablat u. \nablat\phi \Big].\\
\end{align*}
From the relation
\begin{align*}
\beta  \frac{d}{dt}&\restriction{  \int_{\partial \Omega_t}\mathcal{A}(t)V_n \nabla_{\partial \Omega_t}u. \nabla_{\partial \Omega_t}\phi~d\sigma_t}{t=0}\ =\ \int_{\partial B_R}
\frac{d}{dt}\restriction{\Big( \mathcal{A}(t)V_n \nabla_{\partial \Omega_t}u. \nabla_{\partial \Omega_t}\phi\Big)}{t=0}~d\sigma\\
& \ \ \ \ \ \ \ \ \ + \int_{\partial B_R}V_n\partial_n \Big(\mathcal{A}V_n \nablat u. \nablat\phi\Big)~d\sigma
+ \int_{\partial B_R}HV_n^2\mathcal{A} \nabla_t u. \nablat \phi~d\sigma_t,
\end{align*}
we gather all the terms and obtain  $ B(u,\phi)=\langle B^{(1)}u',\phi\rangle+\langle B^{(2)}u',\phi\rangle$; we then set $$\langle B^{2)}u,\phi\rangle=\sum_{i=1}^4 \langle B^{(2,i)}u,\phi\rangle,$$ where 
\begin{eqnarray*}
\langle B^{(2,1)}u',\phi\rangle&=&
-\beta \displaystyle\int_{\partial B_R} \divet[V_{n }A \cdot\nablat u' ]\phi~d\sigma  , \\
\langle B^{(2,1)} u,\phi \rangle &=&-\beta \displaystyle\int_{\partial B_R} \divet[ (\Rd W\Bk+HV_n^2+V_n\partial_n V_n)~\mathcal{A}\cdot\nablat u] \phi~d\sigma,\\
\langle B^{(2,2)} u,\phi \rangle &=&-\beta \displaystyle \int_{\partial B_R}\divet [\partial_n u ~V_n \mathcal{A}\cdot \nablat V_n ] \phi~d\sigma,\\
\langle B^{(2,3)}u,\phi\rangle&=&
-\beta \displaystyle\int_{\partial B_R} \divet[V_{n }\mathcal{A}'\cdot \nablat u]\phi~d\sigma ,\\
 \langle B^{(2,4)}u,\phi\rangle&=&
\beta \displaystyle\int_{\partial B_R} V_n^2\partial_n \left[\mathcal{A} \cdot\nablat u.\nablat \phi\right]~d\sigma.    \nonumber\\
\end{eqnarray*}
We get 
\begin{eqnarray*}
B^{(2,4)}u,\phi\rangle&=&\beta \displaystyle\int_{\partial B_R} V_n^2\left( \partial_n [\mathcal{A}]\cdot \nablat u.\nablat \phi\right)~d\sigma  +\beta \displaystyle\int_{\partial B_R} V_n^2  \mathcal{A} \cdot\nablat \partial_n u.\nablat \phi~d\sigma \nonumber\\
& &  -\beta \displaystyle\int_{\partial \Omega} 2(D^2b \mathcal{A})\cdot \nablat u.\nablat \phi ~d\sigma  \nonumber\\
&=&-\beta  \displaystyle\int_{\partial B_R} 
\divet \Big[V_n^2 \Big( \tilde{\mathcal{A}} \cdot\nablat u+ \mathcal{A}\cdot \nablat [\partial_n u])-2D^2b \mathcal{A}\cdot\nablat u     \Big)\Big]\phi~d\sigma \nonumber
\end{eqnarray*}
Let $B^{(2,k)}~k=1,2,3,4$ denote the respective matrices  associated to  the operator with respect to the basis of eigenvectors. We have the following result:

\begin{lemma}\label{trace:B}
We have 
$$
\Tr(\sum_{i=1}^{4}B^{(2,i)})\ =- \beta(d-1) R K(R)\  \int_{\partial B_R}|\nablat V_n|^2~d\sigma \ + 2\cfrac{\beta K(R)}{R}\displaystyle\int_{\partial B_{R}}V_n^2~d\sigma.
$$
\end{lemma}

\begin{proofof}{Lemma \ref{trace:B}}
Using the same arguments as before, we prove easily that
$
\Tr(B^{(2,1)})=\Tr(B^{(2,2)})=0.
$
\par
\noindent
For the other terms, above all we have  to focus on the term $$\Tr{(B^{(2,3)})}=\beta \displaystyle\int_{\partial B_R} V_{n }\sum_{i=1}^d (\mathcal{A}'\cdot \nablat u_i).\nablat u_i~d\sigma.$$ We have, thanks to the expression of shape derivation of the normal vector and of the mean curvature given in Proposition \ref{derivees:normale:H}:
$$
\mathcal{A}(t)=H(t)-2D^2 b(t)\Rightarrow \mathcal{A}'=-\Deltat{V_n}+2D(\nablat V_n);
$$
then 
$$
\begin{array}{lll}
\Tr{(B^{(2,3))}}&=&\beta \displaystyle\int_{\partial B_R} V_{n }\sum_{i=1}^d (\mathcal{A}'\cdot \nablat u_i).\nablat u_i~d\sigma\\
&=&-\beta \displaystyle\int_{\partial B_R}  V_n\Deltat{V_n}\sum_{i=1}^d |\nablat u_i|^2~d\sigma+ 2  \beta \displaystyle\int_{\partial B_R}  V_n \sum_{i=1}^d\Big[ D(\nablat{V_n})\cdot\nablat u_i \Big].\nablat u_i~d\sigma\\
&=&-\beta \displaystyle\int_{\partial B_R}  V_n\Deltat{V_n}\sum_{i=1}^d |\nablat u_i|^2~d\sigma+ 2  \beta \displaystyle\int_{\partial B_R}  V_n \sum_{i=1}^d\Big[ D_\tau(\nablat {V_n})\cdot\nablat u_i \Big].\nablat u_i~d\sigma\\
&=&-\beta \displaystyle\int_{\partial B_R}  V_n\Deltat{V_n}\sum_{i=1}^d |\nablat u_i|^2~d\sigma+ 2  \beta \displaystyle\int_{\partial B_R}  V_n \sum_{i=1}^d\Big[ D^2_\tau{V_n}\cdot\nablat u_i \Big].\nablat u_i~d\sigma\\
&=&-\beta \displaystyle\int_{\partial B_R}  V_n\Deltat{V_n}\sum_{i=1}^d |\nablat u_i|^2~d\sigma+ 2  \beta \displaystyle\int_{\partial B_R}  V_n  \Tr{(D^2_\tau{V_n})}\sum_{i=1}^d |\nablat u_i |^2~d\sigma
\end{array}
$$
Since $  \Tr{(D^2_\tau{V_n})}=\Deltat V_n$, and since $\sum_{i=1}^d | \nabla_\tau u_i |^2=RK(R),~\text{on~}\partial B_R$  we get 
$$
\Tr(B^{(2,3)})=\beta \displaystyle\int_{\partial B_{R}} V_n\Deltat V_n \sum_{i=1}^d |\nablat  u_i|^2~d\sigma\\
=\beta(d-1) RK(R) \displaystyle\int_{\partial B_{R}} V_n\Deltat V_n ~d\sigma.
$$
Concerning $\Tr(B^{(2,4)})$,  we have to distinguish the case $d=2$ from the case $d=3$. 
If $d=3$ then $\mathcal{A}=0$ ; this implies that  $\Tr(B^{(2,4)})$ is reduced to 
$$
\Tr(B^{(2,4)})=(d-1)K(R)\cfrac{\beta}{R}  \displaystyle\int_{\partial B_R} V_n^2 ~d\sigma.
$$
If $d=2$, then $\mathcal{A}+\tilde{\mathcal{A}}$ is a null matrix and this leads to 
$$
\begin{array}{lll}
\Tr(B^{(2,4)})&=&2\beta  \displaystyle\int_{\partial B_R} V_n^2 \sum_{i=1}^d D^2b\cdot \nablat  u_i.\nablat  u_i~d\sigma\\
&=& 2K(R)\cfrac{\beta}{R}  \displaystyle\int_{\partial B_R}  V_n^2~d\sigma.
\end{array}$$
Then for $d=2,3$ we get 
$$
\Tr(B^{(2,4)})=2\beta \cfrac{K(R)}{R};
$$
\par
\noindent
\end{proofof}

\paragraph{Study of $A(u,u',\phi)$.}
We \Be have\Bk 
\begin{align*}
& \restriction{\frac{d}{dt} \left( \int_{\partial \Omega_t} ~V_n{}\nablat u .\nablat \phi d\sigma_t \right)}{t=0}=  \int_{\partial B_R} \Rd W \Bk \nablat u.\nablat{}\phi~d\sigma +  \int_{\partial B_R} V_n \nablat u'.\nablat{}\phi~d\sigma\\
& ~~~~~~~~~~+   \int_{\partial B_R} V_n~\nablat V_n.\Big[\partial_n u ~\nablat \phi+ \partial_n \phi  \nablat  u\Big] + \Big(V_n \partial_n[V_n\nablat u.\nablat \phi]+ H V_n^2 \nablat u.\nablat\phi\Big)~d\sigma.
\end{align*}
Since $\partial_n \phi=0$, it comes that
\begin{align*}
  \int_{\partial B_R} V_n~&\nablat V_n.\Big[\partial_n u ~\nablat \phi+ \partial_n \phi  \nablat  u\Big] ~d\sigma=- \frac{1}{2} \int_{\partial B_R} V_n^2\Big[\partial_n u \Deltat \phi +\nablat[\partial_n u].\nablat \phi\Big]\\
&~~~~~~~~~~= - \frac{1}{2} \int_{\partial B_R} V_n^2 \left( \partial_n[\nablat u.\nablat  \phi]+2D^2b \nablat u.\nablat  \phi \right)~d\sigma - \frac{1}{2} \int_{\partial B_R} V_n^2~\partial_n u \Deltat \phi ~d\sigma\\
\end{align*}
Hence, gathering the equivalent terms we get 
\begin{align*}
\frac{d}{dt}& \restriction{ \int_{\partial \Omega_t} ~V_n{}\nablat u .\nablat \phi d\sigma_t}{t=0}=   \int_{\partial B_R} \Rd W \Bk  \nablat u'.\nablat \phi~d\sigma +  \int_{\partial B_R} V_n \nablat u'.\nablat \phi~d\sigma\\
&~~~~~~~~~~ - \frac{1}{2} \int_{\partial B_R} \Deltat[V_n^2~\partial_n u ] \phi -\partial_n\left(V_n^2 \nablat  u.\nablat \phi \right)~d\sigma ~+ ~ \int_{\partial B_R} ( HI_{d} -D^2b)V_n^2 \nablat u.\nablat\phi~d\sigma.
\end{align*}
We split these terms into 
$ A(u,\phi)=\langle A^{(1)} u',\phi \rangle   +   \langle A^{(2)} u,\phi\rangle$. As before, we set $  \langle A^{(2)} u,\phi\rangle=  \sum_{i=1}^3\langle A^{(i)} u,\phi\rangle$
where
\begin{eqnarray*}
\langle A^{(1)} u',\phi \rangle &=& 
\int_{\partial B_{R}}  -\divet [V_n \nablat u' ]\phi~d\sigma,\\
\langle A^{(2,1)} u,\phi \rangle &=&\int_{\partial B_R}  -\divet[(\Rd W \Bk+HV_n^2+V_n\partial_n V_n)~\nablat u]~\phi~d\sigma,\\
\langle A^{(2,2)} u,\phi \rangle &=&\int_{\partial B_R}\divet[ \partial_n u ~V_n \nablat V_n]~\phi~d\sigma,\\
\langle A^{(2,3)} u,\phi \rangle &=&\int_{\partial B_R} \divet[V_n^2\left(2D^2b\nablat u -\nablat( \partial_n u)\right)] \phi~d\sigma.
\end{eqnarray*}
We have 
\par
\noindent
\begin{lemma} \label{Trace:A}
We have
$$\Tr(A^{(2,1)})=0,~\Tr(A^{(2,2)})=0 \textrm{ and }
\Tr(A^{(2,3)})=-K(R)\displaystyle\int_{\partial B_{R}} V_n^2~~d\sigma .$$ 
\end{lemma}
The proof of Lemma \ref{Trace:A} follows the lines of the proof of Lemma \ref{trace:B}.
 
\paragraph{Study of $C(u,u',u'',\phi)$. } We decompose $C(u,u',u'',\phi)$ as follows: 
$$ C(u,\phi) =\langle C^{(0)}u'',\phi\rangle + \langle C^{(1)}u',\phi\rangle +\langle  C^{(2)}u,\phi\rangle $$ 
with $\langle  C^{(2)}u,\phi\rangle =\sum_{i=3}^6 \langle C^{(2,i)}u,\phi\rangle$ 
where 
\begin{eqnarray*}
\langle C^{(0)}u'',\phi\rangle&=&-\lambda  \int_{\partial B_R} u'' \phi~d\sigma\\
\langle C^{(1)}u',\phi\rangle &=&-2  \int_{\partial B_R} \left(\lambda'  u'+\lambda V_n ( \partial_nu'+Hu') \right) \phi~d\sigma \\
\langle C^{(2,1)}u,\phi\rangle &=&-\lambda''   \int_{\partial B_R} u\phi-\lambda'   \int_{\partial B_R }V_n\partial_n u\phi~d\sigma\\
\langle C^{(2,2)}u,\phi\rangle &=&-\lambda  \int_{\partial B_R} \left(\Rd W \Bk+V_n\partial_n V_n+HV_n^2)(\partial_n u + H u)\right){}\phi~d\sigma\\
\langle C^{(2,3)}u,\phi\rangle &=&-\lambda   \int_{\partial B_R}   V_n\Big( -\nablat V_n . \nablat u+H'u \Big)   {}\phi  ~d\sigma\\
&=&\lambda   \int_{\partial B_R}   V_n\Big( \nablat V_n . \nablat u+\Deltat V_n \ u \Big)   {}\phi  ~d\sigma\\
\langle C^{(2,4)}u,\phi\rangle &=&-\lambda \int_{\partial B_R} V_n^2\Big(   \partial_n^2u-u \sum_{i=1}^{d-1} \kappa_i^2 + H\partial_n u \Big){}\phi~d\sigma\\
&=&0.
\end{eqnarray*}
Denoting  by $(C^{(2,j)}),~j=1,2,3,4$ the matrices associated to the linear operators $C^{(2,p)},~p=1,2,3,4$ in the basis of eigenvectors, we get:

\begin{lemma}\label{trace:C}
We have 
$$
\sum_{j=1}^{4} \Tr(C^{(2,j)})=\lambda R^3 K(R) \int_{\partial B_{R}} V_n \Deltat V_n~d\sigma=-\left((d-1)\beta+R\right)RK(R) \int_{\partial B_{R}} |\nablat V_n|^2~d\sigma.$$
\end{lemma}

\begin{proofof}{Lemma \ref{trace:C}}
The proof is straightforward and obeys to the same arguments used before. The only non null trace concerns the factor in $-H'=\Deltat V_n.$
\end{proofof}

\section{Computing $u'$}\label{app:u'}

In this section, we focus on the computation of the trace of $E^{(1)}$ introduced in Section \ref{ssect:E''}.
We recall that $t\mapsto(\lambda(t),u(t,\cdot))$ is solution of 
\begin{equation}\label{derivation_1}
\begin{array}{rlll}
\Delta u&=&0&~\textrm{in~}T_{t}(B_{R}),\\
-\beta \Deltat u+\partial_{n} u-\lambda(t) u&=&0&~\textrm{on~}\partial T_{t}(B_R).
\end{array}
\end{equation}
To compute the second derivative, one must know $u'=u'(0)$. For the reader convenience,  we recall the problem \eqref{shape_derivative} solved by $u'$.
\begin{align*}
\Delta u' = 0 &\textrm{ in }B_{R}, \nonumber\\
-\beta \Deltat u'+\partial_n{u'} -& \lambda v' = \beta \Deltat (V_n \partial_n u) -\beta \divet\big(V_n(2D^2b-H I_{d})\nablat u\big)\nonumber \\
&+ \divet(V_n\nablat u)  -\lambda' u+\lambda  V_n( \partial_nu+ H u) \textrm{ on }\partial B_{R}.
\end{align*}  
First,  Fredholm's alternative insures the existence of a unique harmonic function $\tilde  u_j$ orthogonal to the 
eigenfunctions $u_1,u_2,\ldots,u_d$ and satisfying on $\partial B_R $ the boundary condition 
\begin{eqnarray}\label{tildefi}
-\beta \Deltat  \tilde{u}_j+\partial_n \tilde{u}_j-\lambda \tilde{u}_j&=&\beta\Big[\Deltat[V_n\partial_n  u_j]+\divet[V_n(H I_{d}-2D^2b)\cdot\nablat u_j]\Big]
\nonumber \\
&&+\divet[V_n\nablat  u_j]+\lambda' u_j+\lambda V_n(\partial_n u_j+H u_j).
\end{eqnarray}
It follows that 
\begin{equation}
u'=\sum_{j=1}^d \tilde{c}_j  u_j+\sum_{j=1}^m c_j  \tilde{u}_j
\end{equation}
for some $c_j,\tilde{c}_j$ when $ j=1,\ldots,d$.   We point out that the $(c_j)$ are the same coefficients as the decomposition of $u$ in the basis $(u_j)$ of the eigenspace associated to $\lambda$: $u=c_1 u_1 +\dots +c_{d}u_{d}$. 

\begin{remark}\label{remarque:simplification:du:noyau}
We recall that we only need the terms $\tilde u_{j}$: we inject this decomposition of $u'$ in $E^{(1)}$:
\begin{eqnarray*}
E^{(1)}\phi&=&-2\sum_{j=1}^d \tilde{c}_{j} \Big[\int_{\partial B_R}  V_n\partial_n 	u_{j} \partial_n \phi~d\sigma+2\cfrac{R+\beta(d-3)}{R}\displaystyle \int_{\partial B_R}   V_n \nablat u_{j}.\nablat \phi~d\sigma\Big]\\ 
&& -2\sum_{j=1}^m c_{j} \Big[ \int_{\partial B_R}  V_n\partial_n 	\tilde u_{j} \partial_n \phi~d\sigma+2\cfrac{R+\beta(d-3)}{R}\displaystyle \int_{\partial B_R}   V_n \nablat \tilde u_{j}.\nablat \phi~d\sigma\\
& & ~~ -2\lambda  \int_{\partial B_R}  V_n  H u_{j}  \phi~d\sigma ~-2\lambda  \int_{\partial B_R}  V_n  H \tilde u_{j}  \phi~d\sigma \Big]. 
\end{eqnarray*}
By construction the first sum cancels and we simply get
\begin{equation*}
E^{(1)}_{jk}=2\displaystyle \int_{\partial\Omega} V_n\Big(-\partial_n \tilde u_j \partial_n u_k -H \lambda \tilde u_ju_k+ (I+\beta\left(HI_{d}-2D^2 b)\right)\nablat \tilde u_j.\nablat u_k \Big)~d\sigma
\end{equation*}
\end{remark}

\subsection{Explicit resolution of \eqref{tildefi} to compute $ \tilde{u}_j$} 

Let us now compute $ \tilde{u}_j$ solution of \eqref{tildefi}. This step consists in technical computations. For the completeness of the presentation, we present the case of dimension three, we will then simply state the results in dimension two. From now on, we do not consider the case $d\geq4$ for technical reasons. 
\subsubsection{Explicit representation of $ \tilde{u}_j$ in the case $d=2$.} We illustrate the computation of the elements $\tilde  u_i,~i=1,2$ in the case $d=2$. 
The eigenfunctions are the normalized coordinates functions that is $(u_{1},u_{2})$ given by
$$u_1(r,\theta)=r\cfrac{\cos{\theta}}{\sqrt{\pi R^3}}  \textrm{ and }
u_2(r,\theta)=r\cfrac{\sin{\theta}}{\sqrt{\pi R^3}}.$$
We have
\begin{lemma}\label{tildeu_cas2}
Let  $V$ be a deformation of normal component  
$
V_n=R^k(v_1^{(k)}\cos{k \theta}+v_2^{(k)}\sin{k\theta}),
$ then
\begin{eqnarray}
\tilde{u}_1(r,\theta)&=&\cfrac{r^{k+1}}{2\sqrt{\pi}R^{\frac{7}{2}}} \cfrac{1-k}{k}\Big[v_1^{(k)}\cos{(k+1)\theta}+v_2^{k}\sin{(k+1)\theta}\Big]\\
& &+   \cfrac{r^{k-1}}{2\sqrt{\pi}R^{\frac{3}{2}}}    \cfrac{1+k}{k-2}  \Big[\displaystyle\frac{\beta(2-k)+R}{ k \beta+R }\Big]\Big[ v_1^{(k)}\cos{(k-1)\theta}+v_2^{k}\sin{(k-1)\theta} \Big] \nonumber
\end{eqnarray}
and 
\begin{eqnarray}
\tilde{u}_2(r,\theta)&=&\cfrac{r^{k+1}}{2\sqrt{\pi}R^{\frac{7}{2}}} \cfrac{1-k}{k}\Big[- v_2^{(k)}\cos{(k+1)\theta}+v_1^{k}\sin{(k+1)\theta} \Big]\\
& &+ \cfrac{r^{k-1}}{2\sqrt{\pi}R^{\frac{3}{2}}}    \cfrac{1+k}{k-2}  \Big[\displaystyle\frac{\beta(2-k)+R}{ k \beta+R }\Big]\Big[ v_2^{(k)}\cos{(k-1)\theta}-v_1^{k}\sin{(k-1)\theta} \Big] \nonumber
 \end{eqnarray}
 \end{lemma}
In order to justify these formulae, one has to compute $a,b,c,d$ the coefficients 
$$
 \tilde{u}_j= a^{(k)} \cos{(k+1)\theta}+b^{(k)} \sin{(k+1)\theta}+c^{(k)} \cos{(k-1)\theta}+d^{(k)} \cos{(k-1)\theta}
$$
such that $ \tilde{u}_j$ satisfies (\ref{tildefi}) with $ u_i=\displaystyle\frac{x_i}{\parallel x_i\parallel_{L^2(\partial B_R)}}.$ We left the tedious computations to the reader. 
\subsubsection{Explicit representation of $ \tilde{u}_j$ in the case $d=3$} 
We begin with the case where $V_{n}=r^l Y_{l}^m$ and $\varphi_{p}=rY_{1}^p$ where $-l\leq m\leq l$ and $-1\leq p \leq 1$. We introduce the coefficients:
$$C^{(l,1,m,p)}_{l-1,p}=(-1)^{m+p} \ \sqrt{\cfrac{3(2l-1)(2l+1)}{4\pi}}
	\begin{pmatrix} l &1 &l-1\\ m& p &-m-p \end{pmatrix} \begin{pmatrix} l &1 &l-1\\ 0& 0 &0 \end{pmatrix} ,$$
and
$$C^{(l,1,m,p)}_{l+1,p}=(-1)^{m+p} \ \sqrt{\cfrac{3(2l+1)(2l+3)}{4\pi}}
	\begin{pmatrix} l &1 &l+1\\ m& p &-m-p \end{pmatrix} \begin{pmatrix} l &1 &l+1\\ 0& 0 &0 \end{pmatrix}, $$
where we use the Wigner $3j$ symbol and Clebsch-Gordan coefficients. We set  $\alpha = \beta/R$ in order to obtain an adimensional constant.	

\begin{lemma}	\label{calcul3d}
Let $l\ne 0$ be a natural integer and let $-l\leq m\leq l$. Let $V_{n}=r^l Y_{l}^m$ and $u_{p}=rY_{1}^p$ where  $-1\leq p \leq 1$. The unique solution of \eqref{tildefi} that is orthogonal to $\textrm{Span}(Y_{1}^{-1},Y_{1}^0,Y_{1}^1)$ is given by
$$\tilde{u}_{p}=a^{(l,1,m,p)}_{l-1,p,\alpha} r^{l-1} Y^{m+p}_{l-1} + a^{(l,1,m,p)}_{l+1,p,\alpha} \cfrac{r^{l+1}}{R^2} Y^{m+p}_{l+1}$$
where
$$a^{(l,1,m,p)}_{l-1,p,\alpha} = \frac{l+2}{l-2} \ \frac{1+\alpha(3-l)}{1+\alpha(1+l)} C^{(l,1,m,p)}_{l-1,p} \textrm{ and }
a^{(l,1,m,p)}_{l+1,p,\alpha} = \frac{l-1}{l} \ \frac{1+\alpha(4+l)}{1+\alpha(3+l)} C^{(l,1,m,p)}_{l+1,p}.$$
\end{lemma}

\begin{proofof}{Lemma \ref{calcul3d}} We first decompose the right hand side of \eqref{tildefi} into the basis of spherical harmonics. Taking into account that 
$$\begin{pmatrix} l_{1}&l_{2}&L\\ 0&0&0 \end{pmatrix}=0$$
whenever $(l_{1},l_{2},L)$ satisfies the triangular inequality and $l_{1}+l_{2}+L$ is odd, we get
$$ \beta V_{n} \partial_{n} u_{p}=\beta R^l Y_{l}^m Y_{1}^p= \beta R^l \left[ C^{(l,1,m,p)}_{l-1,p} Y_{l-1}^{m+p}+ C^{(l,1,m,p)}_{l+1,p} Y_{l+1}^{m+p}\right]$$
and then
$$\beta \Deltat(V_{n} \partial_{n}u_{p})=\alpha R^{l-1} \left[ l(1-l) C^{(l,1,m,p)}_{l-1,p} Y^{m+p}_{l-1} - (l+1)(l+2) C^{(l,1,m,p)}_{l+1,p} Y^{m+p}_{l+1}\right].$$
We also have
\begin{eqnarray*}
\nablat V_{n} . \nablat u_{p} & =& \frac{1}{2} \left[  \Deltat(V_{n} u_{p})- V_{n} \Deltat u_{p} -u_{p} \Deltat V_{n}\right]\\
&=& \frac{R^{l-1}}{2} \Big[~~ l \Be(1- l)\Bk  C^{(l,1,m,p)}_{l-1,p} Y^{m+p}_{l-1} ~~- (l+1)(l+2) C^{(l,1,m,p)}_{l+1,p} Y^{m+p}_{l+1} \\
& &~~~~~~~+~~2~~~~~  C^{(l,1,m,p)}_{l-1,p} Y^{m+p}_{l-1}~ +~~2~~~~ C^{(l,1,m,p)}_{l+1,p} Y^{m+p}_{l+1}\\
& &~~~~~~~+  l(l+1)  C^{(l,1,m,p)}_{l-1,p} Y^{m+p}_{l-1} + l (l+1) ~~C^{(l,1,m,p)}_{l+1,p} Y^{m+p}_{l+1}	\Big]\\
&=& R^{l-1} \left[  (l+1) C^{(l,1,m,p)}_{l-1,p} Y^{m+p}_{l-1} -l\ C^{(l,1,m,p)}_{l+1,p} Y^{m+p}_{l+1} \right]. 
\end{eqnarray*}
Since $\divet{V_{n}\nablat u_{p}}=\nablat V_{n}.\nablat u_{p}+V_{n}\Deltat  u_{p}$, it comes
$$\divet{V_{n}\nablat u_{p}}= R^{l-1} \left[  (l-1) C^{(l,1,m,p)}_{l-1,p} Y^{m+p}_{l-1} -(l+2)\ C^{(l,1,m,p)}_{l+1,p} Y^{m+p}_{l+1} \right]. $$
Hence, gathering the various terms in the right hand side of \eqref{tildefi}, we see that $\tilde{u}_{p}$ is solution of 
\begin{align*}
-\beta\Deltat \tilde{u}_{p}&+\partial_{n} \tilde{u}_{p} -\lambda_{2}\tilde{u}_{p} =\\R^{l-1}&\left[(l+2)(1+\alpha(3-l))C^{(l,m,1,p)}_{l-1,p}Y_{l-1}^{m+p}+(1-l)(1+\alpha(4+l))C^{(l,m,1,p)}_{l+1,p}Y_{l+1}^{m+p}\right].
\end{align*}
After identification, we obtain:
$$\tilde{u}_{p} = a_{l-1,p,\alpha}^{(l,1,m,p)} r^{l-1} Y^{m+p}_{l-1} + a_{l+1,p,\alpha}^{(l,1,m,p)} \cfrac{r^{l+1}}{R^2} Y^{m+p}_{l+1},$$
where the coefficients $ a_{l\pm1,p,\alpha}^{(l,1,m,p)}$ are defined in Lemma \ref{calcul3d}.
\end{proofof}
As a corollary, we deduce the general case for $V_{n}$.

\begin{corollary}\label{calcul:tilde:phi}
If 
$$V_{n}=\sum_{l=2}^\infty r^l \sum_{m=-l}^l v_{l,m}Y_{l}^m \textrm{ and }u_{p}=\sum_{p=-1}^1 \alpha_{p} Y^p_{1},$$ then
$$\tilde{u}_{p} = \sum_{l=2}^\infty \sum_{m=-l}^l \sum_{p=-1}^1 \alpha_{p} v_{l,m} \left[a_{l-1,p,\alpha}^{(l,1,m,p)} r^{l-1} Y^{m+p}_{l-1} + a_{l+1,p,\alpha}^{(l,1,m,p)} \cfrac{r^{l+1}}{R^2} Y^{m+p}_{l+1}\right].$$
\end{corollary}

\subsection{The explicit expression of the trace of $E^{(1)}$}\label{app:d3}
We leave the tedious but easy computations of the case $d=2$ to the reader; the obtained result is written in \eqref{eq:trE2d}. We focus here on the much more technical case $d=3$.\\
We set $u_{j}=\Be{K}(R)\Bk (\alpha_{-1}^{i}Y_{1}^{-1}+\alpha_{0}^{i}Y_{1}^{0}+ \alpha_{1}^{i}Y_{1}^{1})$ for $1\leq j \leq 3$ where 
\begin{alignat*}{5}
\alpha_{-1}^{1}=1/\sqrt{2},  &  ~~~~ &\alpha_{0}^{1}=0,&~~~~ &\alpha_{1}^1=1/\sqrt{2},\\
\alpha_{-1}^{2}=0 , &&\alpha_{0}^{2}=1,& &\alpha_{1}^2=0,\\  
\alpha_{-1}^{3}=-i/\sqrt{2}, & &\alpha_{0}^{3}=0,& &\alpha_{1}^3=i/\sqrt{2}.
\end{alignat*}
On the sphere in dimension $3$, the deviatoric part of the curvature cancels and the entries of $E^{(1)}$ are 
\begin{align*}
\Tr(E^{(1)})=\sum_{j=1}^3 E^{(1)}_{jj} \textrm{ where } E^{(1)}_{jj}=\int_{\partial\Omega} V_n\Big(-\partial_n \tilde u_j \partial_n  u_j -H \lambda \tilde u_j u_j+ \nablat \tilde u_j.\nablat  u_j \Big)~d\sigma,
\end{align*}
where each $\tilde u_{j}$ corresponding to $u_j$ is computed thanks to Corollary \ref{calcul:tilde:phi}.
\par
\noindent
We first state a technical result to perform this summation. We postpone its proof to the end of the section.  

\begin{lemma} \label{technique:calcul:trace:E1}
Let $V_{n}=R^l Y_{l}^m$,  $-l\leq m\leq l$ and 
$$\psi=r Y^p_{1}$$ 
for $-1\leq p \leq1$. Let $m'$ and $p'$ be integer such that $-l\leq m'\leq l$ and $-1\leq p' \leq1$ and suppose 
$$\tilde\psi = a \ r^{l-1}Y_{l-1}^{m'+p'} +b \ \cfrac{r^{l+1}}{R^2} Y^{m'+p'}_{l+1}.$$ 
Then 
\begin{align*}
\int_{\partial B_{R}} &V_n\Big(-\partial_n \tilde\psi \partial_n \psi -H \lambda \tilde\psi \psi + \nablat \tilde\psi .\nablat \psi \Big)~d\sigma\\
&= \ -a\ (4\alpha+2l) \ R^{2l-1} \ \int_{\partial B_{1}}Y_{l-1}^{m'+p'} Y_{l}^{m} Y_{1}^{p}\ -\ b\ (4\alpha+2)\ R^{2l-1}\ \int_{\partial B_{1}}Y_{l+1}^{m'+p'} Y_{l}^{m} Y_{1}^{p}.
\end{align*}
\end{lemma}
As a consequence, we get for $j=1,2,3$
\begin{align*}
E^{(1)}_{jj}=\  -\  \Be K\Bk (R)\ R^{2l\Be + \Bk1} &\left[ (4\alpha+2l)\ \frac{l+2}{l-2} \ \frac{1+\alpha(3+l)}{1+\alpha(1+l)}\ \sum_{m=-l}^l \sum_{p=-1}^1 
|\alpha_{p}^j|^2 \ |v_{l,m}|^2 \left(\int_{\partial B_{1}} \overline{Y_{l-1}^{m+p}} Y_{l}^m Y_{1}^p \right)^2\right. \\
&+\left.  \ (4\alpha+2) \ \frac{l-1}{l}\ \frac{1+\alpha(4+l)}{1+\alpha(3+l)} \  \sum_{m=-l}^l \sum_{p=-1}^1 
|\alpha_{p}^j|^2 \ |v_{l,m}|^2 \left(\int_{\partial B_{1}} \overline{Y_{l+1}^{m+p}} Y_{l}^m Y_{1}^p \right)^2\right].
\end{align*} 
We are now in position to prove Proposition \ref{trace:E1:dimension:3} concerning the trace of $E^{(1)}$ in dimension $d=3$.

\begin{proofof}{Proposition \ref{trace:E1:dimension:3}} We have to sum the $E^{(1)}_{jj}$ obtained before the statement of Proposition \ref{trace:E1:dimension:3}. By the normalization condition $\sum_{j}|\alpha_{p}^{j}|^2=1$, our main task is to compute the sum over $p=-1,0,1$ of the integrals involving three spherical harmonics. The values of this type of integral is recalled in Propositions \ref{inetrage:produit:trois:harmoniques} and \ref{inetrage:produit:trois:harmoniques:bis}.  Elementary computations then give
$$	\sum_{m=-l}^l \sum_{p=-1}^1 \left(\int_{\partial B_{1}} \overline{Y_{l-1}^{m+p}} Y_{l}^m Y_{1}^p\right)^2 = \frac{3}{4\pi} \ \frac{l}{2l+1}\textrm{ and } \sum_{m=-l}^l \sum_{p=-1}^1 \left( \int_{\partial B_{1}} \overline{Y_{l+1}^{m+p}} Y_{l}^m Y_{1}^p\right)^2=\ \frac{3}{4\pi} \ \frac{l+1}{2l+1}.$$
\end{proofof}

\begin{proofof}{Lemma \ref{technique:calcul:trace:E1}}
We compute:
\begin{eqnarray*}
-V_{n} \partial_{n} \tilde\psi \partial_{n}\psi &= & -R^{2l-1} \left[ a (l-1)Y_{l-1}^{m'+p'} +b (l+1) Y_{l+1}^{m'+p'} \right] Y_{l}^{m} Y_{1}^{p}, \\
-\lambda H V_{n} \tilde\psi \psi &= &  -R^{2l-1} (4\alpha+2) \left[ a Y_{l-1}^{m'+p'} +b Y_{l+1}^{m'+p'} \right] Y_{l}^{m} Y_{1}^{p},
\end{eqnarray*}
We have also 
\begin{eqnarray*}
\int_{\partial B_{R}} V_{n} \nablat \tilde\psi.\nablat \psi &=& \frac{1}{2} \int_{\partial B_{R}} V_{n} \left[ \Deltat (\tilde\psi \psi)-\psi\Deltat \tilde\psi-\tilde\psi \Deltat \psi\right]\\
&=& -\ \frac{1}{2}\ l(l+1)\  R^{2l-1} \ \int_{\partial B_{1}} ( a Y_{l-1}^{m'+p'} +b Y_{l+1}^{m'+p'})  Y_{l}^{m} Y_{1}^{p}\\
&&+ R^{2l+1} \ \int_{\partial B_{1}} ( a Y_{l-1}^{m'+p'} +b Y_{l+1}^{m'+p'})  Y_{l}^{m} Y_{1}^{p}\\
&&+ \frac{1}{2} \ R^{2l-1} \ \int_{\partial B_{1}} \left[ a \ l(l-1) \ Y_{l-1}^{m'+p'} +b\  (l+1)(l+2)\  Y_{l+1}^{m'+p'} \right]  Y_{l}^{m} Y_{1}^{p}\\
&=& R^{2l-1} \  \int_{\partial B_{1}} \left[ a \ (l-1) \ Y_{l-1}^{m'+p'} +b\ (l+2)\  Y_{l+1}^{m'+p'} \right]  Y_{l}^{m} Y_{1}^{p}.
\end{eqnarray*}
We obtain the result by summing the three terms.
\end{proofof}

\section{Shape Derivatives of Steklov and Laplace-Beltrami eigenvalues problem}\label{app:SteLB}

The following result is obtained by taking $\beta=0$ in Theorem \ref{Theoreme:gradient:cas}.
\begin{theorem}\label{Theoreme:gradient:casSteklov}[Steklov eigenvalues]
We distinguish the case of simple and multiple eigenvalue. 
\begin{itemize}
\item If $\lambda{=\lambda_{k}(\Omega)}$ is a simple eigenvalue of the Steklov problem and $u$ an associated eigenfunction, then the  application $t \rightarrow \lambda(t)=\lambda_{k}((I+t\V)(\Om))$ is differentiable and the derivative at $t=0$ is 
$$
\lambda'(0)=\int_{\partial\Omega}V_n\Big( \vert  \nablat u\vert ^2-\vert \partial_{n}{}u \vert ^2 -\lambda H \vert  u\vert ^2\Big)~d\sigma.
$$
The shape derivative $u'$of the eigenfunction satisfies
\begin{align*}
\Delta u' = 0 &\textrm{ in }\Omega, \nonumber\\
\partial_n{u'} -& \lambda u' =  \divet(V_n\nablat u)  -\lambda'(0) u+\lambda  V_n( \partial_nu+ H u) \textrm{ on }\partial\Omega.
\end{align*}
\item Let $\lambda$ be a multiple eigenvalue of order $m\ge 2$.  Let   $(u_j)$ for $1 \leq j\leq m$ denote the eigenfunctions associated to $\lambda$. Then there exists $m$ functions $t\mapsto \lambda_k(t), k=1,\ldots,m$ defined in a neighborhood of 0 such that
\begin{itemize}
\item $\lambda_{k}(0)=\lambda$,
\item for every $t$ in a neighborhood of 0, $\lambda_{k}(t)$ is an Steklov eigenvalue of $\Om_{t}=(I+t\V)(\Om)$,
\item the functions $t\mapsto\lambda_{k}(t), k=1,\ldots,m$ admit derivatives which are the eigenvalues of the $m\times m$ matrix $M{=M_{\Om}(V_{n}})$ of entries $(M_{ij})$ defined by
$$
M_{jk}=\displaystyle \int_{\partial\Omega} V_n\Big(-\partial_n{}u_j \partial_n{}u_k -H \lambda{}u_ju_k+ \nablat u_j.\nablat u_k \Big)~d\sigma.
$$
\end{itemize}
\end{itemize}\end{theorem}

\noindent The following result is obtain by taking $\beta \rightarrow +\infty$ in Theorem \ref{Theoreme:gradient:cas}.\begin{theorem}\label{Theoreme:gradient:casLaplace-Beltrami}[Laplace-Beltrami eigenvalues]
We distinguish the case of simple and multiple eigenvalue. 
\begin{itemize}
\item If $\lambda{=\lambda_{k}(\Omega)}$ is a simple eigenvalue of the Laplace-Beltrami problem and $u$ an associated eigenfunction, then the  application $t \rightarrow \lambda(t)=\lambda_{k}((I+t\V)(\Om))$ is differentiable and the derivative at $t=0$ is 
$$
\lambda'(0)=\int_{\partial\Omega}V_n\Big((H~I_{d}-2D^2b)\nablat u.\nablat u\Big)~d\sigma.
$$
The shape derivative $v'$of the eigenfunction satisfies
\begin{align*}
\Delta u' = &\;0 \textrm{ in }\Omega, \nonumber\\
-\Deltat u' =& \;\Deltat (V_n \partial_n u) -\divet\big(V_n(2D^2b-H I_{d})\nablat u\big) -\lambda'(0) u  \textrm{ on }\partial\Omega.
\end{align*}
\item Let $\lambda$ be a multiple eigenvalue of order $m\ge 2$.  Let   $(u_j)$ for $1\leq j \leq m$ denote the eigenfunctions associated to $\lambda$. Then there exists $m$ functions $t\mapsto \lambda_k(t), k=1,\ldots,m$ defined in a neighborhood of 0 such that
\begin{itemize}
\item $\lambda_{k}(0)=\lambda$,
\item for every $t$ in a neighborhood of 0, $\lambda_{k}(t)$ is a Laplace-Beltrami eigenvalue of $\Om_{t}=(I+t\V)(\Om)$,
\item the functions $t\mapsto\lambda_{k}(t), k=1,\ldots,m$ admit derivatives which are the eigenvalues of the $m\times m$ matrix $M{=M_{\Om}(V_{n})}$ of entries $(M_{ij})$ defined by
$$
M_{jk}=\displaystyle \int_{\partial\Omega} V_n\Big(\left(HI_{d}-2D^2 b\right)\nablat u_i.\nablat u_j  \Big)~d\sigma.
$$
\end{itemize}
\end{itemize}\end{theorem}
\Bk

\textit{Acknowledgements}. Part of the work was supported by the project Projet ANR-12-BS01-0007 OPTIFORM financed by the French Agence Nationale de la Recherche (ANR). We would also like to thank the
anonymous reviewer for its careful reading of the previous version of the manuscript, that helped to improve and clarify the paper.

\bibliographystyle{plain}
\bibliography{references}

\begin{thebibliography}{10}

\bibitem{BendaliLemrabet}
A.~Bendali and K.~Lemrabet.
\newblock The effect of a thin coating on the scattering of a time-harmonic
  wave for the {H}elmholtz equation.
\newblock {\em SIAM J. Appl. Math.}, 56(6):1664--1693, 1996.

\bibitem{BBMP}
M.~F. Betta, F.~Brock, A.~Mercaldo, and M.~R. Posteraro.
\newblock A weighted isoperimetric inequality and applications to
  symmetrization.
\newblock {\em J. Inequal. Appl.}, 4(3):215--240, 1999.

\bibitem{Bleecker}
D.~D. Bleecker.
\newblock The spectrum of a {R}iemannian manifold with a unit {K}illing vector
  field.
\newblock {\em Trans. Amer. Math. Soc.}, 275(1):409--416, 1983.

\bibitem{BDHV10}
V.~Bonnaillie-No{\"e}l, M.~Dambrine, F.~H{\'e}rau, and G.~Vial.
\newblock On generalized {V}entcel's type boundary conditions for {L}aplace
  operator in a bounded domain.
\newblock {\em SIAM J. Math. Anal.}, 42(2):931--945, 2010.

\bibitem{BrascoDePhilippisRuffini}
L.~Brasco, G.~De~Philippis, and B.~Ruffini.
\newblock Spectral optimization for the {S}tekloff-{L}aplacian: the stability
  issue.
\newblock {\em J. Funct. Anal.}, 262(11):4675--4710, 2012.

\bibitem{Brock}
F.~Brock.
\newblock An isoperimetric inequality for eigenvalues of the {S}tekloff
  problem.
\newblock {\em ZAMM Z. Angew. Math. Mech.}, 81(1):69--71, 2001.

\bibitem{CaubetDambrineKateb}
F.~Caubet, M.~Dambrine, and D.~Kateb.
\newblock Shape optimization methods for the inverse obstacle problem with
  generalized impedance boundary conditions.
\newblock {\em Inverse Problems}, 29(11), 2013.

\bibitem{Clarcke}
F.~H. Clarke.
\newblock {\em Optimization and nonsmooth analysis}.
\newblock Canadian Mathematical Society Series of Monographs and Advanced
  Texts. John Wiley \& Sons, Inc., New York, 1983.
\newblock A Wiley-Interscience Publication.

\bibitem{CohenTannoudji}
C.~Cohen~Tannoudji, B.~Diu, and F.~Laloe.
\newblock {\em M\'ecanique Quantique}.
\newblock Hermann, Paris, 1997.

\bibitem{ColboisDodziuk}
B.~Colbois and J.~Dodziuk.
\newblock Riemannian metrics with large {$\lambda_1$}.
\newblock {\em Proc. Amer. Math. Soc.}, 122(3):905--906, 1994.

\bibitem{ColboisDrydenElSsoufi}
B.~Colbois, E.~B. Dryden, and A.~El~Soufi.
\newblock Bounding the eigenvalues of the {L}aplace-{B}eltrami operator on
  compact submanifolds.
\newblock {\em Bull. Lond. Math. Soc.}, 42(1):96--108, 2010.

\bibitem{DK2012}
M.~Dambrine and D.~Kateb.
\newblock Persistency of wellposedness of {V}entcel's boundary value problem
  under shape deformations.
\newblock {\em J. Math. Anal. Appl.}, 394(1):129--138, 2012.

\bibitem{DelfourZolesio}
M.~C. Delfour and J.-P. Zol{\'e}sio.
\newblock {\em Shapes and geometries}, volume~4 of {\em Advances in Design and
  Control}.
\newblock Society for Industrial and Applied Mathematics (SIAM), Philadelphia,
  PA, 2001.
\newblock Analysis, differential calculus, and optimization.

\bibitem{DesaintZolesio}
F.~R. Desaint and J.-P. Zol{\'e}sio.
\newblock Manifold derivative in the {L}aplace-{B}eltrami equation.
\newblock {\em J. Funct. Anal.}, 151(1):234--269, 1997.

\bibitem{Goldstein}
G.~R. Goldstein.
\newblock Derivation and physical interpretation of general boundary
  conditions.
\newblock {\em Adv. Differential Equations}, 11(4):457--480, 2006.

\bibitem{HadJol05}
H.~Haddar, P.~Joly, and H.-M. Nguyen.
\newblock Generalized impedance boundary conditions for scattering by strongly
  absorbing obstacles: the scalar case.
\newblock {\em Math. Models Methods Appl. Sci.}, 15(8):1273--1300, 2005.

\bibitem{HenrotPierre}
A.~Henrot and M.~Pierre.
\newblock {\em Variation et optimisation de formes}, volume~48 of {\em
  Math\'ematiques \& Applications (Berlin) [Mathematics \& Applications]}.
\newblock Springer, Berlin, 2005.
\newblock Une analyse g{\'e}om{\'e}trique. [A geometric analysis].

\bibitem{Hersch1}
J.~Hersch.
\newblock Caract\'erisation variationnelle d'une somme de valeurs propres
  cons\'ecutives; g\'en\'eralisation d'in\'egalit\'es de {P}\'olya-{S}chiffer
  et de {W}eyl.
\newblock {\em C. R. Acad. Sci. Paris}, 252:1714--1716, 1961.

\bibitem{Hersch2}
J.~Hersch.
\newblock Quatre propri\'et\'es isop\'erim\'etriques de membranes sph\'eriques
  homog\`enes.
\newblock {\em C. R. Acad. Sci. Paris S\'er. A-B}, 270:A1645--A1648, 1970.

\bibitem{HileXu}
G.~N. Hile and Z.~Y. Xu.
\newblock Inequalities for sums of reciprocals of eigenvalues.
\newblock {\em J. Math. Anal. Appl.}, 180(2):412--430, 1993.

\bibitem{Kennedy}
J.~Kennedy.
\newblock A {F}aber-{K}rahn inequality for the {L}aplacian with generalised
  {W}entzell boundary conditions.
\newblock {\em J. Evol. Equ.}, 8(3):557--582, 2008.

\bibitem{LemrabetTeniou}
K.~Lemrabet and D.~Teniou.
\newblock Vibrations d'une plaque mince avec raidisseur sur le bord.
\newblock {\em Maghreb Math. Rev.}, 2(1):27--41, 1992.

\bibitem{Nedelec}
J.-C. N{\'e}d{\'e}lec.
\newblock {\em Acoustic and electromagnetic equations}, volume 144 of {\em
  Applied Mathematical Sciences}.
\newblock Springer-Verlag, New York, 2001.
\newblock Integral representations for harmonic problems.

\bibitem{ZuazuaOrtega}
J.~H. Ortega and E.~Zuazua.
\newblock Generic simplicity of the eigenvalues of the {S}tokes system in two
  space dimensions.
\newblock {\em Adv. Differential Equations}, 6(8):987--1023, 2001.

\bibitem{SteinWeiss}
E.~M. Stein and G.~Weiss.
\newblock {\em Introduction to {F}ourier analysis on {E}uclidean spaces}.
\newblock Princeton University Press, Princeton, N.J., 1971.
\newblock Princeton Mathematical Series, No. 32.

\end{thebibliography}
\bigskip

\noindent------------------------------------------------------------------

\noindent Marc Dambrine

\smallskip

\noindent Universit\'{e} de Pau et des Pays de l'Adour

\smallskip

\noindent E-mail: \texttt{marc.dambrine@univ-pau.fr}
\newline\bigskip
\texttt{http://web.univ-pau.fr/~mdambrin/Marc{\_}Dambrine/Home.html}

\noindent Djalil Kateb

\smallskip

\noindent Universit\'e de Technologie de Compi\`egne.

\smallskip

\noindent E-mail: \texttt{djalil.kateb@utc.fr} \newline\noindent
\texttt{http://www.lmac.utc.fr/membres/kateb}

\bigskip

\noindent Jimmy Lamboley

\smallskip

\noindent Universit\'{e} Paris-Dauphine

\smallskip

\noindent E-mail: \texttt{lamboley@math.cnrs.fr}

\noindent\texttt{https://www.ceremade.dauphine.fr/\symbol{126}lamboley/}

\end{document}